\documentclass[10pt]{article}
\pdfoutput=1
\usepackage{xcolor}
\usepackage{amsmath}
\usepackage{amssymb}
\usepackage{amsthm}
\usepackage{dsfont}
\usepackage{fullpage}
\usepackage{graphics}
\usepackage{graphicx}
\usepackage{subcaption}
\usepackage{caption}
\usepackage[font=footnotesize,labelfont=bf]{caption}
\usepackage{mathtools}
\usepackage{mathrsfs}
\usepackage{bbm}
\usepackage{multicol}
\usepackage[shortlabels]{enumitem}
\usepackage{pgfplots}
\pgfplotsset{compat=newest}
\usepackage{bm}
\usepackage{tikz}
\usepackage{comment}
\usetikzlibrary{patterns}
\usetikzlibrary{decorations.pathmorphing}
\usetikzlibrary{decorations.pathreplacing,decorations.markings}
\usetikzlibrary{shapes}
\makeatletter
\def\grd@save@target#1{%
	\def\grd@target{#1}}
\def\grd@save@start#1{%
	\def\grd@start{#1}}
\tikzset{
	grid with coordinates/.style={
		to path={%
			\pgfextra{%
				\edef\grd@@target{(\tikztotarget)}%
				\tikz@scan@one@point\grd@save@target\grd@@target\relax
				\edef\grd@@start{(\tikztostart)}%
				\tikz@scan@one@point\grd@save@start\grd@@start\relax
				\draw[minor help lines] (\tikztostart) grid (\tikztotarget);
				\draw[major help lines] (\tikztostart) grid (\tikztotarget);
				\grd@start
				\pgfmathsetmacro{\grd@xa}{\the\pgf@x/1cm}
				\pgfmathsetmacro{\grd@ya}{\the\pgf@y/1cm}
				\grd@target
				\pgfmathsetmacro{\grd@xb}{\the\pgf@x/1cm}
				\pgfmathsetmacro{\grd@yb}{\the\pgf@y/1cm}
				\pgfmathsetmacro{\grd@xc}{\grd@xa + \pgfkeysvalueof{/tikz/grid with coordinates/major step}}
				\pgfmathsetmacro{\grd@yc}{\grd@ya + \pgfkeysvalueof{/tikz/grid with coordinates/major step}}
				\foreach \x in {\grd@xa,\grd@xc,...,\grd@xb}
				\node[anchor=north] at (\x,\grd@ya) {\pgfmathprintnumber{\x}};
				\foreach \y in {\grd@ya,\grd@yc,...,\grd@yb}
				\node[anchor=east] at (\grd@xa,\y) {\pgfmathprintnumber{\y}};
			}
		}
	},
	minor help lines/.style={
		help lines,
		step=\pgfkeysvalueof{/tikz/grid with coordinates/minor step}
	},
	major help lines/.style={
		help lines,
		line width=\pgfkeysvalueof{/tikz/grid with coordinates/major line width},
		step=\pgfkeysvalueof{/tikz/grid with coordinates/major step}
	},
	grid with coordinates/.cd,
	minor step/.initial=.2,
	major step/.initial=1,
	major line width/.initial=0.25mm,
}

\pgfdeclarepatternformonly{north east lines wide}%
{\pgfqpoint{-1pt}{-1pt}}%
{\pgfqpoint{10pt}{10pt}}%
{\pgfqpoint{5pt}{5pt}}%
{
	\pgfsetlinewidth{.8pt}
	\pgfpathmoveto{\pgfqpoint{0pt}{0pt}}
	\pgfpathlineto{\pgfqpoint{9.1pt}{9.1pt}}
	\pgfusepath{stroke}
}

\tikzset{
	on each segment/.style={
		decorate,
		decoration={
			show path construction,
			moveto code={},
			lineto code={
				\path [#1]
				(\tikzinputsegmentfirst) -- (\tikzinputsegmentlast);
			},
			curveto code={
				\path [#1] (\tikzinputsegmentfirst)
				.. controls
				(\tikzinputsegmentsupporta) and (\tikzinputsegmentsupportb)
				..
				(\tikzinputsegmentlast);
			},
			closepath code={
				\path [#1]
				(\tikzinputsegmentfirst) -- (\tikzinputsegmentlast);
			},
		},
	},
	mid arrow/.style={postaction={decorate,decoration={
				markings,
				mark=at position .5 with {\arrow[#1]{stealth}}
	}}},
	rmid arrow/.style={postaction={decorate,decoration={
				markings,
				mark=at position .5 with {\arrowreversed[#1]{stealth}}
	}}},
	end arrow/.style={postaction={decorate,decoration={
				markings,
				mark=at position 1 with {\arrow[#1]{stealth}}
	}}},
	start arrow/.style={postaction={decorate,decoration={
				markings,
				mark=at position 0 with {\arrow[#1]{stealth}}
	}}},
	mid3 arrow/.style={postaction={decorate,decoration={
				markings,
				mark=at position .3 with {\arrow[#1]{stealth}}
	}}},
	rmid3 arrow/.style={postaction={decorate,decoration={
				markings,
				mark=at position .7 with {\arrowreversed[#1]{stealth}}
	}}},
	mid4 arrow/.style={postaction={decorate,decoration={
				markings,
				mark=at position .4 with {\arrow[#1]{stealth}}
	}}},
	rmid4 arrow/.style={postaction={decorate,decoration={
				markings,
				mark=at position .4 with {\arrowreversed[#1]{stealth}}
	}}},
}

\makeatother
\tikzset{every state/.style={minimum size=0pt}}

\tikzset{
	mark position/.style args={#1(#2)}{
		postaction={
			decorate,
			decoration={
				markings,
				mark=at position #1 with \coordinate (#2);
			}
		}
	}
}

\usetikzlibrary{3d}
\usepackage{xcolor}
\usepackage{xifthen}

\usetikzlibrary{decorations}

\pgfdeclaredecoration{ignore}{final}
{
	\state{final}{}
}

\pgfdeclaremetadecoration{middle}{initial}{
	\state{initial}[
	width={0pt},
	next state=middle
	]
	{\decoration{moveto}}
	
	\state{middle}[
	width={\pgfdecorationsegmentlength*\pgfmetadecoratedpathlength},
	next state=final
	]
	{\decoration{curveto}}
	
	\state{final}
	{\decoration{ignore}}
}

\tikzset{middle segment/.style={decoration={middle},decorate, segment length=#1}}

\makeatletter
\renewcommand\paragraph{\@startsection{paragraph}{4}{\z@}%
	{-2.5ex\@plus -1ex \@minus -.25ex}%
	{1.25ex \@plus .25ex}%
	{\normalfont\normalsize\bfseries}}
\makeatother
\newtheorem{thm}{Theorem}
\newtheorem{lm}[thm]{Lemma}
\newtheorem{defn}[thm]{Definition}
\newtheorem{prop}[thm]{Proposition}
\newtheorem{rmk}[thm]{Remark}

\newcommand{\complexC}{\mathbb{C}}

\newcommand{\dd}{{\mathrm d}}

\newcommand{\ddbarr}[1]{\frac{{\mathrm d}#1}{2\pi {\mathrm i}}}

\newcommand{\hh}{\mathsf{h}}

\newcommand{\ii}{\mathrm{i}}

\newcommand{\inn}{\mathrm{in}}
\newcommand{\intZ}{\mathbb{Z}}
\newcommand{\Kess}{\mathrm{ch}}
\newcommand{\limess}{\chi}

\newcommand{\LL}{\mathrm{L}}

\newcommand{\out}{\mathrm{out}}

\newcommand{\prob}{\mathbb{P}}
\newcommand{\realR}{\mathbb{R}}

\newcommand{\RR}{\mathrm{R}}

\renewcommand{\Re}{\mathrm{Re}}

\renewcommand{\a}{\omega}
\newcommand{\nXX}{\omega}
\newcommand{\nHH}{\theta}

\newcommand{\h}{\mathfrak{h}}

\renewcommand{\d}{\diff}

\newcommand{\st}{\mathsf{t}}
\newcommand*\diff{\mathop{}\!\mathrm{d}}
\newcommand{\TT}{\tau}
\newcommand{\XX}{\alpha}
\newcommand{\HH}{\beta}
\newcommand{\he}{\mathcal{H}}
\newcommand{\rs}{\mathrm{s}}

\newcommand{\DD}{\mathbf{D}}
\newcommand{\uc}{\mathrm{UC}}
\newcommand{\ucc}{\mathrm{UC}_c}
\newcommand{\B}{\mathbf{B}}
\newcommand{\ta}{\bm{\tau}}
\newcommand{\sig}{\bm{\sigma}}
\newcommand{\CL}{\mathbb{C}_{\mathrm{L}}}
\newcommand{\CR}{\mathbb{C}_{\mathrm{R}}}
\newcommand{\xib}{\bm{\xi}}
\newcommand{\etab}{\bm{\eta}}
\newcommand{\mnw}{m\mathrm{NW}}
\newcommand{\1}{\mathbf{1}}

\numberwithin{equation}{section} 
\numberwithin{thm}{section}
	
\author{Yuchen Liao\footnote{School of Mathematical Sciences, University of Science and Technology of China, Hefei, Anhui 230026, P.R.China. Email: \texttt{ycliao@ustc.edu.cn}}\and 
		Zhipeng Liu\footnote{Department of Mathematics, University of Kansas, Lawrence, KS 66045. Email: \texttt{zhipeng@ku.edu}}}
	\title{Multipoint distributions of the KPZ fixed point with compactly supported initial conditions}
	
\begin{document}
		\maketitle
        \begin{abstract}
            The KPZ fixed point is a universal limiting space-time random field for the Kardar-Parisi-Zhang universality class. While the joint law of the KPZ fixed point at a fixed time has been studied extensively, the multipoint distributions of the KPZ fixed point in the general space-time plane are much less well understood. More explicitly, formulas were only available for the narrow wedge initial condition \cite{johansson_rahman_multitime,Liu2022} and  the flat initial condition \cite{Liu2022} for the multipoint distributions, and the half-Brownian and Brownian initial conditions \cite{johansson_rahman_inhomogeneous,rahman25temporal} for the two-point distributions. 
            
            In this paper, we obtain the first formula for the space-time joint distributions of the KPZ fixed point with general initial conditions of compact support. The formula is 
            obtained through taking $1:2:3$ KPZ scaling limit of the multipoint distribution formulas for the totally asymmetric simple exclusion process (TASEP). A key ingredient is a probabilistic representation, inspired by \cite{matetski2021kpz},  of the 
            kernel encoding the initial condition for TASEP, which was first defined through an implicit characterization in \cite{Liu2022}. Moreover, we also verify that the equal time version of our formula matches the path integral formula in \cite{matetski2021kpz} for the KPZ fixed point when the initial condition is of compact support.

        \end{abstract}
    \section{Introduction}

    \subsection{Background}

    The Kardar–Parisi–Zhang (KPZ) universality class \cite{KPZ1986} contains a broad family of random growth models in $(1+1)$-dimensions, including models from directed polymers {\cite{timo,corwin2014tropical}}, interacting particle systems {\cite{johansson2000shape}}, stochastic partial differential equations {\cite{KPZ1986,Hairer13}}, etc. In the past four decades,  the KPZ universality class has become a central subject of study in probability theory, statistical mechanics, and mathematical physics. For a more thorough introduction, we refer to the surveys \cite{Corwin_introkpz, Quastel_introkpz, zygouras_review} and the references therein. 
    
    A hallmark of this class is the universal $1:2:3$ scaling exponent for height fluctuations, spatial correlations and temporal correlations and a conjectural universal scaling limit for all the models in the universality class. More precisely, it is conjectured that the random height functions $H(x,t)$ describing the evolutions of different models will all converge to a universal limiting space-time field $\mathcal{H}(\alpha,\tau)$, under the following scaling:
    \begin{equation}
        \lim_{\varepsilon\to 0} c_3\varepsilon^{\frac{1}{2}}H(c_2\alpha\varepsilon^{-1},c_3\tau\varepsilon^{-\frac{3}{2}}; \h^{\varepsilon}) = \mathcal{H}(\alpha,\tau;\h),
    \end{equation}
    where $\h^\varepsilon$ and $\h$ are the initial conditions for the height functions before and after the limit with $\h^\varepsilon\to \h$ in a proper sense. A central question in this area is to understand $\mathcal{H}(\alpha,\tau;\h)$. 

    The field $\mathcal{H}(\alpha,\tau;\h)$ is known as the KPZ fixed point. It was first constructed in \cite{matetski2021kpz}, and can be described as a $1:2:3$ scaling invariant Markov process on the space of upper semicontinuous functions on $\mathbb{R}$ with explicit formulas for its transition probability. Convergence to the KPZ fixed point has only been shown for a few special models, see \cite{matetski2021kpz, quastel20brownian,Virag20,matetski2022polynuclear, Wu23KPZ,aggarwal24fixed}. An alternative description is through a Hopf-Lax type variational formula \cite{quastel20brownian}, with the driving force given by the directed landscape $\mathcal{L}(y,s;x,t)$. This is another universal limiting object in the KPZ universality class first constructed in \cite{dauvergne2018directed}. Convergence to the directed landscape are shown for a few special models in \cite{dauvergne2018directed, dauvergne2021scaling, Wu23KPZ, aggarwal24direct,Dauvergne-Zhang24}. 
    
     It is well known (see, e.g., \cite{baik1999distribution,johansson2000shape,AmirCorwinQuastel2011}) that for special initial conditions, the one point marginals of $\mathcal{H}(\alpha,\tau)$ are described by the Tracy-Widom distribution and its relatives. Extensions to joint laws of multiple spatial points at equal time were obtained in \cite{PhahoferSpohn2002,BorodinFPS2007,BorodinFerrariSasamoto_transition,BaikFerrariPeche_stationary}, leading to explicit descriptions of the spatial process $\mathcal{H}(\cdot, \tau)$ for special initial conditions. In the breakthrough work \cite{matetski2021kpz}, the authors were able to find explicit Fredholm determinant formulas for the joint laws of $\mathcal{H}(\alpha_1,\tau;\h),\ldots,\mathcal{H}(\alpha_m,\tau;\h)$, starting from general upper semicontinuous initial conditions. This leads to a complete description of the Markovian dynamics of the fixed point. We also remark that the results of \cite{matetski2021kpz} were further generalized in \cite{Matetski-Remenik23a,BLSZ23,Matetski-Remenik23}.

    Joint laws along the time direction, or more generally in space-time, are much less
    known until recently. For the narrow wedge initial condition, a formula for the multi-time distribution was obtained by \cite{johansson_rahman_multitime}, which builds on the earlier work of two-time formulas in \cite{johansson2015two, johansson2018two}.  A different multipoint formula which works for both the narrow wedge and the flat initial conditions and possibly equal time parameters, was obtained in \cite{Liu2022}. We remark that a direct proof of the equivalence between the two formulas for the narrow wedge initial conditions is still missing due to the complicatedness of both formulas. 
    Two-time formulas for half-Brownian or Brownian initial conditions were also obtained recently in \cite{johansson_rahman_inhomogeneous,rahman25temporal}.  Besides these distribution formulas, there are also results on the correlation or tail properties of KPZ models at two  times, see \cite{LeDoussal17tail,LeDoussal18replica,le2017maximum,johansson20longshort,FerrariSphon16correlation,FerrariOccelli19time,CorwinGhosalHammond21}. We point out that all these mentioned results on the multi-time problems are studying the KPZ fixed point on $\realR$ and with special initial conditions. It is also worth mentioning the related work \cite{baik2019multipoint,BaikLiu21general,Liao22} for the multipoint distributions of TASEP models in periodic domain.

In this paper, we obtain the first formula for the space-time joint distributions of the KPZ fixed point with general initial conditions of compact support. We will discuss in more details in the following section.

\subsection{Main results}
 The main goal of this paper is to describe the space-time joint distributions of the two-dimensional random field $\mathcal{H}(\alpha,\tau;\h)$, with sufficiently general initial conditions $\h$, in the same spirit as in \cite{matetski2021kpz}. We start with introducing the spaces of initial conditions we will consider. The largest possible space of initial conditions from which the KPZ fixed point will be almost surely finite at all positive time is the following:
    \begin{equation}
        \uc:=\left\{\h:\mathbb{R}\to [-\infty,\infty) \text{ upper semicontinuous}, \h\not\equiv -\infty\text{ and } \limsup_{x\to \pm\infty}\frac{\h(\mathsf{x})}{\mathsf{x}^2}\leq 0\right\}.
    \end{equation}
 For technical reasons, We will mostly work with the dense subspace of $\uc$ consisting of functions that are $-\infty$ outside a compact set.
\begin{defn}[The function spaces of initial conditions and topology] 
Define
    \begin{equation}
        \ucc:= \{\h\in \uc: \text{there exists } L>0 \text{ such that } \h(\mathsf{x})=-\infty \text{ for all }|\mathsf{x}|>L\}.
    \end{equation}
   The space is equipped with the topology of local Hausdorff convergence of hypographs. We will call functions $\h\in \ucc$ compactly supported, where the support of $\h\in \uc$ is defined as 
    \begin{equation}
        \mathrm{supp}(\h):=\overline{\{\mathsf{x}\in \mathbb{R}: \h(\mathsf{x})\neq -\infty\}},
    \end{equation}
    and  $\overline{A}$ means the closure of the set $A$.
    \end{defn}
 Our main results are formulas for the joint distributions of the KPZ fixed point starting with initial condition $\h\in \ucc$,  at arbitrary many distinct space-time points $(\XX_1,\TT_1),\ldots, (\XX_m,\TT_m)$. To state the result, we introduce the following total ordering $\prec$ on the space-time plane  $\mathbb{R}\times \mathbb{R}$: 
     \begin{equation}
         (\XX_1, \TT_1)\prec (\XX_2,\TT_2) \Longleftrightarrow \TT_1<\TT_2, \text{ or } \TT_1=\TT_2\text{ and }\XX_1<\XX_2.
     \end{equation}
    \begin{thm}\label{thm: kpz_multitime}
        Let $\h\in \ucc$. Then for any $m\geq 1$ and any $m$ space-time points $(\XX_1,\TT_1)\prec \cdots\prec (\XX_m,\TT_m)\in \mathbb{R}\times \mathbb{R}_+$,  we have the following formula for the multipoint distribution of the KPZ fixed point $\he(\XX,\TT;\h)$:
        \begin{equation}\label{eq: multi_time}
            \mathbb{P}\left(\bigcap_{\ell=1}^m\left\{\he(\XX_\ell,\TT_\ell;\h)\leq \HH_\ell\right\}\right) =  \oint_0 \frac{\diff z_1}{2\pi\ii z_1(1-z_1)}\cdots \oint_0 \frac{\diff z_{m-1}}{2\pi\ii z_{m-1}(1-z_{m-1})} \DD_{\h}(z_1,\ldots,z_{m-1}),
        \end{equation}    
        where $\oint_0$ denotes an integral along a circle around the origin with counterclockwise orientation and sufficiently small radius. The function $\DD_{\h}(z_1,\ldots,z_{m-1})$ is defined as a Fredholm determinant in Definition \ref{def:D_general}. An equivalent definition through a series expansion will be discussed in Section \ref{sec: equiv_def}.
    \end{thm}
    Similar as in the narrow wedge case \cite{johansson_rahman_multitime, Liu2022}, our multipoint formula for the KPZ fixed point with a general initial condition has the form of contour integrals of a Fredholm determinant. The Fredholm determinant $\DD_{\h}$ has a block diagonal kernel acting on nested Airy-type contours. The dependency on the initial condition is only through the top-left corner of the kernel, characterized by a function $\limess_\h(\eta,\xi)$ defined on certain Airy contours, see Section \ref{sec: def_limess} for its definition. For the narrow wedge case, our formula recovers the one in \cite{Liu2022}. 

    The function $\limess_\h(\eta,\xi)$ is defined in terms of Brownian motion hitting expectations, an idea highly inspired by \cite{matetski2021kpz}. Indeed,  $\limess_\h(\eta,\xi)$ should be understood as the Brownian hitting operators in \cite{matetski2021kpz} written in Fourier-like spaces. Nevertheless, we stress that our results do not follow directly from \cite{matetski2021kpz}. In the multi-time situation, direct connections to determinantal point processes and the Eynard-Mehta theorem are lost and the bi-orthoganalization procedure here arises differently and takes a different form. On the contrary, our results are, in some sense, more general. Indeed, if we set the time parameters to be the same (which is allowed in the assumption of the theorem), the right-hand side of \eqref{eq: multi_time} can be shown to recover the formulas in \cite{matetski2021kpz}, after some quite non-trivial manipulations. We refer to Section \ref{sec: equal_equiv} for the details, see also \cite{LiuOrtiz25} which treats the special narrow wedge case.

\subsection{Outline of the proof and some discussions}
Theorem \ref{thm: kpz_multitime} is proved by taking a $1:2:3$ scaling limit of the corresponding multipoint distribution formulas of the totally asymmetric simple exclusion process (TASEP). The starting point is an algebraic formula obtained in \cite[Theorem 2.1]{Liu2022} for the multipoint (space-time) distribution of TASEP starting from any right-finite initial condition. The dependency of the TASEP formula on the initial condition is encoded in a function $\Kess_Y(v,u)$, which is characterized by an implicit reproducing-type property, see Definition \ref{def:KY_ess}. An explicit form of $\Kess_Y(v,u)$ in terms of symmetric functions was also obtained in \cite{Liu2022}, which is suitable for asymptotics for the step and (pseudo) flat initial conditions. Thus it led to the corresponding multipoint formula in \cite{Liu2022} for the KPZ fixed point starting from the narrow wedge and flat initial conditions after taking limits.

A key ingredient of this paper is an explicit probabilistic representation of the function $\Kess_Y(v,u)$, through a hitting expectation with respect to geometric random walks, see Theorem \ref{thm:essential_hitting}. The probabilistic representation is suitable for asymptotic analysis for more general initial conditions and leads to the Brownian hitting representation in the limit. For technical reasons, we first take the limit of the TASEP formula under the assumption that the KPZ fixed point starts with initial conditions consisting of finitely many narrow wedges, and then extend the formula to compactly supported initial condition at the level of the KPZ fixed point, using a density argument and the continuity of the law of the KPZ fixed point with respect to initial conditions.

Finally, we comment on our assumptions on the initial conditions. It would be desirable if one can get a formula that works for all initial conditions $\h\in \uc$, in particular, the flat initial condition $\h\equiv 0$. The reason we choose to restrict to the subspace $\ucc$ is not merely a technical issue. There are genuine structural difficulties in this generality: the characteristic function $\limess_\h(\eta,\xi)$ of the initial condition (see Definition \ref{def: limess}) may not be well-defined pointwisely in general. It might be possible to define $\limess_\h(\eta,\xi)$ in a proper sense case by case when $\h$ is not compactly supported. Indeed, for the flat initial condition $\h\equiv 0$, one can show from our formula that $\limess_\h(\eta,\xi)$ is the dirac delta integral kernel $\delta_{\eta=-\xi}$ \cite{Liu2022}. However, our current strategy requires to control the rate of the growth of the kernel, where compactness of the initial condition is needed.
%This suggests that in general, one should
It might be possible that one needs to interpret $\limess_\h(\eta,\xi)$ in the sense of distribution instead of a function. It might also be possible to conjugate our formula to real spaces so that the limit when the support goes to infinity exists, at the level of operators acting on real spaces. % We choose to stick to contour integral/Fourier type kernels, as it makes the algebraic structure of the multi-time formula much nicer and the analogy between narrow wedge and general initial conditions much more transparent. 
We leave it as a future project to extend our formula to any $\h\in \uc$, with a proper way to understand $\limess_\h(\eta,\xi)$. %a kernel acting on real spaces.
\subsection*{Notation and conventions}
Throughout the paper, we will mostly use English letters $x,t,h,u,v,w,\ldots$ for the pre-limit (TASEP) formulas and Greek letters $\alpha,\tau,\beta,\xi,\eta,\zeta,\ldots$ for the limiting (KPZ fixed point) formulas.  A detailed  summary of the notation we use is in the following table.

\begin{tabular}{|l|c|c|} 
    \hline
    \textbf{Notation} & \textbf{Pre-limit (TASEP) formulas} & \textbf{Limiting (KPZ fixed point) formulas} \\
    \hline
     time, space, height& $t,x,h$ & $\tau,\alpha,\beta$ \\
    \hline
    initial height function& $\mathsf{h}(\cdot)$ & $\h(\cdot)$\\
    \hline
    the height function& $H(x,t;\mathsf{h})$  & $\mathcal{H}(\alpha,\tau;\h)$ \\
    \hline
    integration contours & $\Sigma_{\LL},\Sigma_\RR$ & $ \Gamma_{\LL},\Gamma_{\RR}$\\
    \hline
    integration variables  & $u,v,w$ & $\xi,\eta,\zeta$\\
    \hline
\end{tabular}

\subsection*{Organization of the paper}
The rest of the paper is organized as follows. In Section \ref{sec: KPZ_formula} we present the formulas for the main part $\DD_\h$ appearing in Theorem \ref{thm: kpz_multitime}, both as a Fredholm determinant in Section \ref{sec: def_Dlim}, and as a Fredholm series expansion in Section \ref{sec: equiv_def}.  Then in Section \ref{sec: tasep formula} we present and prove the corresponding pre-limit formulas for TASEP, in particular in Section \ref{sec: tasep char} we prove that the characteristic function of the initial condition is given by a random walk hitting expectation. In Section \ref{sec: tasep convergence} we prove convergence of the TASEP formulas to the KPZ fixed point formulas, under the assumption that the initial condition of the KPZ fixed point consists of finitely many narrow wedges. We then extend the KPZ fixed point formula to any compactly supported initial condition in Section \ref{sec: compact support}. Finally, in Section \ref{sec: equal time} we show that at equal-time, our formula reduces to a genuine Fredholm determinant, which is then shown to be equivalent to the path integral formula of \cite{matetski2021kpz}. 
\subsection*{Acknowledgements}
 YL acknowledges the support of a Junior Fellowship at Institut Mittag-Leffler in Djursholm, Sweden, during the Fall semester of 2024, funded by the Swedish Research Council under grant no. 2021-06594. The work of ZL is partially supported by NSF DMS-2246683. The authors thank Konstantin Matetski, Daniel Remenik, Tejaswi Tripathi, Li-Cheng Tsai, Xuan Wu, and Nikolaos Zygouras for their comments and suggestions.
    
    \section{Multipoint distribution formula for the KPZ fixed point}\label{sec: KPZ_formula}
In this section, we explain in details the  function $\DD_\h(z_1,\ldots,z_{m-1})$ appearing on the right-hand side of \eqref{eq: multi_time}. Proofs of the formula will be deferred to Section \ref{sec: tasep convergence} and Section \ref{sec: compact support}.

\subsection{Fredholm determinant representation of $\DD_{\h}(z_1,\ldots,z_{m-1})$}\label{sec: def_Dlim}

The function $\DD_{\h}$ is defined in the same way as its analog in \cite{Liu2022} for the narrow wedge initial condition. The only difference is  the part involving the initial condition $\h$, which only appears in the top-left corner of the integral kernel. Below, we introduce the Fredholm determinant representation of the function $\DD_{\h}$.

Denote two regions of the complex plane
\begin{equation}
\complexC_\LL:=\left\{\zeta\in\complexC: \mathrm{Re}(\zeta)<0\right\}, \quad \text{ and }
\quad \complexC_\RR:=\left\{\zeta\in\complexC: \mathrm{Re}(\zeta)>0\right\}.
\end{equation}
 Let $\Gamma_{m,\LL}^{\out}$, $\ldots$, $\Gamma_{2,\LL}^{\out}$, $\Gamma_{1,\LL}$, $\Gamma_{2,\LL}^{\inn}$, $\ldots$, $\Gamma_{m,\LL}^{\inn}$ be $2m-1$ ``nested'' contours in the region $\complexC_\LL$. They are all unbounded contours from $\infty\mathrm{e}^{-2\pi\ii/3}$ to $\infty\mathrm{e}^{2\pi\ii/3}$. Moreover, they are located from the right (corresponding to the superscript ``$\out$'') to the left (``$\inn$''). The superscripts ``$\out$'' and ``$\inn$'' should be understood with respect to the point $-\infty$. Similarly, let $\Gamma_{m,\RR}^{\out}$,  $\ldots$, $\Gamma_{2,\RR}^{\out}$, $\Gamma_{1,\RR}$, $\Gamma_{2,\RR}^{\inn}$, $\ldots$, $\Gamma_{m,\RR}^{\inn}$ be $2m-1$ ``nested'' contours from left to right on the half plane $\complexC_\RR$.  They are from $\infty\mathrm{e}^{-\pi\ii/5}$ to $\infty\mathrm{e}^{\pi\ii/5}$. Their superscripts ``$\out$'' and ``$\inn$'' could be understood with respect to the point $+\infty$. Note that the angles for the left contours and right contours are chosen differently. The choice of the angles guarantees super-exponential decay of the kernel along the contours even if $\TT_i=\TT_{i+1}$ for some $i$. See \cite{Liu2022,LiuZhang25} for more discussions on the choices of the angles. See Figure~\ref{fig:contours_limit} for an illustration of the contours. 
 
 	\begin{figure}[t]
 	\centering
 	\begin{tikzpicture}[scale=1]
 	\draw [line width=0.4mm,lightgray] (-0.5,0)--(6,0) node [pos=1,right,black] {$\realR$};
 	\draw [line width=0.4mm,lightgray] (2.5,-2)--(2.5,2) node [pos=1,above,black] {$\mathrm{i}\realR$};
 	\fill (2.5,0) circle[radius=2.5pt] node [below,shift={(0pt,-3pt)}] {$0$};

			\draw[red,thick] (0.5,1.4)--(1.5,0)--(0.5,-1.4);
			\node[text width=0.1cm,font=\bfseries] at (0.2,1.6) {\scriptsize $\Gamma_{1,\LL}$};
			\draw[blue,thick] (5.5,1.4)--(3.5,0)--(5.5,-1.4);
			\node[text width=0.1cm,font=\bfseries] at (5.3,1.6) {\scriptsize $\Gamma_{1,\RR}$};
			\draw[blue,thick] (0,1.4)--(1,0)--(0,-1.4);
			\node[text width=0.1cm,font=\bfseries] at (-0.3,-1.6) {\scriptsize $\Gamma_{2,\LL}^\mathrm{in}$};
			\draw[red,thick] (6,1.4)--(4,0)--(6,-1.4);
			\node[text width=0.1cm,font=\bfseries] at (5.8,-1.6) {\scriptsize $\Gamma_{2,\RR}^\mathrm{in}$};
			\draw[blue,thick] (1,1.4)--(2,0)--(1,-1.4);
			\node[text width=0.1cm,font=\bfseries] at (0.7,-1.6) {\scriptsize $\Gamma_{2,\LL}^\mathrm{out}$};
			\draw[red,thick] (5,1.4)--(3,0)--(5,-1.4);
			\node[text width=0.1cm,font=\bfseries] at (4.8,-1.6) {\scriptsize $\Gamma_{2,\RR}^\mathrm{out}$};

		\end{tikzpicture}

 	\caption{Illustration of the contours for $m=2$: $\mathrm{S}_1$ is the union of the red contours and $\mathrm{S}_2$ is the union of the blue contours.}\label{fig:contours_limit}
 \end{figure}
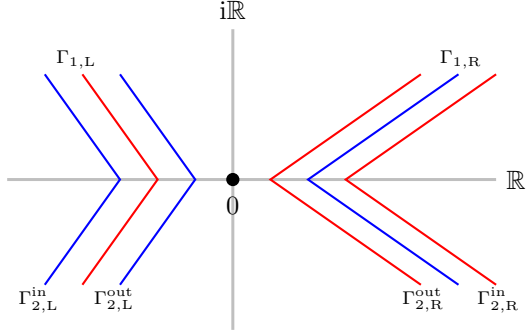

We define
\begin{equation*}
\Gamma_{\ell,\LL}:=\Gamma_{\ell,\LL}^\out\cup \Gamma_{\ell,\LL}^\inn, \qquad \Gamma_{\ell,\RR}:=\Gamma_{\ell,\RR}^\out\cup \Gamma_{\ell,\RR}^\inn,\qquad \ell=2,\ldots,m,
\end{equation*}
and
\begin{equation*}
\mathrm{S}_1:= \Gamma_{1,\LL} \cup \Gamma_{2,\RR} \cup \cdots \cup \begin{dcases}
\Gamma_{m,\LL}, & \text{ if $m$ is odd},\\
\Gamma_{m,\RR}, & \text{ if $m$ is even},
\end{dcases}
\end{equation*}
and 
\begin{equation*}
\mathrm{S}_2:= \Gamma_{1,\RR} \cup \Gamma_{2,\LL} \cup \cdots \cup \begin{dcases}
\Gamma_{m,\RR}, & \text{ if $m$ is odd},\\
\Gamma_{m,\LL}, & \text{ if $m$ is even}.
\end{dcases}
\end{equation*}

We introduce a measure on these contours. Let
\begin{equation}\label{eq:def_dmu}
\dd \mu(\zeta) = \dd \mu_{\boldsymbol{z}} (\zeta) :=
\begin{dcases}
\frac{-z_{\ell-1}}{1-z_{\ell-1}} \ddbarr{\zeta}, & \zeta \in \Gamma_{\ell,\LL}^\out \cup \Gamma_{\ell,\RR}^\out, \quad \ell=2,\ldots,m,\\
\frac{1}{1-z_{\ell-1}} \ddbarr{\zeta}, & \zeta\in \Gamma_{\ell,\LL}^\inn \cup \Gamma_{\ell,\RR}^\inn, \quad \ell=2,\ldots,m,\\
\ddbarr{\zeta}, & \zeta \in \Gamma_{1,\LL} \cup \Gamma_{1,\RR}.
\end{dcases}
\end{equation}

 Let $Q_1$ and $Q_2$ be as follows:,
\begin{equation}
\label{eq:Q}
Q_1(j) :=\begin{dcases}
1-z_{j}, & \text{ if $j$ is odd and $j<m$},\\
1-\frac{1}{z_{j-1}}, & \text{ if $j$ is even},\\
1,& \text{if $j=m$ is odd},
\end{dcases}
\qquad 
Q_2(j) :=\begin{dcases}
1-z_{j}, & \text{ if $j$ is even and $j<m$},\\
1-\frac{1}{z_{j-1}}, & \text{ if $j$ is odd and $j>1$},\\
1,& \text{if $j=m$ is even, or $j=1$}.
\end{dcases}
\end{equation}

\begin{defn}
	\label{def:D_general}
	We define
	\begin{equation*}
	\mathrm{D}_{\h}(z_1,\ldots,z_{m-1})=\det\left(\mathrm{I}-\mathrm{K}_1\mathrm{K}_{\h}\right),
	\end{equation*}
	where the operators
	\begin{equation*}
	\mathrm{K}_1: L^2(\mathrm{S}_2,\dd\mu) \to L^2(\mathrm{S}_1,\dd\mu),\qquad \mathrm{K}_{\h}: L^2(\mathrm{S}_1,\dd\mu)\to L^2(\mathrm{S}_2,\dd\mu)
	\end{equation*}
	are defined by their kernels
	\begin{equation}
	\label{eq:K1_lim}
	\mathrm{K}_1(\zeta,\zeta'):= \left(\delta_i(j) + \delta_i( j+ (-1)^i)\right) \frac{ \widehat{\mathrm{f}}_i(\zeta) }{\zeta-\zeta'} Q_1(j)
	\end{equation}
	and 
	\begin{equation}\label{eq:K2_lim}
	\mathrm{K}_{\h}(\zeta',\zeta):=
        \begin{cases}
            \left(\delta_j (i) + \delta_j(i - (-1)^j)\right) \frac{ \widehat{\mathrm{f}}_j(\zeta') }{\zeta-\zeta'} Q_2(i),\quad &i\geq 2,\\
            \delta_j(1)\widehat{\mathrm{f}}_1(\zeta')\limess_{\h}(\zeta',\zeta), & i=1,
        \end{cases}
	\end{equation}
	for any $\zeta\in (\Gamma_{i,\LL}\cup \Gamma_{i,\RR}) \cap \mathrm{S}_1$ and $\zeta'\in (\Gamma_{j,\LL}\cup \Gamma_{j,\RR}) \cap \mathrm{S}_2$ with $1\le i,j\le m$. Here the function
	\begin{equation}
	\label{eq:def_fi}
	\widehat{\mathrm{f}}_i(\zeta):=\begin{dcases}
	\mathrm{f}_i(\zeta), & \Re(\zeta)<0,\\
	\frac{1}{\mathrm{f}_{i}(\zeta)}, & \Re(\zeta)>0,
	\end{dcases}
	\end{equation}
	with
	\begin{equation}
	\label{eq:Fi}
	\mathrm{f}_i(\zeta) := \begin{dcases}
	\mathrm{e}^{-\frac{1}{3}(\TT_i-\TT_{i-1}) \zeta^3 + (\XX_i-\XX_{i-1}) \zeta^2 + (\HH_i-\HH_{i-1}) \zeta}, & i=2,\ldots,m,\\
	\mathrm{e}^{-\frac{1}{3}\TT_1 \zeta^3 + \XX_1 \zeta^2 + \HH_1 \zeta}, & i=1.
	\end{dcases}
	\end{equation}
    The kernel $\limess_{\h}(\zeta',\zeta)$ is defined in Section \ref{sec: def_limess}, see Definition \ref{def: limess}. 
\end{defn}

\subsubsection{The characteristic function $\limess_{\h}$}\label{sec: def_limess}
The dependency on the initial condition of the entire formula is through the function $\limess_{\h}$ defined on $\mathbb{C}_\RR\times \mathbb{C}_L$. Recall that we always use $\xi$ and $\eta$ to denote a variable on the $\Gamma$-contours on $\complexC_\LL$ and $\complexC_\RR$ respectively throughout the paper. Note that $\limess_{\h}$ is a function on $((\Gamma_{1,\LL}\cup \Gamma_{1,\RR}) \cap \mathrm{S}_2)\times ((\Gamma_{1,\LL}\cup \Gamma_{1,\RR})\cap \mathrm{S}_1)= \Gamma_{1,\RR}\times \Gamma_{1,\LL}$. So we will use the notation $\limess_{\h}(\eta,\xi)$ in the paper. The function $\limess_\h$ is defined using a Brownian motion hitting expectation as follows:
\begin{defn}\label{def: limess}
    Let $\B(t)$ be a two-sided Brownian motion with diffusivity constant $2$. Let $\h\in \ucc$ and $\ta_\pm$ be the hitting time of $\B$ to the hypograph of the positive (respectively, negative) part of $\h$, i.e., 
    \begin{equation}
        \ta_+:=\inf\{\XX\geq 0: \B(\XX)\leq \h(\XX)\},\quad \ta_-:=\sup\{\XX\leq 0: \B(\XX)\leq \h(\XX)\}.
    \end{equation}
    Then for any $\xi\in \mathbb{C}_\LL$ and $\eta\in \mathbb{C}_\RR$, we define 
    \begin{equation}\label{eq: limess}
    \begin{aligned}
        \limess_{\h}(\eta,\xi) &:=\int_{\mathbb{R}} \diff s\, \mathrm{e}^{+s\eta}\cdot  \mathbb{E}_{\B(0)=s}\left[\exp\left(-\ta_+ \xi^2-\B(\ta_+)\xi\right)\1_{\ta_+<\infty}\right]\\
        &\quad+\int_{\mathbb{R}} \diff s\, \mathrm{e}^{-s\xi}\cdot  \mathbb{E}_{\B(0)=s}\left[\exp\left(\ta_- \eta^2+\B(\ta_-)\eta\right)\1_{\ta_->-\infty}\right]\\
        &\quad-\int_{\mathbb{R}} \diff s\,   \mathbb{E}_{\B(0)=s}\left[\exp\left(-\ta_+ \xi^2-\B(\ta_+)\xi\right)\1_{\ta_+<\infty}\right]\cdot\mathbb{E}_{\B(0)=s}\left[\exp\left(\ta_- \eta^2+\B(\ta_-)\eta\right)\1_{\ta_->-\infty}\right].
    \end{aligned}
    \end{equation}
\end{defn}

We remark that when $\h(\alpha)=-\infty \mathbf{1}_{\alpha\ne 0}$ is the narrow wedge initial condition, the above characteristic function can be evaluated as $\limess_{\h}(\eta,\xi) = \frac{1}{\eta-\xi}$, which matches the corresponding function in \cite{Liu2022} for the narrow wedge initial condition. 

\begin{prop}\label{prop:convergence_lim}
     For $\h \in \ucc$, the function $\limess_{\h}(\eta,\xi)$ is well-defined and analytic for $\zeta\in \CL$ and $\eta\in \CR$. Moreover, for any $L>0$ and $\beta\in \mathbb{R}$ such that $\mathrm{supp}(\h)\subset [-L,L]$ and $\max_{\XX\in \mathbb{R}}\h(\XX)\leq \beta$, we have 
  %  \begin{equation}
   %     |\limess_{\h}(\eta,\xi)|\leq 6\cdot \mathrm{e}^{(\beta+1)\mathrm{Re}(\eta-\xi)+ L(|\xi|^2+|\eta|^2)}\cdot \left(1+\frac{1}{\mathrm{Re}(\eta)}+\frac{1}{\mathrm{Re}(-\xi)}\right).
  %  \end{equation}
       \begin{equation}
       \label{eq:bound_chi_h}
        |\limess_{\h}(\eta,\xi)|\leq \mathrm{e}^{(\beta+1)\mathrm{Re}(\eta-\xi)+ 2L(|\xi|^2+|\eta|^2)}\cdot \left(\frac{2}{\mathrm{Re}(\eta)}+\frac{2}{\mathrm{Re}(-\xi)} +8L + \frac{2^{5/2}L^{3/2}}{\sqrt{\pi}}\right).
    \end{equation}
\end{prop}
\begin{proof}
Assume $\mathrm{supp}(\h)\subset [-L,L]$ and $\max_{\XX\in \mathbb{R}}\h(\XX)\leq \beta$. We will show that the integrals on the right-hand side of \eqref{eq: limess} are absolutely convergent and uniformly bounded by the right-hand side of \eqref{eq:bound_chi_h}. Therefore $\limess_\h$ is well defined and analytic on $\CR\times\CL$.

We consider the first integral on the right-hand side of \eqref{eq: limess}. Note that $\mathbf{1}_{\tau_+<\infty}=\mathbf{1}_{\tau_+\le L}$, and $\mathrm{Re}(-\ta_+ \xi^2-\B(\ta_+)\xi)$ is bounded by $L|\xi^2|+\beta\mathrm{Re}(-\xi)$ when $\tau_+\le L$. These facts imply
\begin{equation}
\label{eq:bound_Kess02}
\begin{split}
    &\int_{\realR} \diff s\, |\mathrm{e}^{+s\eta}| \cdot  \mathbb{E}_{\B(0)=s}\left[|\exp\left(-\ta_+ \xi^2-\B(\ta_+)\xi\right)|\1_{\ta_+<\infty}\right]\\
    &\leq \mathrm{e}^{L|\xi^2|-\beta\mathrm{Re}(\xi)} \int_{\realR} \diff s\, \mathrm{e}^{s\mathrm{Re}(\eta)}\cdot \prob_{\B(0)=s}(\ta_+\leq L)\\
    &\leq \mathrm{e}^{L|\xi^2|-\beta\mathrm{Re}(\xi)} \int_{\realR} \diff s\, \mathrm{e}^{s\mathrm{Re}(\eta)}\cdot \prob_{\B(0)=s}(\sig_+\leq L),
    \end{split}
\end{equation}
where we used the following fact that $\ta_+\geq \sig_+$ with $\sig_+$ being the following new stopping time
\begin{equation*}
        \sig_+:=\inf\{\XX\geq 0: \B(\XX)\leq \beta\}.
\end{equation*}
Note that the distribution of $\sig_+$ can be computed using the reflection principle, see, e.g., \cite[(7.4.4)]{durretpte}. When $s\ge \beta+1$, we can estimate
\begin{equation}
    \begin{split}
        \prob_{\B(0)=s}(\sig_+\leq L)&=2\mathbb{P}_{\B(0)=s}\left(\B(L) \leq \beta\right) \\
        & = \frac{1}{\sqrt{\pi L}}\int_{-\infty}^{\beta} \mathrm{e}^{-\frac{(y-s)^2}{4L}} \diff y
         = \frac{1}{\sqrt{\pi L}} \mathrm{e}^{-\frac{(\beta-s)^2}{4L}} \int_{-\infty}^0 e^{-\frac{(\beta-s)z}{2L} -\frac{z^2}{4L}}\diff z\\
        &\leq \frac{1}{\sqrt{\pi L}} \mathrm{e}^{-\frac{(\beta-s)^2}{4L}} \int_{-\infty}^0 e^{-\frac{(\beta-s)z}{2L} }\diff z = \frac{2\sqrt{L}}{\sqrt{\pi}(s-\beta)}\mathrm{e}^{-\frac{(\beta-s)^2}{4L}} \\
        &\leq \frac{2\sqrt{L}}{\sqrt{\pi}}\mathrm{e}^{-\frac{(\beta-s)^2}{4L}}.
    \end{split}
    \end{equation} 
    Therefore, 
    \begin{equation}
    \label{eq:bound_Kess03}
        \begin{split}
            \int_{\realR} \diff s\, \mathrm{e}^{s\mathrm{Re}(\eta)}\cdot \prob_{\B(0)=s}(\sig_+\leq L)
            &\leq \int_{-\infty}^{\beta+1}\diff s\, \mathrm{e}^{s\mathrm{Re}(\eta)} + \frac{2\sqrt{L}}{\sqrt{\pi}}\int_{\beta+1}^\infty\diff s\, \mathrm{e}^{s\mathrm{Re}(\eta)-\frac{(\beta-s)^2}{4L}} \\
            &\leq \frac{\mathrm{e}^{(\beta+1)\mathrm{\Re}(\eta)}}{\mathrm{Re}(\eta)}+4L \mathrm{e}^{L(\mathrm{Re}(\eta))^2 + \beta\mathrm{Re}(\eta)}.
        \end{split}
    \end{equation}
By inserting this bound to \eqref{eq:bound_Kess02} and noting that $(\mathrm{Re}(\eta))^2\leq |\eta^2|$, we obtain
\begin{equation}
\label{eq:kess_1st_integral_bound}
    \int_{\realR} \diff s\, |\mathrm{e}^{+s\eta}| \cdot  \mathbb{E}_{\B(0)=s}\left[|\exp\left(-\ta_+ \xi^2-\B(\ta_+)\xi\right)|\1_{\ta_+<\infty}\right] \leq \mathrm{e}^{(\beta+1)\mathrm{Re}(\eta-\xi)+L(|\xi|^2+|\eta|^2) }\cdot \left(4L+\frac{1}{\mathrm{Re}(\eta)}\right).
\end{equation}

The second integral on the right-hand side of \eqref{eq: limess} can be handled similarly. We get
\begin{equation}
\label{eq:kess_2nd_integral_bound}
     \int_{\mathbb{R}} \diff s\, |\mathrm{e}^{-s\xi}|\cdot  \mathbb{E}_{\B(0)=s}\left[|\exp\left(\ta_- \eta^2+\B(\ta_-)\eta\right)|\1_{\ta_->-\infty}\right] \leq \mathrm{e}^{(\beta+1)\mathrm{Re}(\eta-\xi)+L(|\xi|^2+|\eta|^2) }\cdot \left(4L+\frac{1}{\mathrm{Re}(-\xi)}\right).
    \end{equation}

For the last integral, we need to slightly modify the estimate. We need a different estimate of the integral when $s\leq \beta+1$. We still have $\tau_+\leq L$ when $\tau_+<\infty$, and $\tau_-\geq -L$ when $\tau_->-\infty$. Also note that $\B(\tau_+)\leq \max_{t\in[0,L]}\B(t)$ and $\B(\tau_-)\leq \max_{t\in[-L,0]}\B(t)$. Thus we have
\begin{equation}
    \label{eq:kess_3rd_integral_bound1}
    \begin{split}
        &\int_{-\infty}^{\beta+1} \diff s\,   \mathbb{E}_{\B(0)=s}\left[|\exp\left(-\ta_+ \xi^2-\B(\ta_+)\xi\right)|\1_{\ta_+<\infty}\right]\cdot\mathbb{E}_{\B(0)=s}\left[|\exp\left(\ta_- \eta^2+\B(\ta_-)\eta\right)|\1_{\ta_->-\infty}\right]\\
        &\leq \mathrm{e}^{L(|\xi^2|+|\eta^2|)} \int_{-\infty}^{\beta+1} \diff s\,  \mathbb{E}_{\B(0)=s}  \left[\mathrm{e}^{\mathrm{Re}(-\xi)\max_{t\in[0,L]}\B(t)}\right]\cdot \mathbb{E}_{\B(0)=s}  \left[\mathrm{e}^{  \mathrm{Re}(\eta)\max_{t\in[-L,0]}\B(t)}\right]\\
        &=\mathrm{e}^{L(|\xi^2|+|\eta^2|)}\int_{-\infty}^{\beta+1} \diff s\,  \mathrm{e}^{s\mathrm{Re}(\eta-\xi)}\cdot \mathbb{E}_{\B(0)=0}  \left[\mathrm{e}^{\mathrm{Re}(-\xi)\max_{t\in[0,L]}\B(t) }\right]\cdot \mathbb{E}_{\B(0)=0}  \left[\mathrm{e}^{\mathrm{Re}(\eta)\max_{t\in[-L,0]}\B(t)}\right] \\
        &= \mathrm{e}^{L(|\xi^2|+|\eta^2|)} \cdot \frac{\mathrm{e}^{(\beta+1)\mathrm{Re}(\eta-\xi)}}{\mathrm{Re}(\eta-\xi)}\cdot \left(\frac{1}{\sqrt{\pi L}}\int_0^\infty \mathrm{e}^{-\frac{x^2}{4L}+\mathrm{Re}{(-\xi)}x}\diff x \right)
        \cdot \left(\frac{1}{\sqrt{\pi L}}\int_0^\infty \mathrm{e}^{-\frac{x^2}{4L}+\mathrm{Re}(\eta)x}\diff x \right)\\
        &\leq \mathrm{e}^{L(|\xi^2|+|\eta^2|)} \cdot \frac{\mathrm{e}^{(\beta+1)\mathrm{Re}(\eta-\xi)}}{\mathrm{Re}(\eta-\xi)}\cdot \mathrm{e}^{L(\mathrm{Re}{(-\xi}))^2 +L(\mathrm{Re}(\eta))^2}\\
        &\leq \mathrm{e}^{2L(|\xi^2|+|\eta^2|)+(\beta+1)\mathrm{Re}(\eta-\xi)} \cdot\left(\frac{1}{\mathrm{Re}(-\xi)} +\frac{1}{\mathrm{Re}(\eta)}\right),
    \end{split}
\end{equation}
where we used the fact that $\max_{t\in[0,L]}\B(t)$ and $\max_{t\in[-L,0]}\B(t)$ have the same distribution as $|\B(L)|$ when $\B(0)=0$.
When $s>\beta+1$, we use the following bound which is similar to \eqref{eq:bound_Kess02} and \eqref{eq:bound_Kess03},
\begin{equation}
\label{eq:kess_3rd_integral_bound2}
    \begin{split}
        &\int_{\beta+1}^\infty \diff s\,   \mathbb{E}_{\B(0)=s}\left[|\exp\left(-\ta_+ \xi^2-\B(\ta_+)\xi\right)|\1_{\ta_+<\infty}\right]\cdot\mathbb{E}_{\B(0)=s}\left[|\exp\left(\ta_- \eta^2+\B(\ta_-)\eta\right)|\1_{\ta_->-\infty}\right]\\
        &\leq \mathrm{e}^{\beta\mathrm{Re}(\eta-\xi)+ L(|\xi^2|+|\eta|^2)}\int_{\beta+1}^\infty \diff s\, \prob_{\B(0)=s}(\sig_+\leq L)\prob_{\B(0)=s}(\sig_-\geq -L) \\
        &\leq \mathrm{e}^{\beta\mathrm{Re}(\eta-\xi)+ L(|\xi^2|+|\eta|^2)} \cdot \frac{4L}{\pi}\int_{\beta+1}^\infty \diff s\, \mathrm{e}^{-\frac{(\beta-s)^2}{2L}}\\
        &\leq \mathrm{e}^{\beta\mathrm{Re}(\eta-\xi)+ L(|\xi^2|+|\eta|^2)} \cdot  \frac{2^{5/2}L^{3/2}}{\sqrt{\pi}},
    \end{split}
\end{equation}
where $\sig_-:=\sup\{\alpha\le 0: \B(\alpha)\le \beta\}$ is a stopping time which has the same law as $-\sig_+$. 
By combining all the bounds \eqref{eq:kess_1st_integral_bound}, \eqref{eq:kess_2nd_integral_bound}, \eqref{eq:kess_3rd_integral_bound1} and \eqref{eq:kess_3rd_integral_bound2}, we prove the desired bound \eqref{eq:bound_chi_h}.

\end{proof}

For convenience, we were using the origin $0$ as the starting point of the Brownian motions in the hitting expectation formula. The next proposition shows that this is not necessary, and one can start with any point on the real line $\mathbb{R}$ and get the same characteristic function $\limess_\h$.

\begin{prop}\label{prop: shift}
    One can change the starting point of the Brownian motions in the definition of the function $\limess_{\h}$ defined in \eqref{eq: limess}.  More precisely, for any $\a\in \mathbb{R}$ one has
    \begin{equation}\label{eq: limess2}
    \begin{aligned}
        \limess_{\h}(\eta,\xi) &= \mathrm{e}^{\a\eta^2}\int_{\mathbb{R}} \diff s\, \mathrm{e}^{+s\eta}\cdot  \mathbb{E}_{\B(\a)=s}\left[\exp\left(-\ta_+ \xi^2-\B(\ta_+)\xi\right)\1_{\ta_+<\infty}\right]\\
        &\quad+\mathrm{e}^{-\a \xi^2}\int_{\mathbb{R}} \diff s\, \mathrm{e}^{-s\xi}\cdot  \mathbb{E}_{\B(\a)=s}\left[\exp\left(\ta_- \eta^2+\B(\ta_-)\eta\right)\1_{\ta_->-\infty}\right]\\
        &\quad-\int_{\mathbb{R}} \diff s\,  \mathbb{E}_{\B(\a)=s}\left[\exp\left(\ta_- \eta^2+\B(\ta_-)\eta-\ta_+ \xi^2-\B(\ta_+)\xi\right)\1_{|\ta_\pm|<\infty}\right],
    \end{aligned}
    \end{equation}
    where the hitting times $\ta_\pm$ are now defined as 
    \begin{equation}\label{eq: hittingtimes}
        \ta_+:=\inf\{\XX\geq \a: \B(\XX)\leq \h(\XX)\},\quad \ta_-:=\sup\{\XX\leq \a: \B(\XX)\leq \h(\XX)\}.
    \end{equation}
    In particular if $\mathrm{supp}(\h)\subset [-L,L]$ for some $L>0$, then 
        \begin{equation}\label{eq: limess3}
    \begin{aligned}
        \limess_{\h}(\eta,\xi) &=\mathrm{e}^{-L\eta^2}\int_{\mathbb{R}} \diff s\, \mathrm{e}^{s\eta}\cdot  \mathbb{E}_{\B(-L)=s}\left[\exp\left(-\ta_+ \xi^2-\B(\ta_+)\xi\right)\1_{\ta_+\leq L}\right].
    \end{aligned}
    \end{equation}
\end{prop}
The proof of Proposition \ref{prop: shift} will be given in Section \ref{sec: shift}. We point out that it is purely a result about the Brownian motion, but we are not able to find it in the literature.

    \subsection{An equivalent series expansion formula}\label{sec: equiv_def}
    Due to the block diagonal structure of the kernel $\mathbf{K}_1$ and $\mathbf{K}_{\h}$, the Fredholm determinant $\mathrm{D}_{\h}(z_1,\ldots,z_{m-1})$ admits a series expansion, which we will be working with more frequently in the subsequent sections. To introduce the formula, we first introduce a few notation. Given $W=(w_1,\ldots,w_n)\in \mathbb{C}^n$ and $W'=(w_1',\ldots,w_m')\in \mathbb{C}^m$, we denote 
    \begin{equation}
        W\sqcup W':= (w_1,\ldots,w_n, w_1',\ldots,w_m')\in \mathbb{C}^{m+n}.
    \end{equation}
    Assume in addition that $n=m$ and $w_i\neq w_j'$ for all $1\leq i,j\leq n$. We denote 
    \begin{equation}
        \mathrm{C}(W;W'):= \det\begin{bmatrix}
            \displaystyle\frac{1}{w_i-w_j'}
        \end{bmatrix}_{1\leq i,j\leq n} = (-1)^{\frac{n(n-1)}{2}}\frac{\prod_{1\leq i<j\leq n}(w_j-w_i)(w_j'-w_i')}{\prod_{1\leq i,j\leq n}(w_i-w_j')},
    \end{equation}
    which is the usual Cauchy determinant. The Cauchy determinant $\mathrm{C}(W
    \sqcup W'; \widehat{W}\sqcup\widehat{W}')$ is defined in the same way with the combined variables $W\sqcup W'$ and another set of variables $\widehat{W}\sqcup\widehat{W}'$ with the same dimension as $W\sqcup W'$.
    
    \begin{prop}[Series expansion for $\DD_{\h}(z_1,\ldots,z_{m-1})$]\label{prop: series}
        Alternatively, we have
        \begin{equation}
            \DD_{\h}(z_1,\ldots,z_{m-1}) = \sum_{\substack{n_\ell\geq 0,\\1\leq \ell \leq m}}\frac{1}{(n_1!\cdots n_m!)^2}\mathrm{D}^{(\mathbf{n})}_{\h}(z_1,\ldots,z_{m-1}),
        \end{equation}
        where $\mathbf{n}= (n_1,\ldots,n_m)\in (\mathbb{Z}_{\geq 0})^m$, and
	\begin{equation}\label{eq: Dn_lim}
	\begin{split}
	&\mathrm{D}^{(\mathbf{n})}_{\h}(z_1,\ldots,z_{m-1})=\mathrm{D}^{(\mathbf{n})}_{\h}(z_1,\ldots,z_{m-1}; (\XX_1,\TT_1,\HH_1),\ldots,(\XX_m,\TT_m,\HH_m))\\
	&= \prod_{\ell=1}^{m-1} \left(1-z_\ell\right)^{n_\ell}\left(1-z_\ell^{-1}\right)^{n_{\ell+1}}\left(\prod_{\ell=1}^m\prod_{i_\ell=1}^{n_\ell} \int_{\Gamma_{\ell,\LL}}\mathrm{d}\mu_{\boldsymbol{z}}(\xi_{i_\ell}^{(\ell)})\int_{\Gamma_{\ell,\RR}}\mathrm{d}\mu_{\boldsymbol{z}}(\eta_{i_\ell}^{(\ell)})\right)\prod_{\ell=1}^m\prod_{i_\ell=1}^{n_\ell} \frac{\mathrm{f}_\ell(\xi_{i_\ell}^{(\ell)})}{ \mathrm{f}_\ell(\eta_{i_\ell}^{(\ell)})}\\
	&\quad\cdot  \det\begin{bmatrix}
	    \displaystyle\chi_{\h}(\eta^{(1)}_i,\xi^{(1)}_j)
	\end{bmatrix}_{1\leq i,j\leq n_1}\cdot \prod_{\ell=1}^{m-1}\mathrm{C}\left(\xib^{(\ell)}\sqcup \etab^{(\ell+1)}; \etab^{(\ell)}\sqcup \xib^{(\ell+1)}\right)\cdot \mathrm{C}(\xib^{(m)};\etab^{(m)}),
	\end{split}
        \end{equation}
        with $\xib^{(\ell)}=(\xi_1^{(\ell)},\ldots, \xi_{n_\ell}^{(\ell)})$ and $\etab^{(\ell)}=(\eta_1^{(\ell)},\ldots, \eta_{n_\ell}^{(\ell)})$, for $1\leq \ell\leq m$. Here 
        \begin{equation*}
            \mathrm{f}_\ell(w):= \mathrm{e}^{-\frac{1}{3}(\TT_\ell-\TT_{\ell-1})w^3+(\XX_\ell-\XX_{\ell-1})w^2+(\HH_\ell-\HH_{\ell-1})w},
        \end{equation*}
        for $1\leq \ell\leq m$, with the convention that $\TT_0=\XX_0=\HH_0:=0$.
    \end{prop}
     We remark that \eqref{eq: Dn_lim} looks slightly different from the version in \cite[(2.27)]{Liu2022} because we use the function $\mathrm{C}$ instead of the Vandermonde type products $\Delta$. The formula \eqref{eq: Dn_lim} also appears in \cite[Proposition 3.1]{LiuZhang25} when $\chi_{\h}(\eta,\xi)=1/(\eta-\xi)$ for the narrow wedge initial condition. The equivalence of Proposition \ref{prop: series} and Theorem \ref{thm: kpz_multitime} follows from \cite[Proposition 2.9]{Liu2022}, see also \cite[Lemma 4.8]{baik2019multipoint} and \cite[Lemma 5.6]{BaikLiu21general}.

    \section{Multipoint distribution formulas for TASEP}\label{sec: tasep formula}

Our formulas for the KPZ fixed point are obtained by taking $1:2:3$ scaling limit of the analogous formulas for the totally asymmetric simple exclusion process, which we discuss in this section. The totally asymmetric simple exclusion process (TASEP) on $\mathbb{Z}$ is a continuous-time Markov chain $X_t=(\mathrm{x}_i(t))_{i\geq 1}$, consisting of particles on $\mathbb{Z}$ performing independent Poisson random walks subject to the exclusion rule. Each particle tries to jump to its right neighbor after an independent exponential waiting time with rate $1$ but the jump is forbidden if the target site is occupied. The exponential clock is reset after each jump attempt. We will assume there is a right-most particle with index $1$ and label the particles from right to left, so the $i$-th particle at time $t$ has location $\mathrm{x}_i(t)$ and
    \begin{equation*}
        \cdots<\mathrm{x}_3(t)<\mathrm{x}_2(t)<\mathrm{x}_1(t).
    \end{equation*}
    The initial configuration is denoted by $Y=(y_i)_{i\geq 1}:= (\mathrm{x}_i(0))_{i\geq 1}$. Our key observation, following the work \cite{Liu2022}, is that the initial condition $Y$ can be encoded in a two-variable function $\Kess_Y(v,u)$, defined as an expectation involving random walk hitting problems. We begin by introducing this key object.

    \subsection{Characteristic function of the initial condition}\label{sec: tasep char}

In this subsection, we discuss how to characterize the initial condition in the TASEP formulas, and provide a probabilistic representation of the characterization.

One feature of the TASEP is that the distribution of any finite set of the rightmost particles, up to a fixed label $N$, is independent of the state of particles to their left. Conversely, a TASEP model with $N$ particles can be embedded into  a TASEP model with infinitely many particles, where the $N$ rightmost particles correspond to the $N$-particle system, and the states of all other particles are arbitrary. This feature will be used when we characterize the initial condition of TASEP with finitely many particles.

    Consider the following two simply connected regions of $\mathbb{C}$:
\begin{equation}\label{eq: regions}
    \Omega_\LL:=\left\{w\in\mathbb{C}: |w+1|<\frac{1}{2}\right\},\quad \Omega_\RR:=\left\{w\in\mathbb{C}: |w|<\frac{1}{2}\right\}.
\end{equation}

 The following characterization of the initial condition comes from \cite{Liu2022}, see Proposition 2.13 and the subsequent discussion in that paper for further details. It claimed that the initial condition is encoded in any function satisfying two conditions in the multipoint distribution formula of TASEP. We summarize these two conditions and introduce the concept of characteristic function below.
    \begin{defn}\label{def:KY_ess}
        Let $Y=(y_1>y_2>\cdots>y_N)$ where $N$ is a fixed integer.  We say $\Kess_Y$ is a characteristic function of $Y$, if it satisfies the following two conditions:
        \begin{enumerate}
            \item $\Kess_Y: \Omega_{\RR}\times \left(\Omega_{\LL}\backslash\{-1\}\right)\to \mathbb{C}$ is analytic.
            \item For any $1\leq i\leq N$, one has 
            \begin{equation}\label{eq:ess_residue}
                \oint_0 v^{-i}(v+1)^{y_i+i}\cdot \Kess_Y(v,u)\ddbarr{v} = -u^{-i}(u+1)^{y_i+i}.
            \end{equation}
        \end{enumerate}
    \end{defn}
    \begin{rmk}
        As pointed out in Proposition 2.13 of \cite{Liu2022} and the comments thereafter, for any given $Y$, there are infinitely many characteristic functions. However, the law of the TASEP model such as the multipoint distributions we are interested in here does not depend on the choice of $\Kess_Y$. This non-uniqueness comes from the nature of the TASEP model, as we discussed at the beginning of this subsection.
    \end{rmk}
    \begin{rmk}
    One could formally extend the concept of the characteristic ``function'' $\Kess_Y(v,u)$ to an infinite system with particles labeled on $\intZ_+$ and $Y=(\ldots,y_3,y_2,y_1)$ by defining $\Kess_Y(v,u)$ to satisfy \eqref{eq:ess_residue} for all $i\in\intZ_+$. We could extend it even further to a TASEP with particle labeled on $\intZ$, while the first condition is absorbed into the second condition by allowing $i\in\intZ$ in \eqref{eq:ess_residue}. The issue for these extensions is that $\Kess_Y$ is not necessarily well-defined as a function because of the convergence issue.
    \end{rmk}
    
    In \cite{Liu2022}, the author derived a characteristic function expressed in terms of symmetric functions for any $Y$, which is well-suited for asymptotic analysis under the step or flat initial condition. As a result, the author obtained the multipoint distribution of the KPZ fixed point for both the narrow-wedge and flat initial conditions. However, the characteristic function presented in \cite{Liu2022} is not suitable for asymptotic analysis with general initial conditions. One main contribution of this paper is the following characteristic function $\Kess_Y$ defined by an expectation involving random walk hitting problems, which turns out to be suitable for asymptotic analysis.
    %One main contribution of this paper is the following explicit expression of $\Kess_Y$ using an expectation involving random walk hitting problems.
    The idea is heavily inspired by the seminal work \cite{matetski2021kpz}: we compared the two formulas of \cite{matetski2021kpz} and \cite{Liu2022} and guessed the identity. %but the identity does not follow directly from their results. 
    We remark that a similar hitting expectation kernel expression has also been derived in \cite{BLSZ23}, which generalized the approach of \cite{matetski2021kpz}. Another interesting connection is that the left-hand side of the equation \eqref{eq:identity} in our proof has a similar structure with an expression of the $G(z_1,z_2)$ function in \cite[Proposition 4.6]{BLSZ23}, although they are still different\footnote{The left-hand side of \eqref{eq:identity} has two factors  $2^z$ and $2^{-G_\tau}$ which the $G(z_1,z_2)$ function in \cite{BLSZ23} does not have. This results in the different meanings of these two quantities: \eqref{eq:identity} is a binomial coefficient, while $G(z_1,z_2)$ in \cite{BLSZ23} is a probability. }.
    \begin{thm}\label{thm:essential_hitting}
        Let $(G_k)_{k\geq 0}$ be a geometric random walk with transition probability given by 
\begin{equation}\label{eq: transition_geometric}
	\mathbb{P}(G_{k+1}=x \mid G_k=y):= \frac{1}{2^{y-x}}\mathbf{1}_{x<y},
\end{equation}
and $\tau$ be the hitting time of $G$ to the strict epigraph of $Y$, namely
\begin{equation}\label{eq: stopping_geometric}
	\tau:= \min\{m\geq 0: G_m>y_{m+1}\}.
\end{equation}
Then the following function $\Kess_Y$ is a characteristic function of $Y$
%a valid expression for $\Kess_Y$ is given by 
\begin{equation}\label{eq:ess_hitting}
    \Kess_Y(v,u):= \sum_{z\in \mathbb{Z}}(2u+2)^z\cdot \mathbb{E}_{G_0=z}\left[\frac{2}{(2v+2)^{G_\tau+1}}\cdot \left(\frac{-v}{v+1}\right)^\tau\mathbf{1}_{\tau<N}\right].
\end{equation}
    \end{thm}

    We remark that for the narrow wedge initial condition $y_i=-i$, $1\le i\leq N$, the characteristic function defined above is equal to $\Kess_{Y}(v,u) = \sum_{z\ge 0} (2u+2)^z \cdot \frac{2}{(2v+2)^{z+1}} =\frac{1}{v-u}$, which matches the corresponding characteristic function in \cite{Liu2022} for the step initial condition.
\begin{proof}
    First we check the analyticity of $\Kess_Y(v,u)$ in $\Omega_\RR\times (\Omega_\LL\backslash\{-1\})$. Recall that $v\in \Omega_\RR=:\{|w|<1/2\}$ and $u\in \Omega_\LL=:\{|w+1|<1/2\}$. Note that the summand on the right-hand side of \eqref{eq:ess_hitting} is $0$ when $z\le y_{N}$, since the random walk will stay below the epigraph of $Y$ up to time $N$. Moreover, $\tau=0$ and $G_\tau=z$ when $z>y_1$. Therefore, we have
    \begin{equation}
    \label{eq:ess_kernel_rewriting}
    \begin{split}
        \Kess_Y(v,u)&=\sum_{z\ge y_1+1}^\infty (2u+2)^z \frac{2}{(2v+2)^{z+1}} +\sum_{z=y_N+1}^{y_1} (2u+2)^z\cdot \mathbb{E}_{G_0=z}\left[\frac{2}{(2v+2)^{G_\tau+1}}\cdot \left(\frac{-v}{v+1}\right)^\tau\mathbf{1}_{\tau<N}\right]\\
        &=\left(\frac{u+1}{v+1}\right)^{y_1+1}\cdot \frac{1}{v-u}+\sum_{z=y_N+1}^{y_1} (2u+2)^z\cdot \mathbb{E}_{G_0=z}\left[\frac{2}{(2v+2)^{G_\tau+1}}\cdot \left(\frac{-v}{v+1}\right)^\tau\mathbf{1}_{\tau<N}\right],
        \end{split}
     \end{equation}
    where we used the fact that $|u+1|<|v+1|$ to simplify the first summation. Note that the right-hand side of \eqref{eq:ess_kernel_rewriting} is a sum of finitely many terms each of which is analytic in $\Omega_\RR\times (\Omega_\LL\backslash\{-1\})$. This proves the analyticity of $\Kess_Y(v,u)$ in this domain.    
    
    Next, we verify that the right-hand side of \eqref{eq:ess_hitting} satisfies \eqref{eq:ess_residue} for all $1\leq i\leq N$. A Taylor expansion of $u^{-i}$ at $-1$ gives
\begin{equation*}
-u^{-i}(u+1)^{y_i+i}
=(-1)^{i+1}\sum_{j=0}^\infty {{i+j-1\choose j}}(u+1)^{y_i+i+j},
\end{equation*}
where the series converges absolutely for $u\in \Omega_\LL$. From \eqref{eq:ess_kernel_rewriting} we have seen that the right-hand side of \eqref{eq:ess_hitting} converges absolutely as a Laurent series in $u$ for $u\in \Omega_{\mathrm{L}}\backslash \{-1\}$. Hence, it is sufficient to show that
\begin{equation*}
\oint_{0}\frac{\d v}{2\pi\ii}v^{-i} (v+1)^{y_i+i}\cdot\frac{2^z}{v+1}\cdot \mathbb{E}_{G_0=z}\left[\frac{1}{(2v+2)^{G_\tau}}\left(\frac{-v}{v+1}\right)^{\tau}\mathbf{1}_{\tau<N}\right]=(-1)^{i+1}\1_{z\ge y_i+i} {z-y_i-1\choose i-1}.
\end{equation*}
By interchanging the contour integration and the expectation (which is justified by \eqref{eq:ess_kernel_rewriting}), the above is equivalent to
\begin{equation}
\label{eq:identity}
2^z\cdot \mathbb{E}_{G_0=z}\left[\oint_{0}\frac{\d v}{2\pi\ii}\frac{(v+1)^{y_i+i-1-G_\tau-\tau}}{(-v)^{i-\tau}} \frac{1}{2^{G_\tau}}\mathbf{1}_{\tau<N}\right]=-\1_{z\ge y_i+i} {z-y_i-1\choose i-1},
\end{equation}
which is, by using the assumption that $i\leq N$ and the fact that the $v$-integral vanishes when $i\ge \tau$,
\begin{equation}
\label{eq:identity_to_be_proved}
2^z\cdot \mathbb{E}_{G_0=z}\left[- {G_\tau -y_i-1\choose i-\tau-1}\frac{1}{2^{G_\tau}}\mathbf{1}_{\tau<i} \right]=-\1_{z\ge y_i+i} {z-y_i-1\choose i-1}.
\end{equation}
Below we use induction to prove~\eqref{eq:identity_to_be_proved} for any $1\le i\le N$.

When $i=1$, the expectation on the left-hand side of \eqref{eq:identity_to_be_proved} is nonzero if and only if $\tau=0$, which is equivalent to $z\ge y_1+1$. Moreover, when $\tau=0$, $G_\tau =G_0=z$. Thus \eqref{eq:identity_to_be_proved} holds.

Assuming the identity is true for $i-1$, we want to show it holds for $i$.

Note that when $z\ge y_1+1$, we have $\tau=0$ and both sides are equal.

When $z\le y_1$, we define $\hat G_k=G_{k+1}$ and $\hat y_k=y_{k+1}$ for $k=0,1,\ldots$. Then we have, by induction,
\begin{equation*}
2^{\hat z}\cdot \mathbb{E}_{\hat G_0=\hat z}\left[- {\hat G_{\hat \tau} -\hat y_{i-1}-1\choose i-\hat \tau-2}\frac{1}{2^{\hat G_{\hat \tau}}}\mathbf{1}_{\hat \tau<i-1} \right]=-\1_{\hat z\ge \hat y_{i-1}+{i-1}} {\hat z-\hat y_{i-1}-1\choose i-2},
\end{equation*}
here $\hat\tau:=\tau-1$. Thus, we have
\begin{equation}\label{eq: induction_hitting}
\sum_{\hat z=-\infty}^{z-1}\frac{1}{2^{z-\hat z}}\cdot \mathbb{E}_{\hat G_0=\hat z}\left[- {\hat G_{\hat \tau} -\hat y_{i-1}-1\choose i-\hat \tau-2}\frac{1}{2^{\hat G_{\hat \tau}}}\mathbf{1}_{\hat \tau<i-1} \right]=-\sum_{\hat z=-\infty}^{z-1}\frac{1}{2^{z}}\1_{\hat z\ge \hat y_{i-1}+{i-1}} {\hat z-\hat y_{i-1}-1\choose i-2}.
\end{equation}
By using the Markov property for the left-hand side of \eqref{eq: induction_hitting}, we have
\begin{equation*}
\mathbb{E}_{G_0= z}\left[- { G_{ \tau} - y_{i}-1\choose i- \tau-1}\frac{1}{2^{G_{ \tau}}}\mathbf{1}_{ \tau<i} \right]
=-\frac{\mathbf{1}_{z\geq y_i+i}}{2^z}\sum_{\hat z=y_{i}+i-1}^{z-1} {\hat z- y_{i}-1\choose i-2}=-\frac{\mathbf{1}_{z\geq y_i+i}}{2^z}{z-y_i-1\choose i-1},
\end{equation*}
where we used the identity $\sum_{m=k}^n{m-1\choose k-1}={n\choose k}$ in the last step. This finishes the induction and the proof.

\end{proof}
    \subsection{Multipoint distribution of TASEP with general initial configurations}
    As we discussed in the previous subsection, the multipoint distribution of TASEP only depends on a finite number of rightmost particles. Hence, in this subsection, we consider the TASEP model with finitely many particles. 
    
    The following theorem is essentially \cite[Theorem 2.1]{Liu2022}, where the integrand $\mathcal{D}_Y(z_1,\ldots,z_{m-1})$ is defined by a Fredholm determinant formula or a series expansion formula and the initial condition information is encoded in a characteristic function of $Y$. It was proved that the choice of characteristic functions does not affect the value of the $\mathcal{D}_Y$ function. We will present the formula with a new characteristic function $\Kess_Y$ as in \eqref{eq:ess_hitting} in Theorem \ref{thm:essential_hitting} that is suitable for asymptotic analysis.
    %but where the characteristic function $\Kess_Y$ was defined implicitly through Definition \ref{def:KY_ess}. With Theorem \ref{thm:essential_hitting}, the formula now becomes completely explicit and suitable for asymptotic analysis. 
    \begin{thm}\label{thm:TASEP_general} Given $Y=(y_N,\ldots,y_2,y_1)\in \mathbb{Z}^{N}$ satisfying $y_N<\cdots<y_1$, where $N$ is a positive integer. Consider TASEP with initial particle locations $X_0=Y$. Let $(k_1,t_1),\ldots,(k_m,t_m)$ be $m$ distinct points in $\{1,\ldots,N\}\times \mathbb{R}_+$ satisfying $t_1\le t_2\le\cdots\le t_m$. Then for any integers $a_1,\ldots,a_m$,
    \begin{equation}
        \mathbb{P}_Y\left(\bigcap_{\ell=1}^m\{\mathrm{x}_{k_\ell}(t_\ell)\geq a_\ell\}\right) = \oint_0 \frac{\diff z_1}{2\pi\ii z_1(1-z_1)}\cdots \oint_0 \frac{\diff z_{m-1}}{2\pi\ii z_{m-1}(1-z_{m-1})} \mathcal{D}_Y(z_1,\ldots,z_{m-1}),
    \end{equation}
    where $\mathbb{P}_Y$ denotes the probability given $X(0)=Y$. %The contours of integration are counterclockwise circles centered at the origin with radii less than 1. 
    The function $\mathcal{D}_Y(z_1,\ldots,z_{m-1})$ is defined as a Fredholm determinant in Definition \ref{def:operators_K1Y}, or equivalently as a series in Definition \ref{def:D_Y_series}. 
    \end{thm}

\subsubsection{Fredholm determinant representation of $\mathcal{D}_Y(z_1,\ldots,z_{m-1})$}
\label{sec:Fredholm_representation}

The definition of $\mathcal{D}_Y(z_1,\ldots,z_{m-1})$ is very similar to its limiting counterpart $\DD_\h(z_1,\ldots,z_{m-1})$ defined in Section \ref{sec: def_Dlim} and \ref{sec: equiv_def}, either as a Fredholm determinant $\det(\mathrm{I}-\mathcal{K}_1\mathcal{K}_Y)$ or as a Fredholm series expansion. We will only use the series expansion formula in this paper, but we present both formulas here for completeness and possible later uses. 

\paragraph{Spaces of the operators}
\label{sec:spaces_operators}

We will define the operators on two specific spaces of nested contours with complex measures depending on $\boldsymbol{z}=(z_1,\ldots,z_{m-1})$, where $z_\ell\ne 1$ for each $1\le \ell\le m-1$. Recall the definition of the two regions $\Omega_\LL$ and $\Omega_\RR$ from \eqref{eq: regions}.

Suppose $\Sigma_{m,\LL}^{\out},\ldots,\Sigma_{2,\LL}^{\out}$, $\Sigma_{1,\LL}$, $\Sigma_{2,\LL}^{\inn},\ldots,\Sigma_{m,\LL}^\inn$ are $2m-1$ nested simple closed contours, from outside to inside, in $\Omega_\LL$ enclosing the point $-1$. Similarly,
$\Sigma_{m,\RR}^{\out},\ldots,\Sigma_{2,\RR}^{\out}$, $\Sigma_{1,\RR}$, $\Sigma_{2,\RR}^{\inn},\ldots,\Sigma_{m,\RR}^\inn$ are $2m-1$ nested simple closed contours, from outside to inside, in $\Omega_\RR$ enclosing the point $0$.  See Figure~\ref{fig:contours_finite_time} for an illustration of the contours. These contours are all counterclockwise oriented. 
	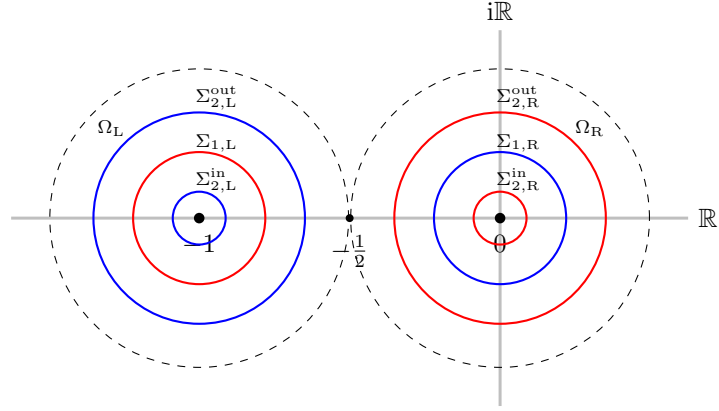
\begin{figure}[t]
 	\centering
 	\begin{tikzpicture}[scale=1]
 	\draw [line width=0.4mm,lightgray] (-2,0)--(7,0) node [pos=1,right,black] {$\realR$};
 	\draw [line width=0.4mm,lightgray] (4.5,-2.5)--(4.5,2.5) node [pos=1,above,black] {$\mathrm{i}\realR$};
 	\fill (2.5,0) circle[radius=1.5pt] node [below,shift={(0pt,-3pt)}] {$-\frac{1}{2}$};
    \fill (4.5,0) circle[radius=2pt] node [below,shift={(0pt,-3pt)}] {$0$};
    \fill (0.5,0) circle[radius=2pt] node [below,shift={(0pt,-3pt)}] {$-1$};
        \draw[dashed] (4.5,0) circle (56.5pt);
        \draw[dashed] (0.5,0) circle (56.5pt);
        \draw[red, thick] (4.5,0) circle (40pt);
        \draw[blue, thick] (4.5,0) circle (25pt);
        \draw[red, thick] (4.5,0) circle (10pt);
        \draw[blue, thick] (0.5,0) circle (40pt);
        \draw[red, thick] (0.5,0) circle (25pt);
        \draw[blue, thick] (0.5,0) circle (10pt);
                \node[text width=0.1cm,font=\bfseries] at (-0.8,1.2) {\scriptsize $\Omega_\LL$};
                 \node[text width=0.1cm,font=\bfseries] at (5.55,1.2) {\scriptsize $\Omega_\RR$};
			\node[text width=0.1cm,font=\bfseries] at (4.5,1.6) {\scriptsize $\Sigma_{2,\RR}^\mathrm{out}$};
                \node[text width=0.1cm,font=\bfseries] at (4.5,1) {\scriptsize $\Sigma_{1,\RR}$};
                \node[text width=0.1cm,font=\bfseries] at (4.5,0.5) {\scriptsize $\Sigma_{2,\RR}^\mathrm{in}$};
			\node[text width=0.1cm,font=\bfseries] at (0.5,1.6) {\scriptsize $\Sigma_{2,\LL}^\mathrm{out}$};
            \node[text width=0.1cm,font=\bfseries] at (0.5,1) {\scriptsize $\Sigma_{1,\LL}$};
                \node[text width=0.1cm,font=\bfseries] at (0.5,0.5) {\scriptsize $\Sigma_{2,\LL}^\mathrm{in}$};
		\end{tikzpicture}

 	\caption{Illustration of the contours for $m=2$: The regions $\Omega_\LL$ and $\Omega_\RR$ are the interior of the two dashed circles, from left to right; the three contours around $-1$ from outside to inside are $\Sigma_{2,\LL}^{\out},\Sigma_{1,\LL},\Sigma_{2,\LL}^{\inn}$ respectively; the three contours around $0$ from outside to inside are $\Sigma_{2,\RR}^{\out},\Sigma_{1,\RR},\Sigma_{2,\RR}^{\inn}$ respectively. $\mathcal{S}_1$ is the union of the red contours, and $\mathcal{S}_2$ is the union of the blue contours.}\label{fig:contours_finite_time}
\end{figure}

We define
\begin{equation}
	\label{eq:Sigma_contours}
\Sigma_{\ell,\LL}:=\Sigma_{\ell,\LL}^\out\cup \Sigma_{\ell,\LL}^\inn, \qquad \Sigma_{\ell,\RR}:=\Sigma_{\ell,\RR}^\out\cup \Sigma_{\ell,\RR}^\inn,\qquad \ell=2,\ldots,m,
\end{equation}
and
\begin{equation*}
\mathcal{S}_1:= \Sigma_{1,\LL} \cup \Sigma_{2,\RR} \cup \cdots \cup \begin{dcases}
\Sigma_{m,\LL}, & \text{ if $m$ is odd},\\
\Sigma_{m,\RR}, & \text{ if $m$ is even},
\end{dcases}
\end{equation*}
and 
\begin{equation*}
\mathcal{S}_2:= \Sigma_{1,\RR} \cup \Sigma_{2,\LL} \cup \cdots \cup \begin{dcases}
\Sigma_{m,\RR}, & \text{ if $m$ is odd},\\
\Sigma_{m,\LL}, & \text{ if $m$ is even}.
\end{dcases}
\end{equation*}

We introduce a measure on these contours in the same way as in~\eqref{eq:def_dmu}. Let
\begin{equation*}
\dd \mu(w) = \dd \mu_{\boldsymbol{z}} (w) :=
 \begin{dcases}
\frac{-z_{\ell-1}}{1-z_{\ell-1}} \ddbarr{w}, & w\in \Sigma_{\ell,\LL}^\out \cup \Sigma_{\ell,\RR}^\out, \quad \ell=2,\ldots,m,\\
\frac{1}{1-z_{\ell-1}} \ddbarr{w}, & w\in \Sigma_{\ell,\LL}^\inn \cup \Sigma_{\ell,\RR}^\inn, \quad \ell=2,\ldots,m,\\
\ddbarr{w}, & w\in \Sigma_{1,\LL} \cup \Sigma_{1,\RR}.
\end{dcases}
\end{equation*}

\paragraph{Operators $\mathcal{K}_1$ and $\mathcal{K}_Y$}
\label{sec:def_operators}

Now we introduce the operators $\mathcal{K}_1$ and $\mathcal{K}_Y$ to define $\mathcal{D}_Y(z_1,\ldots,z_{m-1})$ in Theorem~\ref{thm:TASEP_general}. We assume that   $\boldsymbol{z}=(z_1,\ldots,z_{m-1})$ is the same as in Section~\ref{sec:spaces_operators}. Let
\begin{equation*}
Q_1(j) :=\begin{dcases}
1-z_{j}, & \text{ if $j$ is odd and $j<m$},\\
1-\frac{1}{z_{j-1}}, & \text{ if $j$ is even},\\
1,& \text{if $j=m$ is odd},
\end{dcases}
\qquad 
Q_2(j) :=\begin{dcases}
1-z_{j}, & \text{ if $j$ is even and $j<m$},\\
1-\frac{1}{z_{j-1}}, & \text{ if $j$ is odd and $j>1$},\\
1,& \text{if $j=m$ is even, or $j=1$}.
\end{dcases}
\end{equation*}

\begin{defn}
	\label{def:operators_K1Y}
	We define
	\begin{equation*}
	\mathcal{D}_Y(z_1,\ldots,z_{m-1})=\det\left( \mathrm{I} - \mathcal{K}_1  \mathcal{K}_Y \right),
	\end{equation*}
	where the two operators
	\begin{equation*}
	\mathcal{K}_1: L^2(\mathcal{S}_2,\dd\mu) \to L^2(\mathcal{S}_1,\dd\mu),\qquad \mathcal{K}_Y: L^2(\mathcal{S}_1,\dd\mu)\to L^2(\mathcal{S}_2,\dd\mu)
	\end{equation*}
	are defined by their kernels
	\begin{equation}
	\label{eq:K1}
	\mathcal{K}_1(w,w'):= \left(\delta_i(j) + \delta_i( j+ (-1)^i)\right) \frac{ \widehat{f}_i(w) }{w-w'} Q_1(j),
	\end{equation}
	and
	\begin{equation}
	\label{eq:KY}
	\mathcal{K}_Y(w',w):= \begin{dcases}
	\left(\delta_j (i) + \delta_j(i - (-1)^j)\right) \frac{ \widehat{f}_j(w') }{w'-w} Q_2(i), & i\ge 2,\\
	\delta_j(1)\widehat{f}_j(w')\Kess_Y(w';w), & i=1,
	\end{dcases}
	\end{equation}
	for any $w\in (\Sigma_{i,\LL}\cup \Sigma_{i,\RR}) \cap \mathcal{S}_1$ and $w'\in (\Sigma_{j,\LL}\cup \Sigma_{j,\RR}) \cap \mathcal{S}_2$ with $1\le i,j\le m$. Here $\Kess_Y$ is the characteristic function given by \eqref{eq:ess_hitting}. The function
	\begin{equation*}
	\label{eq:fi}
	\widehat{f}_i(w):=\begin{dcases}
	f_i(w), & w\in \Omega_\LL\setminus\{-1\},\\
	\frac{1}{ f_i(w)}, & w\in \Omega_\RR\setminus\{0\},
	\end{dcases}
	\end{equation*}
	with
	\begin{equation}\label{eq: f_i}
	f_i(w) := \begin{dcases}
	 w^{k_i-k_{i-1}} (w+1)^{-(a_i-a_{i-1})-(k_i-k_{i-1})} \mathrm{e}^{(t_i-t_{i-1}) w}, & i=2,\ldots,m,\\
	 w^{k_1} (w+1)^{-a_1-k_1} \mathrm{e}^{t_1 w}, & i=1,
	\end{dcases}
	\end{equation}
	for all $w\in(\Omega_\LL\setminus\{-1\})\cup (\Omega_\RR\setminus\{0\})$.
\end{defn}

\subsubsection{Series expansion formula for $\mathcal{D}_Y(z_1,\ldots,z_{m-1})$}
\label{sec:series_expansion}
We will be working with the following series expansion formulas, which is equivalent to the Fredholm determinant formula in the previous section, by \cite[Proposition 2.9]{Liu2022}. We use the same notation and conventions as in Section \ref{sec: equiv_def} for the Cauchy determinants. 

\begin{defn}[Alternative definition of $\mathcal{D}_Y$]
	\label{def:D_Y_series}
	We have an alternative definition of $\mathcal{D}_Y$ below
	\begin{equation}
	\label{eq:series_expansion}
	\mathcal{D}_Y(z_1,\ldots,z_{m-1}):=\sum_{\boldsymbol{n}\in(\intZ_{\ge 0})^m}\frac{1}{(\boldsymbol{n}!)^2}\mathcal{D}_{Y}^{(\mathbf{n})}(z_1,\ldots,z_{m-1}),
	\end{equation}
	with $\boldsymbol{n}!=n_1!\cdots n_m!$ for $\boldsymbol{n}=(n_1,\ldots,n_m)$. Here	
		\begin{equation}\label{eq:D_nY}
	\begin{split}
	&\mathcal{D}^{(\mathbf{n})}_{Y}(z_1,\ldots,z_{m-1})=\mathcal{D}^{(\mathbf{n})}_{Y}(z_1,\ldots,z_{m-1}; (x_1,t_1,a_1),\ldots,(x_m,t_m,a_m))\\
	&= \prod_{\ell=1}^{m-1} \left(1-z_\ell\right)^{n_\ell}\left(1-z_\ell^{-1}\right)^{n_{\ell+1}}\left(\prod_{\ell=1}^m\prod_{i_\ell=1}^{n_\ell} \int_{\Sigma_{\ell,\LL}}\mathrm{d}\mu_{\boldsymbol{z}}(u_{i_\ell}^{(\ell)})\int_{\Sigma_{\ell,\RR}}\mathrm{d}\mu_{\boldsymbol{z}}(v_{i_\ell}^{(\ell)})\right)\prod_{\ell=1}^m\prod_{i_\ell=1}^{n_\ell} \frac{f_\ell(u_{i_\ell}^{(\ell)})}{ f_\ell(v_{i_\ell}^{(\ell)})}\\
	&\quad\cdot  \det\begin{bmatrix}
	    \Kess_Y(v^{(1)}_i,u^{(1)}_j)
	\end{bmatrix}_{1\leq i,j\leq n_1}\cdot \prod_{\ell=1}^{m-1}\mathrm{C}\left(U^{(\ell)}\sqcup V^{(\ell+1)}; V^{(\ell)}\sqcup U^{(\ell+1)}\right)\cdot \mathrm{C}(U^{(m)};V^{(m)}),
	\end{split}
        \end{equation}
        with $U^{(\ell)}=(u_1^{(\ell)},\ldots, u_{n_\ell}^{(\ell)})$ and $V^{(\ell)}=(v_1^{(\ell)},\ldots, v_{n_\ell}^{(\ell)})$, and the functions $f_\ell$ defined in~\eqref{eq: f_i} for $1\leq \ell\leq m$. 
\end{defn}

    \section{Convergence of the TASEP formula}\label{sec: tasep convergence}
    In this section, we will take the proper scaling limit of the TASEP formulas (see Theorem \ref{thm:TASEP_general}), to get the corresponding KPZ fixed point formulas.  We start with the setup for the proper rescaling.
    \subsection{TASEP height function and $1:2:3$ rescaling}    
    The TASEP particle configuration can be encoded into the corresponding height function $H(x,t)$, defined as the unique function $\mathbb{R}\times \mathbb{R}_+\to \mathbb{R}$ satisfying the following conditions:
    \begin{enumerate}
        \item $H(0,0)=0$,
        \item  $H(x+1,t)=H(x,t)+\widehat{\eta}(x,t)$ for all $x\in \mathbb{Z}$, where 
        \begin{equation}\label{eq: occupation}
            \widehat{\eta}(x,t)=\begin{cases}
            1\quad &\text{if there is a particle at site $x$ at time $t$},\\
            -1\quad &\text{if there is no particle at site $x$ at time $t$},
        \end{cases}
        \end{equation}
         \item $H(\cdot,t)$ is piecewise linear with constant slopes between consecutive integers.
    \end{enumerate}
    The dynamics of the height function is as follows. Each local maximum of the height function turns into a local minimum after an independent exponential time with rate $1$. After each flip of max to min, the height at the flip decreases by $2$ while the height at the other integer points remain unchanged. The values at general $x\in \mathbb{R}$ are then determined by linear interpolations. Note that here we follow the convention in \cite{matetski2021kpz} where the height function decreases in time, instead of increasing as in some other literature. More explicitly, for the TASEP Markov chain $X_t=(\mathrm{x}_i(t))_{i\geq 1}$ we define its height function as 
    \begin{equation}
       H(x,t):= -2\left(X_t^{-1}(x-1)-X_{0}^{-1}(-1)\right)-x,\quad \text{for }x\in \mathbb{Z},
    \end{equation}
    where 
    \begin{equation*}
        X_t^{-1}(u):=\inf\{k\in \mathbb{Z}: \mathsf{x}_k(t)\leq u\}.
    \end{equation*}
    In particular, the initial height function $\mathsf{h}$ corresponding to the initial particle configuration $Y$ is
    \begin{equation*}
         \mathsf{h}(x):= H(x,0)= -2\left(Y^{-1}(x-1)-Y^{-1}(-1)\right)-x,\quad \text{for }x\in \mathbb{Z}.
    \end{equation*}
    We will use $Y(\hh)$ or $\hh(Y)$ to represent the initial particle configuration $Y$ corresponding to the initial height function $\hh$ and vice versa. Under this identification we can express the joint distribution of particle configurations using the height functions and vice versa, for example
    \begin{equation}
        \mathbb{P}_{Y}\left(\bigcap_{\ell=1}^m \{\mathrm{x}_{k_\ell}(t_\ell)\geq a_\ell\}\right) = \mathbb{P}_{\hh(Y)}\left(\bigcap_{\ell=1}^m \{H(a_\ell,t_\ell)\leq -a_\ell-2k_\ell\}\right).
    \end{equation}
    Now we introduce the proper rescaling for the TASEP height function so that it will converge to the KPZ fixed point. For $\varepsilon>0$, we define the rescaled TASEP height function $\mathcal{H}^{\varepsilon}(\XX,\TT)$ for $(\XX,\TT)\in \mathbb{R}\times \mathbb{R}_+$ as follows:
    \begin{equation}
        \mathcal{H}^{\varepsilon}(\XX,\TT):= \varepsilon^{\frac{1}{2}}\left(H(2\varepsilon^{-1}\XX, 2\varepsilon^{-\frac{3}{2}}\TT)+ \varepsilon^{-\frac{3}{2}}\TT\right).
    \end{equation}
 In particular $\mathcal{H}^{\varepsilon}(\XX,0)=\varepsilon^{\frac{1}{2}}\cdot H(2\varepsilon^{-1}\XX,0)=: \h^{\varepsilon}(\XX)$. It was shown in \cite[Theorem 3.13]{matetski2021kpz} that if $\h^{\varepsilon}\to \h$ in $\uc$ as $\varepsilon\to 0$, then for any positive integer $m$ one has $(\mathcal{H}^\varepsilon(\cdot,\TT_1;\h^\varepsilon),\ldots,\mathcal{H}^{\varepsilon}(\cdot,\TT_m;\h^\varepsilon))$ converges in distribution to $(\mathcal{H}(\cdot,\TT_1;\h),\ldots,\mathcal{H}(\cdot,\TT_m;\h))$ in the topology of $\uc^m$, where $\mathcal{H}(\cdot,\cdot;\h)$ is the KPZ fixed point starting from the initial condition $\h$. This in particular implies 
 \begin{equation}\label{eq: multitime_limit}
    \mathbb{P}\left(\bigcap_{\ell=1}^m\left\{\mathcal{H}(\XX_\ell,\TT_\ell;\h)\leq \HH_\ell\right\}\right) = \lim_{\varepsilon\to 0}\mathbb{P}_{Y(\h^\varepsilon)}\left(\bigcap_{\ell=1}^m\left\{\mathrm{x}_{\frac{1}{2}\varepsilon^{-\frac{3}{2}}\TT_\ell-\varepsilon^{-1}\XX_\ell-\frac{1}{2}\varepsilon^{-\frac{1}{2}}\HH_\ell}(2\varepsilon^{-\frac{3}{2}}\TT_\ell)\geq 2 \varepsilon^{-1}\XX_\ell\right\}\right).
 \end{equation}
We will use \eqref{eq: multitime_limit} and Theorem \ref{thm:TASEP_general} to prove Theorem \ref{thm: kpz_multitime}. Our strategy is to first assume that the initial condition $\h$ is a linear combination of finitely many narrow wedges and prove convergence of the TASEP approximations for such initial conditions. Then we use the density of such initial conditions to extend \eqref{eq: multi_time} to all $\h\in \ucc$. 

 \subsection{Finitely many narrow wedges and approximations}
    \begin{defn}[Multiple narrow wedges]
    Define the space of initial height functions consisting of finitely many narrow wedges: 
    \begin{equation}
            \mnw:= \{\h\in \uc: \h(\nXX):= \sum_{k=0}^{M-1} \nHH_k\mathbf{1}_{\nXX= \nXX_k}-\infty\mathbf{1}_{\nXX\notin \{\nXX_k:0\leq k\leq M-1\}}, M\in \mathbb{Z}_+, \nHH_k\in \mathbb{R}, \nXX_0>\cdots>\nXX_{M-1}\}.
        \end{equation}
    We will also be working with the following subspace of $\mnw$ consisting of normalized height functions:
    \begin{equation}
            \mnw_0:= \{\h\in \mnw: \nXX_0=\nHH_0=0\}.
        \end{equation}
    \end{defn}
     We start with proving Theorem \ref{thm: kpz_multitime} under the additional assumption that the initial condition $\h$ for the KPZ fixed point is in $\mnw_0$, namely it takes the form 
    \begin{equation}\label{eq: severalwedge}
        \h(\nXX):= \sum_{k=0}^{M-1} \nHH_k\mathbf{1}_{\nXX=\nXX_k}-\infty\mathbf{1}_{\nXX\notin \{\nXX_k: 0\leq k\leq M-1\}},
    \end{equation}
    where $M\in \mathbb{Z}_+$, $0=\nXX_0>\nXX_1>\cdots>\nXX_{M-1}$ and $\nHH_0= 0$.  Note that if $\h(\nXX)=\sum_{k}\nHH_k\mathbf{1}_{\nXX\neq \nXX_k}-\infty\mathbf{1}_{\nXX\notin \{\nXX_k: k\}}\in \mnw$, then $\widehat{\h}(\cdot):= \h(\cdot +\nXX_0)-\nHH_0\in \mnw_0$. By the invariance property of the KPZ fixed point and also the structure of $\limess_\h$ one can extend the formula to $\h\in \mnw$ from $\mnw_0$, see Section \ref{sec: proof_thm} for  explanations.
    
    We approximate $\h\in \mnw_0$ by the following sequence of height functions $\{\h^\varepsilon\}_{\varepsilon>0}$: 
    \begin{equation}\label{eq: height_approximate}
        \h^\varepsilon(\nXX):= \varepsilon^{1/2} \mathsf{h}^\varepsilon(2\varepsilon^{-1}\nXX),
    \end{equation}
    where $\mathsf{h}^\varepsilon$ is piecewise linear with slope $\pm 1$ such that $\mathsf{h}^\varepsilon( 2\varepsilon^{-1}\nXX_k)= \varepsilon^{-1/2}\nHH_k+O(1)$ for each $0\leq k\leq M-1$. In terms of TASEP particle configurations, $\mathsf{h}^\varepsilon$ corresponds to setting the occupation functions $\widehat{\eta}(x,0)$ defined in \eqref{eq: occupation} as:
    \begin{equation}\label{eq: occupation_approx}
        \widehat{\eta}(x,0):=\begin{cases}
            +1,\ & \text{if } \varepsilon^{-1}(\nXX_k+\nXX_{k+1})+\varepsilon^{-\frac{1}{2}}\frac{\nHH_{k+1}-\nHH_{k}}{2}\leq x< 2\varepsilon^{-1}\nXX_k\text{ for some }0\leq k\leq M-1,\\
            -1,\  &\text{if } 2\varepsilon^{-1}\nXX_{k}\leq x<\varepsilon^{-1}(\nXX_{k-1}+\nXX_{k})+\varepsilon^{-\frac{1}{2}}\frac{\nHH_{k}-\nHH_{k-1}}{2}\text{ for some }0\leq k\leq M-1. 
        \end{cases}
    \end{equation}   
    Roughly, we are putting densely packed particles between $2\varepsilon^{-1}\nXX_k$ and $\varepsilon^{-1}(\nXX_k+\nXX_{k+1})+\varepsilon^{-\frac{1}{2}}\frac{\nHH_k-\nHH_{k+1}}{2}$ and no particles between $2\varepsilon^{-1}\nXX_{k}$ and $\varepsilon^{-1}(\nXX_k+\nXX_{k-1})+\varepsilon^{-\frac{1}{2}}\frac{\nHH_{k-1}-\nHH_{k}}{2}$, for $0\leq k\leq M-1$. Here $\nXX_{M}$ is understood as $-\infty$ and $\nXX_{-1}$ is understood as $+\infty$. 
    
   The following proposition implies Theorem \ref{thm: kpz_multitime} under the additional assumption that $\h\in \mnw_0$.
\begin{prop}\label{prop: convergence_nw}
    Given $\h\in \mnw_0$. Let $(\h^\varepsilon)_{\varepsilon>0}$ be the approximating sequence of initial height functions for TASEP defined as in \eqref{eq: height_approximate} and \eqref{eq: occupation_approx}. Given $z_1,\ldots,z_m\in \mathbb{C}$ with $|z_i|=r<1$ for $1\leq i\leq m-1$.  To lighten the notation we will suppress the dependency on $\varepsilon$ at most places and write 
    \begin{equation*}
        \mathcal{D}_{Y^\varepsilon}(z_1,\ldots,z_{m-1}) := \mathcal{D}_{Y(\mathsf{h}^\varepsilon)}(z_1,\ldots,z_{m-1}; \mathbf{k}^\varepsilon,\mathbf{a}^\varepsilon, \mathbf{t}^{\varepsilon}),
    \end{equation*}
    where $\mathcal{D}_{Y}(z_1,\ldots,z_{m-1})= \mathcal{D}_{Y}(z_1,\ldots,z_{m-1}; \mathbf{k},\mathbf{a},\mathbf{t})$ is defined in Section \ref{sec:Fredholm_representation}. Here we use boldface letters to denote vectors, for example, $\mathbf{k}:=(k_1,\ldots,k_m)$. Assume the parameters satisfy
    \begin{equation}
        k_\ell^\varepsilon :=\frac{1}{2}\varepsilon^{-\frac{3}{2}}\TT_\ell-\varepsilon^{-1}\XX_\ell-\frac{1}{2}\varepsilon^{-\frac{1}{2}}\HH_\ell+O(1),\quad  a^{\varepsilon}_\ell := 2 \varepsilon^{-1}\XX_\ell+O(1),\quad t^{\varepsilon}_{\ell}:= 2\varepsilon^{-\frac{3}{2}}\TT_\ell,\quad \text{for }1\leq \ell\leq m.
    \end{equation}
   Then we have 
    \begin{equation*}
        \lim_{\varepsilon\to 0}  \prod_{\ell=1}^{m-1} \oint_0 \frac{\diff z_\ell}{2\pi\ii z_\ell(1-z_\ell)}\mathcal{D}_{Y^\varepsilon}(z_1,\ldots,z_{m-1}) = \prod_{\ell=1}^{m-1} \oint_0 \frac{\diff z_\ell}{2\pi\ii z_\ell(1-z_\ell)}\mathrm{D}_{\h}(z_1,\ldots,z_{m-1}).
    \end{equation*}
\end{prop}
Proposition \ref{prop: convergence_nw} is a consequence of the following two lemmas and the dominated convergence theorem.  
    \begin{lm}\label{lm: convergence_term}
        Let $\mathcal{D}^{(\mathbf{n})}_{Y^\varepsilon}$ and $\mathrm{D}^{(\mathbf{n})}_{\h}$ be as in \eqref{eq:D_nY} and \eqref{eq: Dn_lim}, where $\h$ is given by \eqref{eq: severalwedge} and $Y^\varepsilon=Y(\mathsf{h}^\varepsilon)$ is described in \eqref{eq: height_approximate}. Then for each $\mathbf{n}\in (\mathbb{Z}_{\geq 0})^m$ and $(z_1,\ldots,z_{m-1})\in (\mathbb{D}(0,1))^{m-1}$, we have 
        \begin{equation}
            \lim_{\varepsilon\to 0}\mathcal{D}^{(\mathbf{n})}_{Y^\varepsilon} (z_1,\ldots,z_{m-1};  \mathbf{k}^\varepsilon,\mathbf{a}^\varepsilon, \mathbf{t}^{\varepsilon})=  \mathrm{D}^{(\mathbf{n})}_{\h} (z_1,\ldots,z_{m-1};  \bm{\XX},\bm{\HH},\bm{\TT}).
        \end{equation}
    \end{lm}

    \begin{lm}\label{lm: bound_term}
        There exists constant $C>0$ such that 
        \begin{equation}
            \left|\mathcal{D}^{(\mathbf{n})}_{Y^\varepsilon} (z_1,\ldots,z_{m-1};  \mathbf{k}^\varepsilon,\mathbf{a}^\varepsilon, \mathbf{t}^{\varepsilon})\right|\leq \prod_{\ell=1}^{m-1} \frac{(1+|z_{\ell+1}|)^{2n_{\ell+1}}}{|z_{\ell}|^{n_{\ell+1}}|1-z_\ell|^{n_{\ell+1}-n_\ell}}\cdot\prod_{\ell=1}^m n_\ell^{n_\ell}\cdot C^{n_1+\cdots+n_m},
        \end{equation}
        for any $\mathbf{n}=(n_1,\ldots,n_m)\in (\mathbb{Z}_{\geq 0})^m$ and $(z_1,\ldots,z_{m-1})\in (\mathbb{D}(0,1))^{m-1}$.
    \end{lm}
The remaining of this section is organized as follows: we will first prove a uniform bound for $\Kess_{Y^\varepsilon}$ and the pointwise convergence of $\Kess_{Y^\varepsilon}$ to $\chi_{\h}$ in Section \ref{sec: essential_convergence}. Then we will use these results to prove Lemma \ref{lm: convergence_term} in Section \ref{sec: proof_convergence} and Lemma \ref{lm: bound_term} in Section \ref{sec: proof_bound}, these complete the proof of Proposition \ref{prop: convergence_nw}. 
\subsubsection{Pointwise convergence of the characteristic function}\label{sec: essential_convergence}
For the approximating sequence of height functions $\mathsf{h}^\varepsilon$ described in \eqref{eq: occupation_approx}, we denote $Y^\varepsilon$ the corresponding particle configurations for TASEP. It consists of exactly $M$ clusters of densely packed particles. To lighten the notations, we denote temporarily the indices of the rightmost particle of each cluster by $\st_0,\ldots,\st_{M-1}$, from right to left. We have 

\begin{equation}\label{eq: indices}
\st_i:=-\lfloor\varepsilon^{-1}\nXX_i\rfloor-\lfloor\frac12\varepsilon^{-\frac{1}{2}}\nHH_i\rfloor+1,\quad y_{\st_i}:=2\lfloor\varepsilon^{-1}\nXX_i\rfloor, \quad 0\leq i\leq M-1.
\end{equation}
Recall that we assume $\nXX_0=\nHH_0=0$. Thus
\begin{equation}
\label{eq:scaling_st_yst}
    \st_0 = 1, \quad \text{ and  } y_{\st_0}=0.
\end{equation}

The goal of this section is to analyze the asymptotic behaviors of the characteristic function $\Kess_{Y^\varepsilon}(v,u)$. We write
\begin{equation}
    u= -\frac{1}{2}+\frac12\varepsilon^{\frac{1}{2}}\xi,\quad v = -\frac{1}{2}+\frac12\varepsilon^{\frac{1}{2}}\eta.
\end{equation}
 The main result of this section is summarized in the following proposition:
\begin{prop}
    Under the same assumption as in Proposition \ref{prop: convergence_nw}, we have 
    \begin{enumerate}[(a)]
        \item  For any $\xi\in \mathbb{C}_\LL, \eta\in \mathbb{C}_\RR$ fixed,
        \begin{equation}
        \label{eq:pointwise_convergence_chi}
            \lim_{\varepsilon\to 0}\frac{1}{2}\varepsilon^{\frac{1}{2}}\cdot \Kess_{Y^{\varepsilon}}\left(-\frac{1}{2}+\frac{1}{2}\varepsilon^{\frac{1}{2}}\eta,-\frac{1}{2}+\frac{1}{2}\varepsilon^{\frac{1}{2}}\xi\right)= \limess_{\mathfrak{h}}(\eta,\xi).
        \end{equation} 
        \item 
        Assume that $\varepsilon>0$, and  $\xi\in \mathbb{C}_\LL, \eta\in \mathbb{C}_\RR$ satisfy $0<|1+\varepsilon^{\frac{1}{2}}\xi|<1$. Then the following estimate holds
        \begin{equation}\label{eq: uniformbound}
        \begin{split}
            &\left|\frac12\varepsilon^{\frac{1}{2}}\cdot \Kess_{Y^{\varepsilon}}\left(-\frac{1}{2}+\frac12\varepsilon^{\frac{1}{2}}\eta,-\frac{1}{2}+\frac12\varepsilon^{\frac{1}{2}}\xi\right)\right|\\
    &\leq \frac{1}{\Re(\eta)}\left(1+(M-1)\frac{|1-\varepsilon \eta^2|^{\st_{M-1}}}{|1+\varepsilon^{\frac{1}{2}}\eta|^{2\st_{M-1}+y_{\st_{M-1}+1}}}\cdot   \frac{|1+\varepsilon^{\frac{1}{2}}\xi|^{y_{\st_{M-1}}+t_{M-1}}}{(2-|1+\varepsilon^{\frac{1}{2}}\xi|)^{\st_{M-1}-1}}\right).
            \end{split}
        \end{equation}
       As a corollary, if we further assume that $|\varepsilon^{\frac{1}{2}}\xi|<100^{-1}$ and $|\varepsilon^{\frac{1}{2}}\eta|<100^{-1}$, then we have
       \begin{equation}
       \label{eq:uniform_bound_quadratic}
           \left|\frac{\varepsilon^{\frac{1}{2}}}{2}\Kess_{Y^{\varepsilon}}\left(-\frac{1}{2}+\frac{\varepsilon^{\frac{1}{2}}}{2}\eta,-\frac{1}{2}+\frac{\varepsilon^{\frac{1}{2}}}{2}\xi\right)\right| \le \frac{\mathrm{e}^{C(|\xi|^2+|\eta|^2+|\xi|+|\eta|+1)}}{\Re(\eta)},
       \end{equation}
       where $C$ is a constant that only depends on the parameters $M$ and $\nXX_i,\nHH_i$, $0\le i\le M-1$.
    \end{enumerate}
\end{prop}

\begin{proof}
        We will prove part (b) first. Note that  the geometric random walk moves strictly downwards, so it can only go above the boundary at the beginning of each cluster, namely     
    \begin{equation*}
        \mathbb{P}(\tau\notin\{\st_0,\ldots,\st_{M-1}\})=0.
    \end{equation*}
        Here $\tau$ is defined as in \eqref{eq: stopping_geometric} and the indices $\st_0,\ldots,\st_{M-1}$ are as in \eqref{eq: indices}. Hence,
    \begin{equation}\label{eq: rescaled_kernel}
        \begin{aligned}
            &\frac{1}{2}\varepsilon^{\frac{1}{2}}\cdot\Kess_{Y^{\varepsilon}}\left(-\frac{1}{2}+\frac{1}{2}\varepsilon^{\frac{1}{2}}\eta,-\frac{1}{2}+\frac{1}{2}\varepsilon^{\frac{1}{2}}\xi\right)\\
            &= \varepsilon^{\frac{1}{2}}\sum_{z\in \mathbb{Z}}(1+\varepsilon^{\frac{1}{2}}\xi)^z\cdot \mathbb{E}_{G_0=z}\left[(1+\varepsilon^{\frac{1}{2}}\eta)^{-G_\tau-\tau-1}\cdot(1-\varepsilon^{\frac{1}{2}}\eta)^{\tau}\mathbf{1}_{\tau\le \max_\ell\{\st_\ell\}}\right]\\
            &= \varepsilon^{\frac{1}{2}}\sum_{k=0}^{M-1}\sum_{\substack{z_i\leq y_{\st_i}\\0\leq i\leq k-1}}\sum_{z_k>y_{\st_k}} \frac{(1+\varepsilon^{\frac{1}{2}}\xi)^{z_0}}{(1+\varepsilon^{\frac{1}{2}}\eta)^{z_k+1}}\cdot \frac{(1-\varepsilon^{\frac{1}{2}}\eta)^{\st_k}}{(1+\varepsilon^{\frac{1}{2}}\eta)^{\st_k}}\cdot \prod_{i=0}^{k-1} p_{\st_{i+1}-\st_{i}}(z_{i+1}-z_{i}),
        \end{aligned}
    \end{equation}
    where $p_{t-s}(z-y)$ is the transition probability $\mathbb{P}(G_t=z|G_s=y)$ for the geometric random walk $(G_k)_{k\geq 0}$ defined as in \eqref{eq: transition_geometric}. It admits the following expression:
    \begin{equation*}
        p_{t-s}(z-y)=2^{z-y}\binom{y-z-1}{t-s-1},
    \end{equation*}
    for $t-s\in \mathbb{Z}_+$ and $z-y\in \mathbb{Z}_-$.   Then \eqref{eq: rescaled_kernel} implies
    \begin{equation}
    \label{eq: rescaled_kernel_uniform_bound}
    \begin{aligned}
        &\left|\frac{1}{2}\varepsilon^{\frac{1}{2}}\cdot\Kess_{Y^{\varepsilon}}\left(-\frac{1}{2}+\frac{1}{2}\varepsilon^{\frac{1}{2}}\eta,-\frac{1}{2}+\frac{1}{2}\varepsilon^{\frac{1}{2}}\xi\right)\right|\\
        &\leq 
   \varepsilon^{\frac{1}{2}}\sum_{k=0}^{M-1}\sum_{\substack{z_i\leq y_{\st_i}\\0\leq i\leq k-1}}\sum_{z_k>y_{\st_k}} \frac{|1+\varepsilon^{\frac{1}{2}}\xi|^{z_0}}{|1+\varepsilon^{\frac{1}{2}}\eta|^{z_k+1}}\cdot \frac{|1-\varepsilon^{\frac{1}{2}}\eta|^{\st_k}}{|1+\varepsilon^{\frac{1}{2}}\eta|^{\st_k}}\cdot \prod_{i=0}^{k-1} p_{\st_{i+1}-\st_{i}}(z_{i+1}-z_{i}).
    \end{aligned}
    \end{equation}

    \bigskip

    (b)     We bound each term on the right-hand side of \eqref{eq: rescaled_kernel_uniform_bound} corresponding to index $k$.  For $k=0$ we have 
\begin{equation}
\label{eq:bound_k=0}
    \varepsilon^{\frac{1}{2}}\sum_{z_0=y_{\st_0}+1}^\infty \frac{|1+\varepsilon^{\frac{1}{2}}\xi|^{z_0}}{|1+\varepsilon^{\frac{1}{2}}\eta|^{z_0+1}}= \frac{\varepsilon^{\frac{1}{2}}}{|1+\varepsilon^{\frac{1}{2}}\eta|-|1+\varepsilon^{\frac{1}{2}}\xi|}.
\end{equation}
For $k\geq 1$, we have 
    \begin{equation}\label{eq: essential_bound}
    \begin{split}
        &\varepsilon^{\frac{1}{2}}\sum_{\substack{z_i\leq y_{\st_i}\\0\leq i\leq k-1}}\sum_{z_k>y_{\st_k}} \frac{|1+\varepsilon^{\frac{1}{2}}\xi|^{z_0}}{|1+\varepsilon^{\frac{1}{2}}\eta|^{z_k+1}}\cdot \frac{|1-\varepsilon^{\frac{1}{2}}\eta|^{\st_k}}{|1+\varepsilon^{\frac{1}{2}}\eta|^{\st_k}}\cdot \prod_{i=0}^{k-1} p_{\st_{i+1}-\st_{i}}(z_{i+1}-z_{i})\\
        &\leq \varepsilon^{\frac{1}{2}}\sum_{z_0\leq y_{\st_0}, z_{k}>y_{\st_k}} \frac{|1+\varepsilon^{\frac{1}{2}}\xi|^{z_0}}{|1+\varepsilon^{\frac{1}{2}}\eta|^{z_k+1}}\cdot \frac{|1-\varepsilon^{\frac{1}{2}}\eta|^{\st_k}}{|1+\varepsilon^{\frac{1}{2}}\eta|^{\st_k}}\cdot p_{\st_{k}-\st_{0}}(z_{k}-z_{0})\\
        &=\varepsilon^{\frac{1}{2}}\frac{|1-\varepsilon^{\frac{1}{2}}\eta|^{\st_k}}{|1+\varepsilon^{\frac{1}{2}}\eta|^{\st_k}}\cdot \sum_{\delta = y_{\st_k}-y_{\st_0}+1}^{\st_0-\st_k}  |1+\varepsilon^{\frac{1}{2}}\eta|^{-\delta}p_{\st_k-\st_0}(\delta)\left(\sum_{z_0= y_{\st_k}-\delta+1}^{y_{\st_0}}  \frac{|1+\varepsilon^{\frac{1}{2}}\xi|^{z_0}}{|1+\varepsilon^{\frac{1}{2}}\eta|^{z_0+1}}\right)\\
        &= \varepsilon^{\frac{1}{2}}\frac{|1-\varepsilon^{\frac{1}{2}}\eta|^{\st_k}}{|1+\varepsilon^{\frac{1}{2}}\eta|^{\st_k}}\cdot \sum_{\delta = y_{\st_k}-y_{\st_0}+1}^{\st_0-\st_k}  |1+\varepsilon^{\frac{1}{2}}\eta|^{-\delta}p_{\st_k-\st_0}(\delta)\left( \frac{\frac{|1+\varepsilon^{\frac{1}{2}}\xi|^{y_{\st_k}-\delta+1}}{|1+\varepsilon^{\frac{1}{2}}\eta|^{y_{\st_k}-\delta+1}}- \frac{|1+\varepsilon^{\frac{1}{2}}\xi|^{y_{\st_0}+1}}{|1+\varepsilon^{\frac{1}{2}}\eta|^{y_{\st_0}+1}}}{|1+\varepsilon^{\frac{1}{2}}\eta|-|1+\varepsilon^{\frac{1}{2}}\xi|}\right).
    \end{split}
    \end{equation}
By the assumptions of $\xi$ and $\eta$, we have 
\begin{equation}
    |1-\varepsilon^{\frac{1}{2}}\eta|<|1+\varepsilon^{\frac{1}{2}}\eta|,\quad |1+\varepsilon^{\frac{1}{2}}\xi|<1<|1+\varepsilon^{\frac{1}{2}}\eta|.
\end{equation}
Hence
\begin{equation}
    \left|\frac{1-\varepsilon^{\frac{1}{2}}\eta}{1+\varepsilon^{\frac{1}{2}}\eta}\right|^{\st_k}\leq 1,\quad \left|\frac{1+\varepsilon^{\frac{1}{2}}\xi}{1+\varepsilon^{\frac{1}{2}}\eta}\right|^{y_{\st_0}+1}\leq \left|\frac{1+\varepsilon^{\frac{1}{2}}\xi}{1+\varepsilon^{\frac{1}{2}}\eta}\right|^{y_{\st_k}-\delta+1}, \quad \text{for all }y_{\st_k}-y_{\st_0}+1\leq \delta, 
\end{equation}
where we are using the fact that $\st_k\geq 0$ and $y_{\st_0}\geq y_{\st_k}-\delta+1$ for all $y_{\st_k}-y_{\st_0}+1\leq \delta$. Thus, we conclude that the right-hand side of \eqref{eq: essential_bound} is bounded above by 
\begin{equation}\label{eq: bound_epsilon}
    \frac{\varepsilon^{\frac{1}{2}}}{|1+\varepsilon^{\frac{1}{2}}\eta|-|1+\varepsilon^{\frac{1}{2}}\xi|}\cdot\frac{|1-\varepsilon^{\frac{1}{2}}\eta|^{\st_k}}{|1+\varepsilon^{\frac{1}{2}}\eta|^{\st_k}}\cdot \frac{|1+\varepsilon^{\frac{1}{2}}\xi|^{y_{\st_k}+1}}{|1+\varepsilon^{\frac{1}{2}}\eta|^{y_{\st_k}+1}}\cdot \sum_{\delta=y_{\st_k}-y_{\st_0}+1}^{\st_0-\st_k} |1+\varepsilon^{\frac{1}{2}}\xi|^{-\delta}p_{\st_k-\st_0}(\delta).
\end{equation}
The summation over $\delta$ is bounded above by
\begin{equation}
\label{eq:bound_sum_delta}
    \sum_{\delta\leq \st_0-\st_k}|1+\varepsilon^{\frac{1}{2}}\xi|^{-\delta}p_{\st_k-\st_0}(\delta) = \left(\frac{|1+\varepsilon^{\frac{1}{2}}\xi|}{2-|1+\varepsilon^{\frac{1}{2}}\xi|}\right)^{\st_k-\st_0},
\end{equation}
which follows from a standard moment generating function computation for the geometric random walk. Finally, note that $|1+\varepsilon^{\frac{1}{2}}\eta|\geq 1+\varepsilon^{\frac{1}{2}}\mathrm{Re(\eta)}>1$ for  $\eta\in\complexC_\RR$ and $|1+\varepsilon^{\frac{1}{2}}\xi|<1$ by our assumption, we have
\begin{equation}\label{eq: difference_bound}
    \frac{\varepsilon^{\frac{1}{2}}}{|1+\varepsilon^{\frac{1}{2}}\eta|-|1+\varepsilon^{\frac{1}{2}}\xi|} \leq  \frac{\varepsilon^{\frac{1}{2}}}{1+\varepsilon^{\frac{1}{2}}\mathrm{Re}(\eta)-1} = \frac{1}{\mathrm{Re}(\eta)}.
\end{equation}
By using the bound from \eqref{eq: difference_bound} in \eqref{eq: bound_epsilon} and summing over $k$, we arrive at the desired estimate \eqref{eq: uniformbound}.

For \eqref{eq:uniform_bound_quadratic}, we note the following simple inequality
\begin{equation}
    C_1|z|< \log(|1+z|) <C_2|z|,\quad \text{ for all $z$ satisfying } |z|<100^{-1}, 
\end{equation}
for some constants $C_1$ and $C_2$ that are independent of $z$. Therefore by our assumption,
\begin{equation}
    \frac{|1-\varepsilon \eta^2|^{\st_{M-1}}}{|1+\varepsilon^{\frac{1}{2}}\eta|^{2\st_{M-1}+y_{\st_{M-1}+1}}} \le \mathrm{e}^{C_2|\eta^2| \varepsilon\st_{M-1} -(|C_1|+|C_2|)|\eta||\varepsilon^{1/2}(2\st_{M-1}+y_{\st_{M-1}+1})|}\le \mathrm{e}^{C(|\eta^2|+|\eta|+1)},
\end{equation}
for some large constant $C$ by using \eqref{eq:scaling_st_yst}. For the other factor, we note that 
\begin{equation}
    x(2-x) \ge 1-c^2 \quad \text{ when } 1-c<x<1+c \text{ and } 0<c<1.
\end{equation}
Therefore
\begin{equation}
    |1+\varepsilon^{\frac12}\xi| (2- |1+\varepsilon^{\frac12}\xi|) \ge 1- \varepsilon |\xi^2|,
\end{equation}
and
\begin{equation}
    (|1+\varepsilon^{\frac12}\xi| (2- |1+\varepsilon^{\frac12}\xi|))^{-\st_{M-1}+1}
    \leq (1- \varepsilon |\xi^2|)^{-\st_{M-1}+1} \leq^{C_1|\xi^2|\varepsilon(-\st_{M-1}+1)} \le\mathrm{e}^{C(|\xi^2|+1)},
\end{equation}
for some constant $C$ by using \eqref{eq:scaling_st_yst}. Finally,
\begin{equation}
    |1+\varepsilon^{1/2}\xi|^{2\st_{M-1}+y_{\st_{M-1}-1}} \le\mathrm{e}^{(|C1|+|C_2|)|\xi|\cdot \varepsilon^{1/2}|2\st_{M-1}+y_{\st_{M-1}-1}|} \le \mathrm{e}^{C(|\xi|+1)},
\end{equation}
for some constant $C$ by using \eqref{eq:scaling_st_yst}. Combining the above estimates, we obtain \eqref{eq:uniform_bound_quadratic}.
 
    \bigskip

        (a) Now we prove part (a). We start with rewriting $\chi_{\h}(\eta,\xi)$ under the assumption that $\h\in \mnw_0$, recall the definition of $\chi_{\h}(\eta,\xi)$ from \eqref{eq: limess}. For $\mathrm{supp}(\h)=\{\nXX_0,\ldots,\nXX_{M-1}\}$ with $0=\nXX_0>\cdots>\nXX_{M-1}$, one has $\mathbb{P}(\ta_+\neq 0)=0$ where $\ta_+$ is defined in \eqref{eq: hittingtimes}. Thus, it is easy to check that the first and third term on the right-hand side of \eqref{eq: limess} are both equal to $\frac{\mathrm{e}^{\h(0)(\eta-\xi)}}{\eta-\xi}$ and they cancel each other. On the other hand, since $\mathbb{P}(\ta_-\notin \{\nXX_0,\ldots,\nXX_{M-1}\})=0$, the second term on the right-hand side of \eqref{eq: limess} is given by 
        \begin{equation}
        \label{eq:chi_integral}
            \sum_{k=0}^{M-1}\int_{\substack{\rs_i\geq \nHH_i, 0\leq i\leq k-1;\\ \rs_k<\nHH_k}}\mathrm{e}^{\rs_k\eta+\nXX_k\eta^2-\rs_0\xi}\cdot\prod_{i=0}^{k-1}\mathrm{p}_{\nXX_{i}-\nXX_{i+1}}(\rs_{i+1}-\rs_{i})\diff \rs_0\cdots\diff \rs_{k},
        \end{equation}
        where $\mathrm{p}_{\nXX-\nXX'}(\rs-\rs') = \frac{1}{\sqrt{4\pi(\nXX-\nXX')}}\mathrm{e}^{-\frac{(\rs-\rs')^2}{4(\nXX-\nXX')}}$ 
        is the transition density of a Brownian motion with diffusivity constant $2$. Thus for $\h\in \mnw_0$, we have 
                \begin{equation}
        \begin{aligned}
            \limess_{\h}(\eta,\xi) = \sum_{k=0}^{M-1}\int_{\substack{\rs_i\geq \nHH_i, 0\leq i\leq k-1;\\ \rs_k<\nHH_k}}\mathrm{e}^{\rs_k\eta+\nXX_k\eta^2-\rs_0\xi}\cdot\prod_{i=0}^{k-1}\mathrm{p}_{\nXX_{i}-\nXX_{i+1}}(\mathrm{s}_{i+1}-\rs_{i})\diff \rs_0\cdots\diff \rs_{k}.
        \end{aligned}
    \end{equation}                  
    Now we fix $\xi\in \mathbb{C}_\LL, \eta\in \mathbb{C}_\RR$, and consider \eqref{eq: rescaled_kernel}. Use the following scaling and recall that $\nXX_0=\nHH_0=0$,
    \begin{equation}
    \label{eq: scaling}
        \st_i:=-\lfloor\varepsilon^{-1}\nXX_i\rfloor-\lfloor\frac12\varepsilon^{-\frac{1}{2}}\nHH_i\rfloor+1,\quad z_i:=2\lfloor\varepsilon^{-1}\nXX_i\rfloor-\lfloor\varepsilon^{-\frac{1}{2}}(\mathrm{s}_i-\nHH_i)\rfloor ,\quad y_{\st_i}:=2\lfloor\varepsilon^{-1}\nXX_i\rfloor.
    \end{equation}  We write \eqref{eq: rescaled_kernel} as a multiple Riemann sum
    \begin{equation}
    \label{eq:riemann_sum}
        \begin{split}
            &\frac{1}{2}\varepsilon^{\frac{1}{2}}\cdot\Kess_{Y^{\varepsilon}}\left(-\frac{1}{2}+\frac{1}{2}\varepsilon^{\frac{1}{2}}\eta,-\frac{1}{2}+\frac{1}{2}\varepsilon^{\frac{1}{2}}\xi\right)\\
            &=\sum_{k=0}^{M-1}\int_{\substack{\rs_i\geq \nHH_i, 0\leq i\leq k-1;\\ \rs_k<\nHH_k}} \frac{(1+\varepsilon^{\frac{1}{2}}\xi)^{z_0}}{(1+\varepsilon^{\frac{1}{2}}\eta)^{z_k+2\st_k+1}}\cdot (1-\varepsilon\eta^2)^{\st_k}\cdot \prod_{i=0}^{k-1} \varepsilon^{-\frac12}p_{\st_{i+1}-\st_{i}}(z_{i+1}-z_{i})
            \diff \rs_0\cdots\diff \rs_{k}.
        \end{split}
    \end{equation}
    Note that when $\rs_0,\ldots,\rs_k$ are all fixed, the factors in the integrand all converge as $\varepsilon\to 0$:
    \begin{equation}
      \begin{split}
          (1+\varepsilon^{\frac{1}{2}}\xi)^{z_0}&\to \mathrm{e}^{-\rs_0\xi},\\
          (1+\varepsilon^{\frac{1}{2}}\eta)^{z_k+2\st_k+1}&\to \mathrm{e}^{-\rs_k\eta},\\
          (1-\varepsilon\eta^2)^{\st_k}&\to \mathrm{e}^{\nXX_k \eta^2},\\
          \varepsilon^{-\frac12}p_{\st_{i+1}-\st_{i}}(z_{i+1}-z_{i}) &\to \mathrm{p}_{\nXX_{i}-\nXX_{i+1}}(\rs_{i+1}-\rs_{i}),
      \end{split}  
    \end{equation}
    where the last convergence follows from the local central limit theorem or a direct computation using the formulas. Thus, we formally obtain that the limit of \eqref{eq:riemann_sum} is equal to \eqref{eq:chi_integral}, therefore \eqref{eq:pointwise_convergence_chi} follows.

    In order to rigorously show the above convergence, we need to show that \eqref{eq:riemann_sum} is uniformly bounded and the dominated convergence theorem applies. Note that the right-hand side of \eqref{eq:riemann_sum} is the same as that of \eqref{eq: rescaled_kernel_uniform_bound}. Therefore, \eqref{eq:uniform_bound_quadratic}  gives a uniform bound for \eqref{eq:riemann_sum}. This completes the proof.
\end{proof}

    \subsection{Convergence of the series expansion}
    In this section we prove Lemma \ref{lm: convergence_term} and \ref{lm: bound_term}. We will make the additional assumption that $0<\tau_1<\cdots<\tau_m$ to make the presentation lighter. The convergence results and arguments in this section still work if $\tau_i=\tau_{i+1}$ for some $i$ but the contours need to be chosen carefully to make sure that the integrand has the desired super-exponential decay. Alternatively, one can directly work with the limiting KPZ fixed point formula  which is continuous with respect to the limit $\tau_{i+1}\to \tau_i$ and our choices of the angles in Figure \ref{fig:contours_limit} guarantees the convergence of the formula \eqref{eq: Dn_lim} even when some time parameters are equal.  
    
    We will deform the $u,v$ contours so that locally near the critical point $-\frac{1}{2}$, they look like the limiting contours for $\xi,\eta$. More concretely, let $\Gamma_\LL$ be a contour in the left half-plane going from $\infty \mathrm{e}^{-2\pi\mathrm{i}/3}$ to $\infty\mathrm{e}^{2\pi\mathrm{i}/3}$ and $\Gamma_\RR$ be a contour in the right half-plane going from $\infty \mathrm{e}^{-\pi\mathrm{i}/5}$ to $\infty\mathrm{e}^{\pi\mathrm{i}/5}$ (see Figure \ref{fig:contours_limit}).
    For each $\varepsilon>0$, we deform the $u$-contour $\Sigma_L$ and $v$-contour $\Sigma_\RR$ (see Figure \ref{fig:contours_finite_time}) so that the corresponding contours $\Gamma^\varepsilon_\LL$ and $\Gamma_\RR^\varepsilon$ for the rescaled variables $\xi:= 2\varepsilon^{-\frac{1}{2}}\left(u+\frac{1}{2}\right)$ and $\eta:= 2\varepsilon^{-\frac{1}{2}}\left(v+\frac{1}{2}\right)$ satisfy     
    \begin{equation}\label{eq: split_contours}
        \begin{aligned}
        \widehat{\Sigma}_\LL\subset \widehat{\Omega}^\varepsilon_\LL,\quad \Gamma_\LL^\varepsilon\cap\{\zeta\in \mathbb{C}: |\zeta|\leq \frac{1}{100}\varepsilon^{-\frac{1}{2}}\} =\Gamma_\LL \cap\{\zeta\in \mathbb{C}: |\zeta|\leq \frac{1}{100}\varepsilon^{-\frac{1}{2}}\},\\
        \widehat{\Sigma}_\RR\subset \widehat{\Omega}_{\RR}^\varepsilon,\quad \Gamma^\varepsilon_\RR\cap\{\zeta\in \mathbb{C}: |\zeta|\leq \frac{1}{100}\varepsilon^{-\frac{1}{2}}\} =\Gamma_\RR \cap\{\zeta\in \mathbb{C}: |\zeta|\leq \frac{1}{100}\varepsilon^{-\frac{1}{2}}\},
        \end{aligned}
    \end{equation}
    where 
\begin{equation}
    \widehat{\Omega}_\LL^\varepsilon := \{-\varepsilon^{-\frac{1}{2}}+ \varepsilon^{-\frac{1}{2}}z: |z|<1 \},\quad\widehat{\Omega}_\RR^\varepsilon := \{\varepsilon^{-\frac{1}{2}}+ \varepsilon^{-\frac{1}{2}}z: |z|<1 \}.
\end{equation}
See Figure \ref{fig:contours_deformed} for an illustration of the deformed contours. 
 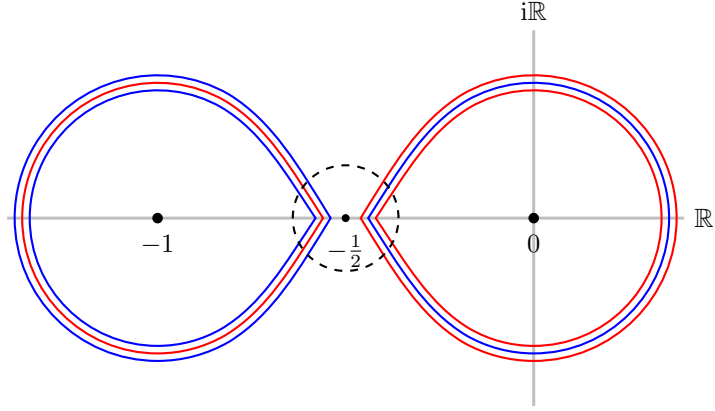
\begin{figure}[t]
 	\centering
 	\begin{tikzpicture}[scale=1]
 	\draw [line width=0.4mm,lightgray] (-2,0)--(7,0) node [pos=1,right,black] {$\realR$};
 	\draw [line width=0.4mm,lightgray] (5,-2.5)--(5,2.5) node [pos=1,above,black] {$\mathrm{i}\realR$};
 	\fill (2.5,0) circle[radius=1.5pt] node [below,shift={(0pt,-3pt)}] {$-\frac{1}{2}$};
    \fill (5,0) circle[radius=2pt] node [below,shift={(0pt,-3pt)}] {$0$};
    \fill (0,0) circle[radius=2pt] node [below,shift={(0pt,-3pt)}] {$-1$};
 \draw[red,thick] 
     (180:1.8cm) 
     arc (180:300:1.8cm)  
    .. controls (320:1.9cm) and (340:1.9cm) .. (0:2.2cm) 
    .. controls (20:1.9cm) and (40:1.9cm) .. (60:1.8cm)
    arc (60:180:1.8cm);
    \draw[blue, thick] 
     (180:1.9cm) 
     arc (180:300:1.9cm)  
    .. controls (320:2cm) and (340:2cm) .. (0:2.3cm) 
    .. controls (20:2cm) and (40:2cm) .. (60:1.9cm)
    arc (60:180:1.9cm);
    \draw[blue, thick] 
     (180:1.7cm) 
     arc (180:300:1.7cm)  
    .. controls (320:1.8cm) and (340:1.8cm) .. (0:2.1cm) 
    .. controls (20:1.8cm) and (40:1.8cm) .. (60:1.7cm)
    arc (60:180:1.7cm);
    
    \begin{scope}[transform canvas={xscale=-1, xshift=-5cm}]
         \draw[blue,thick] 
     (180:1.8cm) 
     arc (180:300:1.8cm)  
    .. controls (320:1.9cm) and (340:1.9cm) .. (0:2.2cm) 
    .. controls (20:1.9cm) and (40:1.9cm) .. (60:1.8cm)
    arc (60:180:1.8cm);
    \draw[red, thick] 
     (180:1.9cm) 
     arc (180:300:1.9cm)  
    .. controls (320:2cm) and (340:2cm) .. (0:2.3cm) 
    .. controls (20:2cm) and (40:2cm) .. (60:1.9cm)
    arc (60:180:1.9cm);
    \draw[red, thick] 
     (180:1.7cm) 
     arc (180:300:1.7cm)  
    .. controls (320:1.8cm) and (340:1.8cm) .. (0:2.1cm) 
    .. controls (20:1.8cm) and (40:1.8cm) .. (60:1.7cm)
    arc (60:180:1.7cm);
    \end{scope}
        \draw[dashed, thick] (2.5,0) circle (20pt);           
		\end{tikzpicture}
 	\caption{The deformed $u,v$-contours. Inside the dashed circle is the region $\{w\in \mathbb{C}: \left|w+\frac{1}{2}\right|<\frac{1}{200}\}$, where we deform the contours to be the same as their limiting counterparts shown in Figure \ref{fig:contours_limit}.}\label{fig:contours_deformed}
    \end{figure}
Note that 
\begin{equation*}
    u\in \Omega_\LL\iff \xi \in  \widehat{\Omega}_\LL^\varepsilon,\quad v\in \Omega_\RR\iff\eta \in  \widehat{\Omega}_\RR^\varepsilon.
\end{equation*}
    Recall the functions $f_i$ defined in \eqref{eq: f_i}. We now introduce some rescaled versions of them. For each $1\leq i\leq m$, $\zeta\in \mathbb{C}_\LL\cup \mathbb{C}_\RR$ and $\varepsilon>0$ sufficiently small, define
    \begin{equation}\label{eq: fi_scaled}
        \mathrm{f}^\varepsilon_i(\zeta): = (1-\varepsilon^{\frac{1}{2}}\zeta)^{k_i^\varepsilon-k_{i-1}^\varepsilon}(1+\varepsilon^{\frac{1}{2}}\zeta)^{(k_{i-1}^\varepsilon-k_{i}^\varepsilon)+(a_{i-1}^\varepsilon-a_{i}^\varepsilon)}\mathrm{e}^{\frac{1}{2}(t_i^\varepsilon-t_{i-1}^\varepsilon)\varepsilon^{\frac{1}{2}}\zeta}.
    \end{equation} Here 
    \begin{equation} \label{eq: parameters_scaled}
    k_\ell^\varepsilon :=\frac{1}{2}\varepsilon^{-\frac{3}{2}}\TT_\ell-\varepsilon^{-1}\XX_\ell-\frac{1}{2}\varepsilon^{-\frac{1}{2}}\HH_\ell+O(1),\quad  a^{\varepsilon}_\ell := 2 \varepsilon^{-1}\XX_\ell+O(1),\quad t^{\varepsilon}_{\ell}:= 2\varepsilon^{-\frac{3}{2}}\TT_\ell,\quad \text{for }1\leq \ell\leq m,
    \end{equation}
    with the convention that $k_0^\varepsilon=a_0^\varepsilon=t_0^\varepsilon:=0$. It is straightforward to check that 
    \begin{equation*}
        \frac{\mathrm{f}^\varepsilon_i(\xi)}{\mathrm{f}^\varepsilon_i(\eta)} = \frac{f_i\left(-\frac{1}{2}+\frac{1}{2}\varepsilon^{\frac{1}{2}}\xi\right)}{f_i\left(-\frac{1}{2}+\frac{1}{2}\varepsilon^{\frac{1}{2}}\eta\right)},\quad \forall \xi\in \mathbb{C}_\LL, \eta\in \mathbb{C}_\RR,
    \end{equation*}
    where $f_i$'s are defined in \eqref{eq: f_i} with the parameters chosen as in \eqref{eq: parameters_scaled}. We begin by stating the needed estimates and asymptotics for the functions $ \mathrm{f}^\varepsilon_i$. 
    \begin{lm} \label{lm: function_converge}
    Assume $0<\tau_1<\cdots<\tau_m$. Let $\mathrm{f}_i^\varepsilon$ be defined as in \eqref{eq: fi_scaled} for $1\leq i\leq m$. The following holds.  
        \begin{enumerate}[(a)] 
            \item For any $\zeta\in \mathbb{C}_\LL\cup\mathbb{C}_\RR$ fixed,  we have 
            \begin{equation}
                \lim_{\varepsilon\to 0} \mathrm{f}^\varepsilon_i(\zeta) = \mathrm{f}_i(\zeta)=:\exp\left(-\frac{1}{3}(\TT_i-\TT_{i-1})\zeta^3+ (\XX_i-\XX_{i-1})\zeta^2+(\HH_i-\HH_{i-1})\zeta\right),\quad 1\leq i\leq m. 
            \end{equation}
            \item There exists constants $c,C>0$ such that 
            \begin{equation}
                |\mathrm{f}_i^\varepsilon(\xi)|\leq C\mathrm{e}^{-c(\TT_i-\TT_{i-1})|\xi|^3},\, \forall \xi\in \Gamma_\LL^\varepsilon,\quad \text{and}\quad
                |\mathrm{f}_i^\varepsilon(\eta)|\geq C^{-1}\mathrm{e}^{c(\TT_i-\TT_{i-1})|\eta|^3},\, \forall \eta\in \Gamma_\RR^\varepsilon.
            \end{equation}
        \end{enumerate}
    \end{lm}
    \begin{proof}
        To lighten the notation we temporarily denote $\TT=\TT_i-\TT_{i-1}$, $\XX=\XX_{i}-\XX_{i-1}$ and $\HH=\HH_i-\HH_{i-1}$. Write
        \begin{align*}
            \mathrm{f}_i^\varepsilon(\zeta) = \exp\left(\varepsilon^{-\frac{3}{2}}\TT\cdot g_3(\zeta) +\varepsilon^{-1}\XX\cdot g_2(\zeta)+\varepsilon^{-\frac{1}{2}}\HH\cdot g_1(\zeta) \right), 
        \end{align*}
        where 
        \begin{equation}\label{eq: split_exponent}
        \begin{aligned}
            &g_1(\zeta):= -\frac{1}{2}\log(1-\varepsilon^{\frac{1}{2}}\zeta)+ \frac{1}{2}\log(1+\varepsilon^{\frac{1}{2}}\zeta),\\
            &g_2(\zeta):= -\log(1-\varepsilon^{\frac{1}{2}}\zeta)-\log(1+\varepsilon^{\frac{1}{2}}\zeta),\\
            &g_3(\zeta):=\frac{1}{2}\log(1-\varepsilon^{\frac{1}{2}}\zeta)- \frac{1}{2}\log(1+\varepsilon^{\frac{1}{2}}\zeta)+\varepsilon^{\frac{1}{2}}\zeta.
        \end{aligned}
        \end{equation}
        Part (a) follows from a straightforward Taylor expansion of \eqref{eq: split_exponent}.  For part (b) we will assume $\zeta=\eta\in \Gamma_\RR^\varepsilon$, the other case is similar. We split into two cases depending on whether $|\eta|\leq \frac{1}{100}\varepsilon^{-\frac{1}{2}}$ or not. Using the elementary bound 
        \begin{equation*}
            \left|\log(1-z)+\sum_{k=1}^{n-1} \frac{z^k}{k}\right|\leq \frac{|z|^n}{n(1-|z|)},\quad \forall |z|<1,
        \end{equation*}
        we have for $|\eta|\leq \frac{1}{100}\varepsilon^{-\frac{1}{2}}$:
        \begin{equation*}
            \left|g_3(\eta)+\frac{1}{3}\varepsilon^{\frac{3}{2}}\eta^3\right|\leq \varepsilon^2|\eta|^4\leq \frac{\varepsilon^{\frac{3}{2}}}{100}|\eta|^3.
        \end{equation*}
        Thus, 
        \begin{equation*}
            \mathrm{Re}\left(g_3(\eta)\right)\geq \mathrm{Re}\left(-\frac{1}{3}\varepsilon^{\frac{3}{2}}\eta^3\right) -\left|g_3(\eta)+\frac{1}{3}\varepsilon^{\frac{3}{2}}\eta^3\right|\geq -\frac{1}{3}\varepsilon^{\frac{3}{2}}\mathrm{Re}(\eta^3) - \frac{1}{100}\varepsilon^{\frac{3}{2}}|\eta|^3\geq c'\varepsilon^{\frac{3}{2}}|\eta|^3,
        \end{equation*}
        for some $c'>0$, due to our choice of the contour $\Gamma_\RR$. Similar argument shows that 
        \begin{equation*}
            |g_2(\eta)|\leq C'\varepsilon|\eta|^2,\quad |g_1(\eta)|\leq C'\varepsilon^{\frac{1}{2}}|\eta|,
        \end{equation*}
        for some $C'>0$. Thus for $|\eta|\leq \frac{1}{100}\varepsilon^{-\frac{1}{2}}$, we have 
        \begin{equation*}
            \begin{aligned}
            |\mathrm{f}_i^\varepsilon(\eta)|&= \exp\left(\varepsilon^{-\frac{3}{2}}\TT\cdot \mathrm{Re}(g_3(\eta)) +\varepsilon^{-1}\XX\cdot \mathrm{Re}(g_2(\eta))+\varepsilon^{-\frac{1}{2}}\HH\cdot \mathrm{Re}(g_3(\eta))\right)\\
            &\geq \exp(c\TT |\eta|^3-C'|\XX||\eta|^2-C'|\HH||\eta|)\geq C\exp(c'\tt |\eta|^3),
            \end{aligned}
        \end{equation*}
        for some constants $c,c',C, C'>0$. On the other hand, it is elementary to check that for $\eta\in \widehat{\Omega}_\RR^\varepsilon\backslash\{|\eta|\leq \frac{1}{100}\varepsilon^{-\frac{1}{2}}\}$, we have 
        \begin{equation*}
            \mathrm{Re}(g_3(\eta))\geq c_3>0,\quad |g_2(\eta))|\leq C_2,\quad |g_1(\eta)|\leq C_1.
        \end{equation*}
        Thus, for such $\eta$ we have 
        \begin{equation*}
            |\mathrm{f}_i^\varepsilon(\eta)|\geq \exp(c_3\varepsilon^{-\frac{3}{2}}-C_2\varepsilon^{-1}-C_1\varepsilon^{-\frac{1}{2}})\geq \exp(c_3'\varepsilon^{-\frac{3}{2}})\geq \exp(c_3'|\eta|^3/8),
        \end{equation*}
        since in this region $\frac{1}{100}\varepsilon^{-\frac{1}{2}}\leq |\eta|\leq 2\varepsilon^{-\frac{1}{2}}$. This completes the proof of part (b) and the lemma.
    \end{proof}

    \subsubsection{Proof of Lemma \ref{lm: convergence_term}}\label{sec: proof_convergence}
        Introduce the change of variables 
        \begin{equation}\label{eq: changevariable}
            u_{i_\ell}^{(\ell)} = -\frac{1}{2}+\frac{1}{2}\varepsilon^{\frac{1}{2}}\xi_{i_\ell}^{(\ell)},\quad v_{i_\ell}^{(\ell)} = -\frac{1}{2}+\frac{1}{2}\varepsilon^{\frac{1}{2}}\eta_{i_\ell}^{(\ell)},
        \end{equation}  
        for $1\leq \ell\leq m$ and $1\leq i_\ell\leq n_\ell$. It is easy to check that under this change of variables we have 
        \begin{align}
            &\mathrm{C}\left(U^{(\ell)}\sqcup V^{(\ell+1)}; V^{(\ell)}\sqcup U^{(\ell+1)}\right) = \left(\frac{2}{\varepsilon}\right)^{\frac{n_{\ell}+n_{\ell+1}}{2}} \mathrm{C}\left(\xib^{(\ell)}\sqcup \etab^{(\ell+1)}; \etab^{(\ell)}\sqcup \xib^{(\ell+1)}\right),\\
            &\mathrm{C}\left(U^{(m)}; V^{(m)}\right) = \left(\frac{2}{\varepsilon}\right)^{\frac{n_{m}}{2}} \mathrm{C}\left(\xib^{(m)}; \etab^{(m)}\right),
        \end{align}
        for $1\leq \ell\leq m-1$. As in \eqref{eq: split_contours}, we split the $\xi,\eta$ contours depending on whether they lie in the region $\Omega_0^\varepsilon:= \{w\in \mathbb{C}: |w|\leq \frac{1}{100}\varepsilon^{-\frac{1}{2}}\}$ or not, and deform the contours so that $\Gamma^\varepsilon_{\LL/\RR}$ agree with $\Gamma_{\LL/\RR}$ inside $\Omega_0^\varepsilon$. Recall the definition of $\mathcal{D}_{Y^\varepsilon}^{(\mathbf{n})}$ in \eqref{eq:D_nY} and $\mathrm{D}_{\h}^{(\mathbf{n})}$ in \eqref{eq: Dn_lim}. Let $A$ be the event that all the variables $\xi_{i_\ell}^{(\ell)}, \eta_{i_\ell}^{(\ell)}$ lie inside $\Omega_0^\varepsilon$ for $1\leq \ell\leq m$ and $1\leq i_\ell\leq n_\ell$, and $A^c$ be the complement. Define
        \begin{equation}\label{eq: DnY_main}
            \begin{split}
            \mathcal{D}_{Y^\varepsilon}^{(\mathbf{n}, \mathrm{main})}&:= \prod_{\ell=1}^{m-1} \left(1-z_\ell\right)^{n_\ell}\left(1-z_\ell^{-1}\right)^{n_{\ell+1}}\left(\prod_{\ell=1}^m\prod_{i_\ell=1}^{n_\ell} \int_{\Gamma_{\ell,\LL}}\mathrm{d}\mu_{\boldsymbol{z}}(\xi_{i_\ell}^{(\ell)})\int_{\Gamma_{\ell,\RR}}\mathrm{d}\mu_{\boldsymbol{z}}(\eta_{i_\ell}^{(\ell)})\right) \mathbf{1}_A\\
            &\hspace{0.5cm}\left(\prod_{\ell=1}^m\prod_{i_\ell=1}^{n_\ell} \frac{\mathrm{f}^\varepsilon_\ell(\xi_{i_\ell}^{(\ell)})}{ \mathrm{f}^\varepsilon_\ell(\eta_{i_\ell}^{(\ell)})}\right)\cdot \det\begin{bmatrix}
	    \frac{1}{2}\varepsilon^{\frac{1}{2}}\Kess_{Y^\varepsilon}\left(-\frac{1}{2}+\frac{1}{2}\varepsilon^{\frac{1}{2}}\eta^{(1)}_j,-\frac{1}{2}+\frac{1}{2}\varepsilon^{\frac{1}{2}}\xi^{(1)}_i\right)
	\end{bmatrix}_{1\leq i,j\leq n_1}\\
    &\quad \cdot \prod_{\ell=1}^{m-1}\mathrm{C}\left(\xib^{(\ell)}\sqcup \etab^{(\ell+1)}; \etab^{(\ell)}\sqcup \xib^{(\ell+1)}\right)\cdot \mathrm{C}(\xib^{(m)};\etab^{(m)}),
            \end{split}
        \end{equation}
        and $\mathcal{D}_{Y^\varepsilon}^{(\mathbf{n}, \mathrm{error})}:= \mathcal{D}_{Y^\varepsilon}^{(\mathbf{n})}-\mathcal{D}_{Y^\varepsilon}^{(\mathbf{n}, \mathrm{main})}$. Note that 
        \begin{equation}\label{eq: DnY_error}
            \begin{aligned}
                \mathcal{D}_{Y^\varepsilon}^{(\mathbf{n}, \mathrm{error})}&= \prod_{\ell=1}^{m-1} \left(1-z_\ell\right)^{n_\ell}\left(1-z_\ell^{-1}\right)^{n_{\ell+1}}\left(\prod_{\ell=1}^m\prod_{i_\ell=1}^{n_\ell} \int_{\Gamma^\varepsilon_{\ell,\LL}}\mathrm{d}\mu_{\boldsymbol{z}}(\xi_{i_\ell}^{(\ell)})\int_{\Gamma^\varepsilon_{\ell,\RR}}\mathrm{d}\mu_{\boldsymbol{z}}(\eta_{i_\ell}^{(\ell)})\right)\mathbf{1}_{A^c}\\         &\hspace{0.5cm}\left(\prod_{\ell=1}^m\prod_{i_\ell=1}^{n_\ell} \frac{\mathrm{f}^\varepsilon_\ell(\xi_{i_\ell}^{(\ell)})}{ \mathrm{f}^\varepsilon_\ell(\eta_{i_\ell}^{(\ell)})}\right)\cdot\det\begin{bmatrix}
	    \frac{1}{2}\varepsilon^{\frac{1}{2}}\Kess_{Y^\varepsilon}\left(-\frac{1}{2}+\frac{1}{2}\varepsilon^{\frac{1}{2}}\eta^{(1)}_j,-\frac{1}{2}+\frac{1}{2}\varepsilon^{\frac{1}{2}}\xi^{(1)}_i\right)
	\end{bmatrix}_{1\leq i,j\leq n_1}\\
    &\quad \cdot \prod_{\ell=1}^{m-1}\mathrm{C}\left(\xib^{(\ell)}\sqcup \etab^{(\ell+1)}; \etab^{(\ell)}\sqcup \xib^{(\ell+1)}\right)\cdot \mathrm{C}(\xib^{(m)};\etab^{(m)}).
            \end{aligned}
        \end{equation}
        We claim that:
        \begin{equation}
            \lim_{\varepsilon\to 0}\mathcal{D}_{Y^\varepsilon}^{(\mathbf{n},\mathrm{main})} = \mathrm{D}_{\h}^{(\mathbf{n})},\quad \lim_{\varepsilon\to 0}\mathcal{D}_{Y^\varepsilon}^{(\mathbf{n},\mathrm{error})} = 0.
        \end{equation}
        For the first part of the claim, note that the integrand on the right-hand side of \eqref{eq: DnY_main} converges pointwise to the integrand on the right-hand side of  \eqref{eq: Dn_lim} by Lemma \ref{lm: function_converge}(a) and Proposition \ref{prop: convergence_nw}(a). On the other hand, the cubic exponential decay bound for $|\mathrm{f}_i^\varepsilon|$ from Lemma \ref{lm: function_converge}(b) and the quadratic exponential growth estimate \eqref{eq:uniform_bound_quadratic} implies that the integrand on the right-hand side of \eqref{eq: DnY_main} has a cubic exponential decay in every variable $\eta_{i_\ell}^{(\ell)},\xi_{i_\ell}^{(\ell)}$ as they go to $\infty$, uniform in $\varepsilon$. Thus, the dominated convergence theorem applies, and the first claim is proved. 

        For the second part of the claim, note that on $A^c$, at least one of the variables, say $\xi_{1}^{(1)}$, lies outside $\widehat{\Omega}_0^\varepsilon$. Then Lemma \ref{lm: function_converge}(b) implies that $|\mathrm{f}_1^{\varepsilon}(\xi_1^{(1)})|\leq C\exp(-c\varepsilon^{-\frac{3}{2}})$ for all $\xi_1^{(1)}\in {\Gamma}_{1,\LL}^\varepsilon\backslash \widehat{\Omega}_0^\varepsilon$.  On the other hand, Proposition \ref{prop: convergence_nw}(b) implies that 
        \begin{equation}
            \left|\frac{1}{2}\varepsilon^{\frac{1}{2}}\Kess_{Y^{\varepsilon}}\left(-\frac{1}{2}+\frac{1}{2}\varepsilon^{\frac{1}{2}}\eta_j^{(1)},-\frac{1}{2}+\frac{1}{2}\varepsilon^{\frac{1}{2}}\xi_1^{(1)}\right)\right|\leq C\exp(c\varepsilon^{-1}),
        \end{equation}
    for some $c,C>0$ and all $1\leq j\leq n_1$. The other parts of the integrand remain bounded. Thus,
    \begin{equation}
        |\mathcal{D}_{Y^\varepsilon}^{(\mathbf{n},\mathrm{error})}|\leq C\varepsilon^{-(n_1+\cdots+n_m)}\cdot \exp(-c\varepsilon^{-\frac{3}{2}})\to 0,\quad \text{as }\varepsilon\to 0.
    \end{equation}
This proves the second part of the claim, and Lemma \ref{lm: convergence_term} follows.

    \subsubsection{Proof of Lemma \ref{lm: bound_term}}\label{sec: proof_bound}
    Recall the expression \eqref{eq:D_nY} for $\mathcal{D}_{Y^\varepsilon}^{(\mathbf{n})}$. We first rewrite it slightly by writing $\mathrm{f}_i^\varepsilon (\zeta)= \mathrm{f}_i^\varepsilon (\zeta)^{\frac{1}{2}}\cdot \mathrm{f}_i^\varepsilon (\zeta)^{\frac{1}{2}}$ and putting one of the square root inside the first determinant, to get 
    \begin{equation}
            \begin{aligned}
                \mathcal{D}_{Y^\varepsilon}^{(\mathbf{n})}&= \prod_{\ell=1}^{m-1} \left(1-z_\ell\right)^{n_\ell}\left(1-z_\ell^{-1}\right)^{n_{\ell+1}}\left(\prod_{\ell=1}^m\prod_{i_\ell=1}^{n_\ell} \int_{\Gamma^\varepsilon_{\ell,\LL}}\mathrm{d}\mu_{\boldsymbol{z}}(\xi_{i_\ell}^{(\ell)})\int_{\Gamma^\varepsilon_{\ell,\RR}}\mathrm{d}\mu_{\boldsymbol{z}}(\eta_{i_\ell}^{(\ell)})\right)\prod_{\ell=2}^m\prod_{i_\ell=1}^{n_\ell} \frac{\mathrm{f}^\varepsilon_\ell(\xi_{i_\ell}^{(\ell)})}{ \mathrm{f}^\varepsilon_\ell(\eta_{i_\ell}^{(\ell)})}\\
            &\hspace{0.5cm}\left(\prod_{i_1=1}^{n_1} \frac{\mathrm{f}^\varepsilon_1(\xi_{i_1}^{(1)})^{1/2}}{ \mathrm{f}^\varepsilon_1(\eta_{i_1}^{(1)})^{1/2}}\right)\cdot\det\begin{bmatrix}
	    \frac{\mathrm{f}^\varepsilon_1(\xi_{i}^{(1)})^{1/2}}{ \mathrm{f}^\varepsilon_1(\eta_{j}^{(1)})^{1/2}}\cdot \frac{1}{2}\varepsilon^{\frac{1}{2}}\Kess_{Y^\varepsilon}\left(-\frac{1}{2}+\frac{1}{2}\varepsilon^{\frac{1}{2}}\eta^{(1)}_j,-\frac{1}{2}+\frac{1}{2}\varepsilon^{\frac{1}{2}}\xi^{(1)}_i\right)
	\end{bmatrix}_{1\leq i,j\leq n_1}\\
    &\quad \cdot \prod_{\ell=1}^{m-1}\mathrm{C}\left(\xib^{(\ell)}\sqcup \etab^{(\ell+1)}; \etab^{(\ell)}\sqcup \xib^{(\ell+1)}\right)\cdot \mathrm{C}(\xib^{(m)};\etab^{(m)}),
            \end{aligned}
        \end{equation}
    Here the choice of square root does not matter as long as we make the same choice for the two. The advantage of this rewriting is that 
    \begin{equation}
        \left|\frac{\mathrm{f}^\varepsilon_1(\xi_{i}^{(1)})^{1/2}}{ \mathrm{f}^\varepsilon_1(\eta_{j}^{(1)})^{1/2}}\cdot \frac{1}{2}\varepsilon^{\frac{1}{2}}\Kess_{Y^\varepsilon}\left(-\frac{1}{2}+\frac{1}{2}\varepsilon^{\frac{1}{2}}\eta^{(1)}_j,-\frac{1}{2}+\frac{1}{2}\varepsilon^{\frac{1}{2}}\xi^{(1)}_i\right)\right|\leq c, \quad \forall 1\leq i,j\leq n_1,
    \end{equation}
    for some constant $c>0$ independent of $\varepsilon$ and $\xib^{(1)},\etab^{(1)}$. Thus, by Hadamard's inequality 
    \begin{equation}
        \left|\det\begin{bmatrix}
	    \frac{\mathrm{f}^\varepsilon_1(\xi_{i}^{(1)})^{1/2}}{ \mathrm{f}^\varepsilon_1(\eta_{j}^{(1)})^{1/2}}\cdot \frac{1}{2}\varepsilon^{\frac{1}{2}}\Kess_{Y^\varepsilon}\left(-\frac{1}{2}+\frac{1}{2}\varepsilon^{\frac{1}{2}}\eta^{(1)}_j,-\frac{1}{2}+\frac{1}{2}\varepsilon^{\frac{1}{2}}\xi^{(1)}_i\right)
	\end{bmatrix}_{1\leq i,j\leq n_1}\right|\leq n_1^{\frac{n_1}{2}}\cdot c^{n_1}.
    \end{equation}
    The same arguments imply that 
    \begin{equation}
        |\mathrm{C}(\xib^{(m)};\etab^{(m)})|\leq n_m^{\frac{n_m}{2}}\cdot D^{n_m},
    \end{equation}
    and 
    \begin{equation}
        |\mathrm{C}\left(\xib^{(\ell)}\sqcup \etab^{(\ell+1)}; \etab^{(\ell)}\sqcup \xib^{(\ell+1)}\right)|\leq (n_\ell+n_{\ell+1})^{\frac{n_\ell+n_{\ell+1}}{2}}\cdot D^{\frac{n_\ell+n_{\ell+1}}{2}}\leq n_\ell^{\frac{n_\ell}{2}}n_{\ell+1}^{\frac{n_{\ell+1}}{2}}\cdot (2D)^{\frac{n_\ell+n_{\ell+1}}{2}},
    \end{equation}
    for $1\leq \ell\leq m-1
    $. Here $D$ is chosen to be the reciprocal of the minimal distance between different $\Gamma^\varepsilon$ contours, which can be made positive. Thus, 
    \begin{equation}
        \begin{aligned}
        |\mathcal{D}_{Y^\varepsilon}^{(\mathbf{n})}|&\leq \prod_{\ell=1}^{m-1} \frac{|1-z_\ell|^{n_\ell+n_{\ell+1}}}{|z_\ell|^{n_{\ell+1}}} \cdot (C')^{n_1+\cdots+n_m}\cdot \prod_{\ell=1}^m n_\ell^{n_\ell}\cdot \prod_{i_1=1}^{n_1} \int_{\Gamma_{1,\LL}^\varepsilon}\frac{\mathrm{d}|\xi_{i_1}^{(1)}|}{2\pi}\int_{\Gamma_{1,\RR}^\varepsilon}\frac{\mathrm{d}|\eta_{i_1}^{(1)}|}{2\pi} \frac{|\mathrm{f}^\varepsilon_1(\xi_{i_1}^{(1)})|^{\frac{1}{2}}}{ |\mathrm{f}^\varepsilon_1(\eta_{i_1}^{(1)})|^{\frac{1}{2}}}\\
        &\hspace{0.3cm}\cdot \prod_{\ell=2}^{m}\prod_{i_\ell=1}^{n_\ell} \left(\frac{1+|z_\ell|}{|1-z_\ell|}\right)^2\int_{\Gamma_{\ell,\LL}^\varepsilon}\frac{\mathrm{d}|\xi_{i_\ell}^{(\ell)}|}{2\pi}\int_{\Gamma_{\ell,\RR}^\varepsilon}\frac{\mathrm{d}|\eta_{i_\ell}^{(\ell)}|}{2\pi} \frac{|\mathrm{f}^\varepsilon_\ell(\xi_{i_\ell}^{(\ell)})|}{ |\mathrm{f}^\varepsilon_\ell(\eta_{\ell_1}^{(\ell)})|}\\
        &\leq  \prod_{\ell=1}^{m-1} \frac{(1+|z_{\ell+1}|)^{2n_{\ell+1}}}{|z_{\ell}|^{n_{\ell+1}}|1-z_\ell|^{n_{\ell+1}-n_\ell}}\cdot\prod_{\ell=1}^m n_\ell^{n_\ell}\cdot C^{n_1+\cdots+n_m},
        \end{aligned}
    \end{equation}
      for some constant $C>0$, since each of the integrals is bounded by some finite constant. This completes the proof of Lemma \ref{lm: bound_term}.

\section{From multiple narrow wedges to compactly supported initial conditions}\label{sec: compact support}
In this section we extend the multipoint formula \eqref{eq: multi_time} to any $\h\in \ucc$ using a density argument. We choose to work at the level of the KPZ fixed point formula, which enjoys more symmetry and nicer decay properties. 

\subsection{Density and Approximation} Our starting point is the following proposition, asserting that our formula is continuous with respect to the initial condition on the space $\ucc$.

\begin{prop}\label{prop: limess_convergence}
    Given $\{\h^n\}_{n\geq 1}\subset \uc$ with $\mathrm{supp}(\h^n)\subset [-L,L]$ and $\sup_{\XX\in \mathbb{R}}\h^n(\XX)\leq \beta$ for all $n\geq 1$. Assume $\h^n\to \h$ in $\uc$ and $\mathrm{supp}(\h)\subset [-L,L],\ \sup_{\XX\in \mathbb{R}}\h(\XX)\leq \beta$. Then
    \begin{equation}\label{eq: ess_converge1}
        \lim_{n\to \infty}\limess_{\h^n} (\eta,\xi) = \limess_{\h}(\eta,\xi),
    \end{equation}
    for all $\xi \in \mathbb{C}_\LL$ and $\eta\in \mathbb{C}_\RR$. Recall the kernel $\limess_{\h}$ was defined in Definition \ref{def: limess}, equation \eqref{eq: limess}. Consequently, 
    \begin{equation}\label{eq: series_converge1}
        \lim_{n\to \infty}  \prod_{\ell=1}^{m-1} \oint_0 \frac{\diff z_\ell}{2\pi\ii z_\ell(1-z_\ell)}\mathrm{D}_{\h^n}(z_1,\ldots,z_{m-1}) = \prod_{\ell=1}^{m-1} \oint_0 \frac{\diff z_\ell}{2\pi\ii z_\ell(1-z_\ell)}\mathrm{D}_{\h}(z_1,\ldots,z_{m-1}).
    \end{equation}
\end{prop}
\begin{proof}
    It suffices to show that 
    \begin{equation*}
        \lim_{n\to \infty}\int_\mathbb{R} \mathrm{e}^{s\eta}\cdot \mathbb{E}_{\mathbf{B}(0)=s} \left[\exp(-\ta_{n}\xi^2-\mathbf{B}(\ta_n)\xi)\mathbf{1}_{\ta_n<\infty}\right] = \int_\mathbb{R} \mathrm{e}^{s\eta}\cdot \mathbb{E}_{\mathbf{B}(0)=s} \left[\exp(-\ta\xi^2-\mathbf{B}(\ta)\xi)\mathbf{1}_{\ta<\infty}\right],
    \end{equation*}
    where 
    \begin{equation*}
        \ta:= \inf\{\XX\geq 0: \mathbf{B}(\XX)\leq \h(\XX)\},\quad \ta_n:= \inf\{\XX\geq 0: \mathbf{B}(\XX)\leq \h^n(\XX)\},
    \end{equation*}
    for a Brownian motion $\mathbf{B}(\XX)$ with diffusivity constant $2$. Here we have suppressed the $+$ sign in the subscript of the hitting time to lighten the notation. The convergence of the other two parts in the definition of $\limess_{\h}$ can be proved in the same way. From \cite{matetski2021kpz}[(B.20)] we know 
    \begin{equation*}
        \mathbb{P}_{\mathbf{B}(0)=s}(\mathbf{B}(\ta_n)\in \mathrm{d}b, \ta_n\in\mathrm{d}T)\to \mathbb{P}_{\mathbf{B}(0)=s}(\mathbf{B}(\ta)\in \mathrm{d}b, \ta\in\mathrm{d}T) \quad \text{weakly as } n\to \infty.
    \end{equation*}
    Take a smooth function $0\leq g_\beta\leq 1$ such that $g_\beta(x)\equiv 1$ for $x\leq \beta$ and $g_\beta(x)\equiv 0$ for $x\geq \beta+1$. Since $\h(\XX)\leq \beta$ and $\h^n(\XX)\leq \beta$ for all $\XX\in \mathbb{R}$ and $n\geq 1$, we have 
    \begin{equation*}
        \mathbf{1}_{\ta_n<\infty}(1-g_\beta(\mathbf{B}(\ta_n))) \equiv 0,\quad \mathbf{1}_{\ta<\infty}(1-g_\beta(\mathbf{B}(\ta))) \equiv 0.
    \end{equation*}
    Hence, 
    \begin{align*}
       &\mathbb{E}_{\mathbf{B}(0)=s} \left[\exp(-\ta_{n}\xi^2-\mathbf{B}(\ta_n)\xi)\mathbf{1}_{\ta_n<\infty}\right]= \mathbb{E}_{\mathbf{B}(0)=s} \left[\exp(-\ta_{n}\xi^2-\mathbf{B}(\ta_n)\xi)\mathbf{1}_{\ta_n<\infty} g_\beta(\B_{\ta_n})\right],\\
       &\mathbb{E}_{\mathbf{B}(0)=s} \left[\exp(-\ta\xi^2-\mathbf{B}(\ta)\xi)\mathbf{1}_{\ta<\infty}\right]= \mathbb{E}_{\mathbf{B}(0)=s} \left[\exp(-\ta\xi^2-\mathbf{B}(\ta)\xi)\mathbf{1}_{\ta_n<\infty} g_\beta(\B_{\ta})\right].
    \end{align*}
    By the weak convergence of $(\ta_n, \mathbf{B}(\ta_n))$ to $(\ta,\mathbf{B}(\ta))$ we know 
    \begin{equation*}
        \lim_{n\to \infty} \mathbb{E}_{\mathbf{B}(0)=s} \left[\exp(-\ta_{n}\xi^2-\mathbf{B}(\ta_n)\xi)\mathbf{1}_{\ta_n<\infty}g_\beta(\B_{\ta_n})\right] = \mathbb{E}_{\mathbf{B}(0)=s} \left[\exp(-\ta\xi^2-\mathbf{B}(\ta)\xi)\mathbf{1}_{\ta_n<\infty}g_\beta(\B_{\ta})\right],
    \end{equation*}
    since the function $f(T,b):=\exp(-T\xi^2-b\xi)\mathbf{1}_{T<\infty}g_\beta(b):[0,\infty]\times \mathbb{R}\to \mathbb{C}$ is bounded and continuous. Therefore
    \begin{equation*}
        \lim_{n\to \infty} \mathbb{E}_{\mathbf{B}(0)=s} \left[\exp(-\ta_{n}\xi^2-\mathbf{B}(\ta_n)\xi)\mathbf{1}_{\ta_n<\infty})\right] = \mathbb{E}_{\mathbf{B}(0)=s} \left[\exp(-\ta\xi^2-\mathbf{B}(\ta)\xi)\mathbf{1}_{\ta_n<\infty}\right].
    \end{equation*}
    Finally, using the bound obtained in Proposition \ref{prop:convergence_lim} and the dominated convergence theorem we have 
    \begin{equation*}
        \lim_{n\to \infty}\int_\mathbb{R} \mathrm{e}^{s\eta}\cdot \mathbb{E}_{\mathbf{B}(0)=s} \left[\exp(-\ta_{n}\xi^2-\mathbf{B}(\ta_n)\xi)\mathbf{1}_{\ta_n<\infty}\right] = \int_\mathbb{R} \mathrm{e}^{s\eta}\cdot \mathbb{E}_{\mathbf{B}(0)=s} \left[\exp(-\ta\xi^2-\mathbf{B}(\ta)\xi)\mathbf{1}_{\ta<\infty}\right].
    \end{equation*}
    This completes the proof of \eqref{eq: ess_converge1}. \eqref{eq: series_converge1} follows from the dominated convergence theorem and a similar bound as in Lemma \ref{lm: bound_term} whose proof is almost identical, and we omit it here. 
  \end{proof}  

To  extend \eqref{eq: multi_time} to all $\h\in \ucc$ using Proposition \ref{prop: limess_convergence}, we also need the following proposition, asserting that the space $\mnw$ is dense in $\ucc$.
\begin{prop}\label{prop: approximation}
    Let $\h\in \ucc$. Assume that $\mathrm{supp}(\h)\subset [-L,L]$ and $\max_{\XX\in \mathbb{R}}\h(\XX) = \beta<\infty$. Then there exists a sequence $\{\h^{(n)}\}_{n\geq 1}\subset \mathrm{NW}$, such that $\mathrm{supp}(\h)\subset [-L,L]$, $\max_{\XX\in \mathbb{R}}\h^{n}\leq \beta$ for all $n\geq 1$ and 
    \begin{equation*}
        \h^{n}\to \h\  \text{in }\uc,\quad \text{as }n\to \infty.
    \end{equation*}
\end{prop}
\begin{proof}
    We will use the following characterization (see, e.g., \cite[Section 3.1]{matetski2021kpz}): a sequence $\{\h^n\}_{n\geq 1}\subset \uc$ converges to $\h\in \uc$ locally if and only if  for any $\mathsf{x}\in \mathbb{R}$, one has 
    \begin{enumerate}
        \item $\limsup_{n\to \infty} \h(\mathsf{x}_n)\leq \h(\mathsf{x})$, for all $\mathsf{x}_n\to \mathsf{x}$;
        \item There exists $\mathsf{x}_n\to \mathsf{x}$ such that $\liminf_{n\to \infty}\h(\mathsf{x}_n)\geq \h(\mathsf{x})$.
    \end{enumerate}It suffices to consider the restrictions of the functions on $[-L,L]$. For each $n$ consider the dyadic intervals $I_{n,k}:= [\frac{k}{2^n}\cdot L, \frac{k+1}{2^n}\cdot L]$, for $k=-2^n,-2^n+1,\ldots, 2^n-1$. On each interval $I_{n,k}$ the maximum of $\h$ exists, let $m_{n,k}\in I_{n,k}$ be one of the argmax, and set 
    \begin{equation*}
        \h^{n}(\XX):= \sum_{k=1}^{2^n-1} \h(m_{n,k})\mathbf{1}_{\XX= m_{n,k}} -\infty\cdot \mathbf{1}_{\XX\neq m_{n,k},\,\forall k},
    \end{equation*}
    here if some $m_{n,k}$ is the argmax of two consecutive intervals, then it should appear only once in the sum. Now it is straightforward to check that for each $\XX$
    \begin{equation*}
        \limsup_{\XX^n\to \XX} \h^n(\XX^n)\leq \h(\XX),
    \end{equation*}
    and there exists $\XX^n\to \XX$ such that 
    \begin{equation*}
         \liminf_{\XX^n\to \XX} \h^n(\XX^n)\geq \h(\XX).
    \end{equation*}
    Therefore, $\h^n\to \h$ locally in $\uc$, and hence globally since $\h^n,\h$ are supported in $[-L,L]$. 
\end{proof}
\subsection{Proof of Proposition \ref{prop: shift}}\label{sec: shift}
Recall that what we have shown in Proposition \ref{prop: approximation} is that the formula \eqref{eq: multi_time} holds for initial conditions $\h\in \mnw_0$. We would like to prove it for $\h\in \mnw$ by a shift argument. To this end it is more convenient to have a more general version of the characteristic function $\limess_\h$, defined through \eqref{eq: limess2}, instead of the original \eqref{eq: limess}. The goal of this section is to prove the equivalence of the two, i.e., Proposition \ref{prop: shift}. To this end, we denote 
\begin{equation}\label{eq: limess_a}
    \begin{aligned}
     \limess^\a_{\h}(\eta,\xi) &:= \mathrm{e}^{-\a \xi^2}\int_{\mathbb{R}} \diff s\, \mathrm{e}^{-s\xi}\cdot  \mathbb{E}_{\B(\a)=s}\left[\exp\left(\ta_-^\a \eta^2+\B(\ta_-^\a)\eta\right)\1_{\ta^\a_->-\infty}\right]\\
        &\quad+\mathrm{e}^{\a\eta^2}\int_{\mathbb{R}} \diff s\, \mathrm{e}^{+s\eta}\cdot  \mathbb{E}_{\B(\a)=s}\left[\exp\left(-\ta^\a_+ \xi^2-\B(\ta^\a_+)\xi\right)\1_{\ta^\a_+<\infty}\right]\\
        &\quad-\int_{\mathbb{R}} \diff s\,  \mathbb{E}_{\B(\a)=s}\left[\exp\left(\ta_-^\a \eta^2+\B(\ta^\a_-)\eta-\ta_+^\a \xi^2-\B(\ta^\a_+)\xi\right)\1_{|\ta^\a_\pm|<\infty}\right]\\
        &:=  \limess^{\a,1}_{\h}(\eta,\xi)+ \limess^{\a,2}_{\h}(\eta,\xi)- \limess^{\a,3}_{\h}(\eta,\xi),
    \end{aligned}
\end{equation}
where 
\begin{equation*}
    \ta_+^\a:=\inf\{\XX\geq \a: \B(\XX)\leq \h(\XX)\},\quad \ta^\a_-:=\sup\{\XX\leq \a: \B(\XX)\leq \h(\XX)\}.
\end{equation*}
We claim that $\limess^\a_{\h}=\limess_{\h}^{\nXX_1}$, for any $\a\in \mathbb{R}$. First assume $\h\in \mnw$, say
\begin{equation*}
    \h(\XX)= \sum_{i=1}^k \nHH_i \mathbf{1}_{\XX=\nXX_i} -\infty \mathbf{1}_{\XX\neq \nXX_i,\,\forall i},
\end{equation*}
where $\nXX_1>\cdots>\nXX_k$. To lighten the notation, we will fix $\eta,\xi$ and denote 
\begin{equation*}
\begin{aligned}
    &E^+(\a,s):= \mathbb{E}_{\B(\a)=s}\left[\exp\left(-\ta_+^\a \xi^2-\B(\ta_+^\a)\xi\right)\1_{\ta^\a_+<\infty}\right],\\
    &E^-(\a,s):=\mathbb{E}_{\B(\a)=s}\left[\exp\left(+\ta^\a_- \eta^2+\B(\ta^\a_-)\eta\right)\1_{\ta^\a_->-\infty}\right].
\end{aligned}
\end{equation*}
For $\a\in [\nXX_1,\infty)$, we have 
\begin{align*}
    \limess_{\h}^{\a,1}(\eta,\xi) &= \mathrm{e}^{-\a\xi^2}\int_{\mathbb{R}}\diff s\int_{\mathbb{R}}\diff b_1 \,\mathrm{e}^{-s\xi}\mathrm{p}_{\a-\nXX_1}(b_1-s)\, E^-(\nXX_1,b_1)\\
    &= \mathrm{e}^{-\nXX_1\xi^2} \int_{\mathbb{R}}\diff b_1\, \mathrm{e}^{-b_1\xi}\, E^-(\nXX_1,b_1) = \limess_{\h}^{\nXX_1,1}(\eta,\xi),
\end{align*}
where we are using the Markov property and the simple identity
\begin{equation}\label{eq: MGF_identity}
    \int_\mathbb{R}\,\mathrm{d}s\,\mathrm{e}^{sW}\mathrm{p}_t(r-s) = \mathrm{e}^{tW^2+rW}.
\end{equation}
Recall that $\mathrm{p}_{t-s}(x-y)$ is the transition density of a Brownian motion with diffusivity constant $2$. On the other hand,
\begin{align*}
     \limess_{\h}^{\a,2}(\eta,\xi)=\limess_{\h}^{\a,3}(\eta,\xi) =  \mathbf{1}_{\a=\nXX_1} \frac{\mathrm{e}^{\nXX_1(\eta^2-\xi^2)+\nHH_1(\eta-\xi)}}{\eta-\xi},\quad \forall \a\geq \nXX_1.
\end{align*}
Thus $\limess_{\h}^\a=\limess_{\h}^{\nXX_1}$ for all $\a\geq \nXX_1$. We proceed by induction. Assume for some $1\leq i\leq k$, $\limess_{\h}^\a=\chi_{\h}^{\nXX_1}$ holds for all $\a\geq \nXX_i$. Now let $\a\in [\nXX_{i+1},\nXX_{i})$. We will show $\limess_{\h}^\a = \limess_{\h}^{\nXX_i}$, which is equal to $\limess_{\h}^{\nXX_1}$ by the induction hypothesis. By the same argument as above we have 
\begin{equation}\label{eq: chi_a2}
    \limess_{\h}^{\a,2}(\eta,\xi) = \limess_{\h}^{\nXX_{i},2}(\eta,\xi)+ \mathbf{1}_{\a=\nXX_{i+1}} \frac{\mathrm{e}^{\nXX_{i+1}(\eta^2-\xi^2)+\nHH_{i+1}(\eta-\xi)}}{\eta-\xi}.
\end{equation}
On the other hand, by the Markov property and the identity \eqref{eq: MGF_identity} we have
\begin{equation}\label{eq: chi_a1}
\begin{aligned}
    \limess_{\h}^{\nXX_i,1}(\eta,\xi) &=  \frac{\mathrm{e}^{\nXX_{i}(\eta^2-\xi^2)+\nHH_i(\eta-\xi)}}{\eta-\xi}+\mathrm{e}^{-\nXX_i\xi^2}\int_{\nHH_i}^{\infty}\diff b_i\int_{\mathbb{R}}\diff s \,\mathrm{e}^{-b_i\xi}\mathrm{p}_{\nXX_i-\a}(s-b_i)\, E^-(\a,s)\\
    & = \frac{\mathrm{e}^{\nXX_{i}(\eta^2-\xi^2)+\nHH_i(\eta-\xi)}}{\eta-\xi}+ \limess_{\h}^{\a,1} - \mathrm{e}^{-\nXX_i\xi^2}\int_{-\infty}^{\nHH_i}\diff b_i\int_{\mathbb{R}}\diff s \,\mathrm{e}^{-b_i\xi}\mathrm{p}_{\nXX_i-\a}(s-b_i)\, E^-(\a,s),
\end{aligned}
\end{equation}
where in the second equality we used Fubini's theorem and the identity \eqref{eq: MGF_identity} to show that 
\begin{equation*}
    \mathrm{e}^{-\nXX_i\xi^2}\int_{\mathbb{R}}\diff b_i\int_{\mathbb{R}}\diff s \,\mathrm{e}^{-b_i\xi}\mathrm{p}_{\nXX_i-\a}(s-b_i)\, E^-(\a,s) = \mathrm{e}^{-\a\xi^2}\int_{\mathbb{R}}\diff s\,\mathrm{e}^{-s\xi}\cdot E^-(\a,s) =\limess_{\h}^{\a,1}.
\end{equation*}
Finally, 
\begin{align*}
    \limess_{\h}^{\a,3}(\eta,\xi) &= \int_{\mathbb{R}}\diff s\,E^{-}(\a,s)\cdot E^{+}(\a,s)\\
    &= \int_{\mathbb{R}}\diff s\, E^-(\a,s)\cdot \left(\int_{\mathbb{R}}\diff b_i\,\mathrm{p}_{\nXX_i-\a}(s-b_i)E^{+}(\nXX_i,b_i)+\mathbf{1}_{\a=\nXX_{i+1}}\cdot \mathrm{e}^{-\nXX_{i+1}\xi^2-s\xi}\cdot \mathbf{1}_{s\leq \nHH_{i+1}}\right).
\end{align*}
We split the $b_i$-integral into two parts depending on whether $b_i\leq \nHH_i$ or $b_i>\nHH_i$. For $b_i>\nHH_i$ we have 
\begin{equation*}
    \int_\mathbb{R}\diff s\, E^{-}(\a,s)\mathrm{p}_{\nXX_i-\a}(s-b_i) = E^{-}(\nXX_i,b_i),
\end{equation*}
by the Markov property. For $b_i\leq \nHH_i$ we have 
\begin{equation*}
    E^{+}(\nXX_i,b_i)= \mathrm{e}^{-\nXX_i\xi^2-b_i\xi}\mathbf{1}_{b_i\leq \nHH_i}.
\end{equation*}
Thus 
\begin{equation}\label{eq: chi_a3}
\begin{aligned}
    \limess_{\h}^{\a,3}(\eta,\xi) &= \int_{\nHH_i}^{\infty}\diff s\, E^{+}(\nXX_i,b_i)E^{-}(\nXX_i,b_i) + \mathrm{e}^{-\nXX_i\xi^2}\int_{-\infty}^{\nHH_i}\diff b_i\int_{\mathbb{R}}\diff s \,\mathrm{e}^{-b_i\xi}\mathrm{p}_{\a-\nXX_i}(s-b_i)\, E^-(\a,s)\\
    &\quad+\mathbf{1}_{\a=\nXX_{i+1}} \frac{\mathrm{e}^{\nXX_{i+1}(\eta^2-\xi^2)+\nHH_{i+1}(\eta-\xi)}}{\eta-\xi}.
\end{aligned}
\end{equation}
By combining \eqref{eq: chi_a2}, \eqref{eq: chi_a1} and \eqref{eq: chi_a3}, we conclude that 
\begin{equation*}
    \limess_{\h}^\a=\limess_{\h}^{\a,1}+\limess_{\h}^{\a,2}-\limess_{\h}^{\a,3} =  \limess_{\h}^{\nXX_i,1}+\limess_{\h}^{\nXX_i,2}-\limess_{\h}^{\nXX_i,3} = \limess_{\h}^{\nXX_i},
\end{equation*}
for all $\a\in [\nXX_{i+1},\nXX_{i})$. Thus, by induction $\limess_{\h}^\a=\limess_{\h}$ for all $\a\in \mathbb{R}$ and $\h\in \mnw$.

For general $\h\in \ucc$ supported on $[-L,L]$, bounded above by $\beta$, we use Proposition \ref{prop: approximation} to find a sequence $\{\h^n\}_{\geq 1}\subset \mnw$  bounded above by $\beta$ whose supports are contained in $[-L,L]$ , such that $\h^n\to \h$ in $\uc$. Then by Proposition \ref{prop: limess_convergence} we know $\limess_{\h^n}\to \limess_{\h}$ as $n\to \infty$. A minor variant of Proposition \ref{prop: limess_convergence} with the same proof shows that $\limess^\a_{\h^n}\to \limess^\a_{\h}$ as $n\to\infty$ for any $\a\in \mathbb{R}$ as well. Thus, 
\begin{equation*}
    \limess_{\h}^\a = \lim_{n\to \infty} \limess_{\h^n}^\a = \lim_{n\to \infty} \limess_{\h^n} =  \limess_{\h},
\end{equation*}
for all $\a\in \mathbb{R}$ and $\h\in \ucc$. This completes the proof of Proposition \ref{prop: shift}.

\subsection{Proof of Theorem \ref{thm: kpz_multitime}}\label{sec: proof_thm}
Recall that in Proposition \ref{prop: convergence_nw}, we have shown Theorem \ref{thm: kpz_multitime} for $\h\in \mnw_0$. We first extend it to $\h\in \mnw$ using Proposition \ref{prop: shift}. There exists $\alpha,\beta\in\mathbb{R}$ such that $\h^{\alpha,\beta}:= \h(\cdot +\alpha)+\beta\in \mnw_0$. Thus, by the invariance property of the KPZ fixed point (see, e.g., \cite{matetski2021kpz}[Theorem 4.5]) we have
\begin{equation}\label{eq: shift_space}
    \mathbb{P}\left(\bigcap_{\ell=1}^m\{\mathcal{H}(\XX_\ell,\TT_\ell; \h)\leq \HH_\ell\} \right)= \mathbb{P}\left(\bigcap_{\ell=1}^m\{\mathcal{H}(\XX_\ell-\alpha,\TT_\ell; \h^{\alpha,\beta})\leq \HH_\ell+\beta\} \right).
\end{equation}
We apply Proposition \ref{prop: convergence_nw} to the right-hand side of \eqref{eq: shift_space} and compare the resulting formula with the right-hand side of \eqref{eq: multi_time}. In order to prove that Theorem \ref{thm: kpz_multitime} holds for $\h\in \mnw$, it suffices to show that 
\begin{equation}\label{eq: shift_characteristic}
    \limess_{\h}(\eta,\xi) = \mathrm{e}^{\alpha(\eta^2-\xi^2)+\beta(\xi-\eta)}\limess_{\h^{\alpha,\beta}}(\eta,\xi).
\end{equation}
Now we prove \eqref{eq: shift_characteristic}. By the definition of the characteristic function in \eqref{eq: limess}, we have 
\begin{align*}
    \limess_{\h^{\alpha,\beta}}(\eta,\xi)&=\int_{\mathbb{R}} \diff s\, \mathrm{e}^{-s\xi}\cdot  \mathbb{E}_{\B(0)=s}\left[\exp\left(\ta^{\alpha,\beta}_- \eta^2+\B(\ta^{\alpha,\beta}_-)\eta\right)\1_{\ta^{\alpha,\beta}_->-\infty}\right]\\
        &\quad+\int_{\mathbb{R}} \diff s\, \mathrm{e}^{+s\eta}\cdot  \mathbb{E}_{\B(0)=s}\left[\exp\left(-\ta^{\alpha,\beta}_+\xi^2-\B(\ta^{\alpha,\beta}_+)\xi\right)\1_{\ta^{\alpha,\beta}_+<\infty}\right]\\
        &\quad-\int_{\mathbb{R}} \diff s\,  \mathbb{E}_{\B(0)=s}\left[\exp\left(\ta^{\alpha,\beta}_- \eta^2+\B(\ta^{\alpha,\beta}_-)\eta\right)\1_{\ta^{\alpha,\beta}_->-\infty}\right]\\
        &\quad\qquad\quad  \cdot \mathbb{E}_{\B(0)=s}\left[\exp\left(-\ta^{\alpha,\beta}_+ \xi^2-\B(\ta^{\alpha,\beta}_+)\xi\right)\1_{\ta^{\alpha,\beta}_+<\infty}\right],
\end{align*}
where 
\begin{equation*}
     \ta_+^{\alpha,\beta}:=\inf\{\mathsf{x}\geq 0: \B(\mathsf{x})\leq \h(\mathsf{x}+\alpha)+\beta\},\quad \ta^{\alpha,\beta}_-:=\sup\{\mathsf{x}\leq 0: \B(\mathsf{x})\leq \h(\mathsf{x}+\alpha)+\beta\}.
\end{equation*}
Note that $\B(\mathsf{x})\leq \h(\mathsf{x}+\alpha)+\beta$ if and only if $\widehat{\B}(\mathsf{x}+\alpha)\leq \h(\mathsf{x}+\alpha)$ where $\widehat{\B}(\mathsf{x}):= \B(\mathsf{x}-\alpha)-\beta$. The invariance of the Brownian motion implies that 
\begin{equation*}
    \mathrm{Law}_{\B(0)=s}(\ta_+^{\alpha,\beta}, \B(\ta_+^{\alpha,\beta})) = \mathrm{Law}_{\widehat{\B}(\alpha)=s-\beta}(\widehat{\ta}_+-\alpha, \widehat{\B}(\widehat{\ta}_+)+\beta),
\end{equation*}
where 
\begin{equation*}
    \widehat{\ta}_-:= \sup\{\mathsf{x}\leq \alpha: \widehat{\B}(\mathsf{x})\leq \h(\mathsf{x})\}.
\end{equation*}
Thus 
\begin{align*}
    &\int_{\mathbb{R}} \diff s\, \mathrm{e}^{-s\xi}\cdot  \mathbb{E}_{\B(0)=s}\left[\exp\left(\ta^{\alpha,\beta}_- \eta^2+\B(\ta^{\alpha,\beta}_-)\eta\right)\1_{\ta^{\alpha,\beta}_->-\infty}\right]\\
    &= \int_{\mathbb{R}} \diff s\, \mathrm{e}^{-s\xi}\cdot  \mathbb{E}_{\widehat{\B}(\alpha)=s-\beta}\left[\exp\left((\widehat{\ta}_--\alpha) \eta^2+(\widehat{\B}(\widehat{\ta}_-)+\beta)\eta\right)\1_{\widehat{\ta}_->-\infty}\right]\\
    &= \mathrm{e}^{\beta(\eta-\xi)-\alpha\eta^2}\int_\mathbb{R}\diff s\, \mathrm{e}^{-s\xi}\cdot  \mathbb{E}_{\widehat{\B}(\alpha)=s}\left[\exp\left(\widehat{\ta}_- \eta^2+\widehat{\B}(\widehat{\ta}_-)\eta\right)\1_{\widehat{\ta}->-\infty}\right].
\end{align*}
By applying similar arguments to the second and third term, we see that 
\begin{align*}
   \mathrm{e}^{\alpha(\eta^2-\xi^2)+\beta(\xi-\eta)}\limess_{\h^{\alpha,\beta}}(\eta,\xi) &=  \mathrm{e}^{-\alpha\xi^2}\int_{\mathbb{R}}\diff s\,\mathbb{E}_{\widehat{\B}(\alpha)=s}\left[\exp\left(\widehat{\ta}_- \eta^2+\widehat{\B}(\widehat{\ta}_-)\eta\right)\1_{\widehat{\ta}_->-\infty}\right]\\
   &\quad+ \mathrm{e}^{\alpha\eta^2}\int_{\mathbb{R}}\diff s\,\mathbb{E}_{\widehat{\B}(\alpha)=s}\left[\exp\left(-\widehat{\ta}_+ \xi^2-\widehat{\B}(\widehat{\ta}_+)\xi\right)\1_{\widehat{\ta}_+<\infty}\right]\\
   &\quad-\int_{\mathbb{R}} \diff s\,  \mathbb{E}_{\widehat{\B}(\alpha)=s}\left[\exp\left(\widehat{\ta}_- \eta^2+\widehat{\B}(\widehat{\ta
   }_-)\eta-\widehat{\ta}_+ \xi^2-\widehat{\B}(\widehat{\ta}_+)\xi\right)\1_{|\widehat{\ta}_\pm|<\infty}\right],
\end{align*}
which is equal to $\limess_{\h}^\alpha$, and hence $\limess_{\h}$, by Proposition \ref{prop: shift}. Finally, for a general $\h\in \ucc$, we can choose a sequence $\{\h^n\}_{n\geq 1}\subset \mnw$ such that $\h^n\to \h$ and they satisfy the conditions in Proposition \ref{prop: approximation}. By the continuity of the law of the KPZ fixed point with respect to the initial condition, we know 
\begin{equation*}
    \lim_{n\to \infty}  \mathbb{P}\left(\bigcap_{\ell=1}^m\{\mathcal{H}(\XX_\ell,\TT_\ell; \h^n)\leq \HH_\ell\} \right) = \mathbb{P}\left(\bigcap_{\ell=1}^m\{\mathcal{H}(\XX_\ell,\TT_\ell; \h)\leq \HH_\ell\} \right).
\end{equation*}
Hence, by Proposition \ref{prop: approximation} we have 
\begin{align*}
    \mathbb{P}\left(\bigcap_{\ell=1}^m\{\mathcal{H}(\XX_\ell,\TT_\ell; \h)\leq \HH_\ell\} \right) &= \lim_{n\to \infty}  \mathbb{P}\left(\bigcap_{\ell=1}^m\{\mathcal{H}(\XX_\ell,\TT_\ell; \h^n)\leq \HH_\ell\} \right)\\
    &= \lim_{n\to\infty}\prod_{\ell=1}^{m-1} \oint_0 \frac{\diff z_\ell}{2\pi\ii z_\ell(1-z_\ell)}\mathrm{D}_{\h^n}(z_1,\ldots,z_{m-1})\\
    &= \prod_{\ell=1}^{m-1} \oint_0 \frac{\diff z_\ell}{2\pi\ii z_\ell(1-z_\ell)}\mathrm{D}_{\h}(z_1,\ldots,z_{m-1}).
\end{align*}
This completes the proof of Theorem \ref{thm: kpz_multitime}.

\section{Reduction in the equal-time case}\label{sec: equal time}
The goal of this section is to simplify the multipoint formula \eqref{eq: multi_time} under the additional assumption that all the time parameters $\TT_1,\ldots, \TT_m$ are the same, say, equal to $1$.  We will show in Section \ref{sec: equal_new} that one can get rid of the additional contour integrals with respect to the parameters $z_i$'s in \eqref{eq: multi_time}, and get a Fredholm determinant formula for the equal-time multipoint distribution. Then, in Section \ref{sec: equal_equiv} we will show that the Fredholm determinant formula matches the path integral formula in \cite{matetski2021kpz}. 
\subsection{A new Fredholm determinant formula for the equal-time multipoint distribution of the KPZ fixed point}\label{sec: equal_new}
The following Fredholm determinant formula holds for t
he KPZ fixed point with a general initial condition $\h$ that is compactly supported. We remark that an analogous Fredholm determinant for the narrow wedge initial condition was obtained in \cite[Proposition 2.1]{LiuOrtiz25}, with the kernel conjugated to $\mathbf{T}_{\h}$  below and defined on a slightly different space.

\begin{prop}\label{prop: contour_rep}
    Let $\alpha_1<\cdots<\alpha_m$. Consider the KPZ fixed point starting from initial condition $\h\in \ucc$. Then the following formula for the multipoint distribution at  space-time points $(\alpha_1,1),\ldots,(\alpha_m,1)$ holds:
    \begin{equation}\label{eq: equal_time1}
        \mathbb{P}\left(\bigcap_{\ell=1}^m \mathcal{H}(\alpha_\ell,1;\h)\leq \beta_\ell\right) =\det (\mathrm{I}+\mathbf{T}_\h)_{L^2(\{1,\ldots,m\}\times \Gamma_{1,\RR})},
    \end{equation}
    where the operator $\mathbf{T}_\h: L^2(\{1,\ldots,m\}\times \Gamma_{1,\RR})\to L^2(\{1,\ldots,m\}\times \Gamma_{1,\RR})$ has the following kernel:
    \begin{equation}\label{eq: equal_kernel2}
        \mathbf{T}_\h(i,\zeta;j,\eta):= 
            \left(\prod_{\ell=1}^{i}\int_{\Gamma_{\ell,\LL}^{\mathrm{in}}} \frac{\diff \xi_\ell}{2\pi \mathrm{i}}\right)\frac{\prod_{\ell=1}^i\mathrm{F}_\ell(\xi_\ell)\cdot \limess_\h(\eta,\xi_1)}{\prod_{\ell=1}^{i-1}(\xi_\ell-\xi_{\ell+1})\cdot (\xi_i-\zeta)}\cdot \frac{1}{\mathrm{f}_j(\eta)}.
    \end{equation}
    Here $\Gamma_{1,\LL}^{\mathrm{in}}:=\Gamma_{1,\LL}$ and
    \begin{equation}
        \mathrm{F}_i(\zeta)=\frac{\mathrm{f}_i(\zeta)}{\mathrm{f}_{i-1}(\zeta)}:=\begin{cases}
            \mathrm{e}^{-\frac{1}{3}\zeta^3+ \alpha_1 \zeta^2+\beta_1 \zeta},\quad &i=1,\\
            \mathrm{e}^{(\alpha_i-\alpha_{i-1})\zeta^2+(\beta_i-\beta_{i-1})\zeta},\quad &2\leq i\leq m. 
        \end{cases}
    \end{equation}
\end{prop}
The proof of Proposition \ref{prop: contour_rep} is based on the following lemma, whose proof is almost identical to \cite[Lemma 2.4]{LiuOrtiz25} and is omitted here. Note that the only difference (modulo obvious change of notation) between \eqref{eq: equal_series} and \cite[(2.31)]{LiuOrtiz25} is that the Cauchy determinant $\mathrm{C}(\etab,\xib)$ is replaced by $\det\left(\limess_\h(\eta_{\ell_i}^{(i)},\xi_{\ell_j}^{(j)})\right)$.

\begin{lm}
    Under the same assumption as in Proposition \ref{prop: contour_rep}, we have 
    \begin{equation}
         \mathbb{P}\left(\bigcap_{\ell=1}^m \mathcal{H}(\alpha_\ell,1;\h)\leq \beta_\ell\right) = \sum_{n_1\geq \cdots\geq n_m\geq 0}\frac{1}{(n_1!\cdots n_m!)^2}\widehat{\mathrm{D}}^\mathbf{(n)}_\h,
    \end{equation}
    where $\mathbf{n}=(n_1,\ldots,n_m)\in \mathbb{Z}_+^m$ and
    \begin{equation}\label{eq: equal_series}
    \begin{aligned}
        \widehat{\mathrm{D}}^\mathbf{(n)}_\h &= \left(\prod_{i=1}^{m}\frac{n_i!}{k_i!}\right)^2 \left(\prod_{i=1}^m \prod_{\ell_i=1}^{k_i}\int_{\Gamma_{1,\LL}}\frac{\diff \xi_{\ell_i}^{(i)}}{2\pi\mathrm{i}}\int_{\Gamma_{1,\RR}}\frac{\diff \eta_{\ell_i}^{(i)}}{2\pi\mathrm{i}}\right)\det\left[\limess_\h(\eta_{\ell_i}^{(i)},\xi_{\ell_j}^{(j)})\right]_{\substack{(i,\ell_i),(j,\ell_j)\\1\leq i,j\leq m, 1\leq \ell_i\leq k_i}}\\
        &\quad\cdot \prod_{i=1}^m \det \left[h_i(\xi_a^{(i)},\eta_b^{(i)})\right]_{a,b=1}^{k_i}\cdot \prod_{i=1}^m\prod_{\ell_i=1}^{k_i} \frac{1}{\mathrm{f}_i(\eta_{\ell_i}^{(i)})}.
        \end{aligned}
    \end{equation}
    Here $k_i:=n_i-n_{i+1}\geq 0$ with the convention that $n_{m+1}:=0$. The functions $h_i(\xi,\eta)$ are defined for all $(\xi,\eta)\in \Gamma_{1,\LL}\times \Gamma_{1,\RR}$ as follows:
    \begin{equation}
        h_i(\xi,\eta):=\begin{cases}
            \displaystyle \frac{\mathrm{F}_1(\xi)}{\xi-\eta},\quad &i=1,\\
            \displaystyle \prod_{\ell=2}^i \int_{\Gamma_{\ell,\LL}^{\mathrm{in}}}\frac{\diff \xi_\ell}{2\pi\mathrm{i}} \frac{\mathrm{F}_1(\xi)\cdot \prod_{\ell=2}^i \mathrm{F}_\ell(\xi_\ell)}{(\xi-\xi_2)\cdot \prod_{\ell=2}^{i-1}(\xi_\ell-\xi_{\ell+1})\cdot (\xi_i-\eta)},\quad & 2\leq i\leq m.
        \end{cases}
    \end{equation}
\end{lm}

\begin{proof}[Proof of Proposition \ref{prop: contour_rep}]
   We apply a generalization of the Andreief's identity obtained in \cite[Lemma 1.2]{LiuOrtiz25} to the $\xi$-integrals. It gives 
    \begin{equation}
    \begin{aligned}
         \widehat{\mathrm{D}}^\mathbf{(n)}_\h &= \prod_{i=1}^m \frac{(n_i!)^2}{k_i!}\left(\prod_{i=1}^m\prod_{\ell_i=1}^{k_i}\int_{\Gamma_{1,\RR}}\frac{\diff \eta_{\ell_i}^{(i)}}{2\pi\mathrm{i}}\right)\det\left[\int_{\Gamma_{1,\LL}}\frac{\diff \xi_1}{2\pi\mathrm{i}}h_i(\xi_1,\eta_{\ell_i}^{(i)})\cdot \frac{\limess_{\h}(\eta_{\ell_j}^{(j)},\xi_1)}{\mathrm{f}_j(\eta_{\ell_j}^{(j)})}\right]_{(i,\ell_i),(j,\ell_j)}\\
         &= \prod_{i=1}^m \frac{(n_i!)^2}{k_i!}\left(\prod_{i=1}^m\prod_{\ell_i=1}^{k_i}\int_{\Gamma_{1,\RR}}\frac{\diff \eta_{\ell_i}^{(i)}}{2\pi\mathrm{i}}\right)\det\left[\mathbf{T}_\h(i,\eta_{\ell_i}^{(i)};j,\eta_{\ell_j}^{(j)})\right]_{(i,\ell_i),(j,\ell_j)},
    \end{aligned}
    \end{equation}
    where $\mathbf{T}_\h(i,\zeta;j,\eta)$ is defined as in \eqref{eq: equal_kernel2}. Thus, we have 
    \begin{align*}
        \mathbb{P}\left(\bigcap_{\ell=1}^m \mathcal{H}(\alpha_\ell,1;\h)\leq \beta_\ell\right) &= \sum_{k_1,\ldots,k_m\geq 0} \frac{1}{k_1!\cdots k_m!}\left(\prod_{i=1}^m\prod_{\ell_i=1}^{k_i}\int_{\Gamma_{1,\RR}}\frac{\diff \eta_{\ell_i}^{(i)}}{2\pi\mathrm{i}}\right)\det\left[\mathbf{T}_\h(i,\eta_{\ell_i}^{(i)};j,\eta_{\ell_j}^{(j)})\right]_{(i,\ell_i),(j,\ell_j)}\\
        &= \det(\mathrm{I}+\mathbf{T}_\h)_{L^2(\{1,\ldots,m\}\times \Gamma_{1,\RR})}.
    \end{align*}
\end{proof}

\subsection{Equivalence with the path integral formula of \cite{matetski2021kpz}}\label{sec: equal_equiv}
The goal of this section is to prove that the equal-time multipoint formula \eqref{eq: equal_time1} is equivalent to the following path integral formula  obtained in \cite{matetski2021kpz} when the initial condition is compactly supported.
\begin{prop}[Proposition 4.3 of \cite{matetski2021kpz}]\label{prop: pathintegral}
    \begin{equation}
        \mathbb{P}\left(\bigcap_{\ell=1}^m \mathcal{H}(\alpha_\ell,1;\h)\leq \beta_\ell\right) =\det (\mathrm{I}-\mathbf{K}_{1,\alpha_1}^{\mathrm{hypo}(\h)}+\mathbf{1}_{\leq \beta_1}\mathrm{e}^{(\alpha_1-\alpha_2)\partial^2}\mathbf{1}_{\leq \beta_2}\cdots \mathbf{1}_{\leq \beta_m}\mathrm{e}^{(\alpha_m-\alpha_1)\partial^2}\mathbf{K}_{1,\alpha_1}^{\mathrm{hypo}(\h)})_{L^2(\mathbb{R})}.
    \end{equation}
\end{prop}
We will first express the path integral kernel in terms of contour integrals, the result is summarized in the following proposition.
\begin{prop}
    Assume $\h\in \ucc$. Let 
    $$\mathbf{S}_\h:= -\mathbf{K}_{1,\alpha_1}^{\mathrm{hypo}(\h)}+\mathbf{1}_{\leq \beta_1}\mathrm{e}^{(\alpha_1-\alpha_2)\partial^2}\mathbf{1}_{\leq \beta_2}\cdots \mathbf{1}_{\leq \beta_m}\mathrm{e}^{(\alpha_m-\alpha_1)\partial^2}\mathbf{K}_{1,\alpha_1}^{\mathrm{hypo}(\h)}.$$
    Then we have the following contour integral representation for the kernel of $\mathbf{S}_\h$:
    \begin{equation}\label{eq: path integral_contour}
    \begin{aligned}
        \mathbf{S}_\h(\lambda,\mu) =& -\mathbf{1}_{\lambda>\beta_1}\int_{\Gamma_{1,\LL}}\frac{\diff \xi}{2\pi\mathrm{i}}\int_{\Gamma_{1,\RR}}\frac{\diff \eta}{2\pi\mathrm{i}}\frac{\mathrm{f}_1(\xi)}{\mathrm{f}_1(\eta)}\cdot \limess_\h(\eta,\xi)\cdot \mathrm{e}^{(\mu-\beta_1)\xi-(\lambda-\beta_1)\eta}\\
        &  +\sum_{i=2}^m \mathbf{1}_{\lambda\leq \beta_1}\left(\prod_{\ell=1}^i\int_{\Gamma_{\ell,\LL}^{\mathrm{in}}}\frac{\diff \xi_i}{2\pi \mathrm{i}}\right)\int_{\Gamma_{1,\RR}}\frac{\diff \eta}{2\pi\mathrm{i}}\frac{\prod_{\ell=1}^i \mathrm{F}_\ell(\xi_\ell)\cdot \limess_\h(\eta,\xi_1)}{\prod_{\ell=2}^{i-1}(\xi_\ell-\xi_{\ell+1})\cdot (\xi_i-\eta)}\frac{\mathrm{e}^{(\mu-\beta_1)\xi_1-(\lambda-\beta_1)\xi_2}}{\mathrm{f}_i(\eta)}.
    \end{aligned}
    \end{equation}
\end{prop}
\begin{proof}
    First note that by writing $\mathbf{1}_{\leq \beta_m}= 1-\mathbf{1}_{>\beta_m}$, we have 
    \begin{equation*}
        \begin{aligned}
        &\mathbf{1}_{\leq \beta_1}\mathrm{e}^{(\alpha_1-\alpha_2)\partial^2}\cdots \mathbf{1}_{\leq \beta_m}\mathrm{e}^{(\alpha_m-\alpha_1)\partial^2}\mathbf{K}_{1,\alpha_1}^{\mathrm{hypo}(\h)}\\
        &= \mathbf{1}_{\leq \beta_1}\mathrm{e}^{(\alpha_1-\alpha_2)\partial^2}\cdots \mathbf{1}_{\leq \beta_{m-1}}\mathrm{e}^{(\alpha_{m-1}
        -\alpha_1)\partial^2}\mathbf{K}_{1,\alpha_1}^{\mathrm{hypo}(\h)} \\
        &\quad- \mathbf{1}_{\leq \beta_1}\mathrm{e}^{(\alpha_1-\alpha_2)\partial^2}\cdots \mathbf{1}_{\leq \beta_{m-1}}\mathrm{e}^{(\alpha_{m-1}-\alpha_{m})}\mathbf{1}_{>\beta_m}\mathrm{e}^{(\alpha_m-\alpha_1)\partial^2}\mathbf{K}_{1,\alpha_1}^{\mathrm{hypo}(\h)},
        \end{aligned}
    \end{equation*}
    where we are using the semigroup property 
    \begin{equation*}
        \mathrm{e}^{(\alpha_{m-1}-\alpha_m)\partial^2}\mathrm{e}^{(\alpha_{m}-\alpha_1)\partial^2}\mathbf{K}_{1,\alpha_1}^{\mathrm{hypo}(\h)} =  \mathrm{e}^{(\alpha_{m-1}-\alpha_1)\partial^2}\mathbf{K}_{1,\alpha_1}^{\mathrm{hypo}(\h)}.
    \end{equation*}
    Repeating this argument for $\mathbf{1}_{\leq \beta_1}\mathrm{e}^{(\alpha_1-\alpha_2)\partial^2}\cdots \mathbf{1}_{\leq \beta_{i}}\mathrm{e}^{(\alpha_{i}
        -\alpha_1)\partial^2}\mathbf{K}_{1,\alpha_1}^{\mathrm{hypo}(\h)}$ with $i=m-1,\ldots,2$,  we see that 
    \begin{equation}\label{eq: path_decomp}
        \begin{aligned}
            &-\mathbf{K}_{1,\alpha_1}^{\mathrm{hypo}(\h)}+\mathbf{1}_{\leq \beta_1}\mathrm{e}^{(\alpha_1-\alpha_2)\partial^2}\mathbf{1}_{\leq \beta_2}\cdots \mathbf{1}_{\leq \beta_m}\mathrm{e}^{(\alpha_m-\alpha_1)\partial^2}\mathbf{K}_{1,\alpha_1}^{\mathrm{hypo}(\h)}\\
            &= -\mathbf{1}_{>\beta_1} \mathbf{K}_{1,\alpha_1}^{\mathrm{hypo}(\h)}-\sum_{i=2}^m \mathbf{1}_{\leq \beta_1}\mathrm{e}^{(\alpha_1-\alpha_2)\partial^2}\cdots \mathbf{1}_{\leq \beta_{i-1}}\mathrm{e}^{(\alpha_{i-1}-\alpha_i)\partial^2}\mathbf{1}_{>\beta_i}\mathrm{e}^{(\alpha_i-\alpha_1)\partial^2}\mathbf{K}_{1,\alpha_1}^{\mathrm{hypo}(\h)}.
        \end{aligned}
    \end{equation}
    Now recall the definition of $\mathbf{K}_{1,\alpha_1}^{\mathrm{hypo}(\h)}$ from \cite[(4.5)]{matetski2021kpz}:
    \begin{equation}\label{eq: Brownianscattering}
        \mathbf{K}_{1,\alpha_1}^{\mathrm{hypo}(\h)} =   \left(S_{1,-\alpha_1}^{\mathrm{hypo}(\h^{-})}\right)^* S_{1,\alpha_1}+S_{1,-\alpha_1}^* S_{1,\alpha_1}^{\mathrm{hypo}(\h^{+})} -\left(S_{1,-\alpha_1}^{\mathrm{hypo}(\h^{-})}\right)^*S_{1,\alpha_1}^{\mathrm{hypo}(\h^{+})},
    \end{equation}
    where 
    \begin{equation}
        S_{t,x}(p,q) = S_{t,x}^{*}(q,p) := \int_{\Gamma_{1,\RR}} \frac{\diff \eta}{2\pi \mathrm{i}}\mathrm{e}^{\frac{t}{3}\eta^3+x \eta^2+(p-q)\eta }=\int_{\Gamma_{1,\LL}} \frac{\diff \xi}{2\pi \mathrm{i}}\mathrm{e}^{-\frac{t}{3}\xi^3+x \xi^2-(p-q)\xi },
    \end{equation}
    and 
    \begin{equation}
        S_{t,x}^{\mathrm{hypo}(\h^+)}(p,q):= \mathbb{E}_{\B(0)=p}\left[S_{t,x-\ta_+}(\B(\ta_+),q)\mathbf{1}_{\ta_+<\infty}\right].
    \end{equation}
    We have 
    \begin{equation*}
        \begin{aligned}
            &S_{1,-\alpha_1}^* S_{1,\alpha_1}^{\mathrm{hypo}(\h^{+})}(p,q) \\
            &= \int_\mathbb{R}\diff s\, \int_{\Gamma_{1,\RR}} \frac{\diff \eta}{2\pi\mathrm{i}} \mathrm{e}^{\frac{1}{3}\eta^3-\alpha_1\eta^2+(s-p)\eta}\cdot \mathbb{E}_{\B(0)=s}\left[\int_{\Gamma_{1,\LL}} \frac{\diff \xi}{2\pi\mathrm{i}}\mathrm{e}^{-\frac{1}{3}\xi^3+(\alpha_1-\ta_+)\xi^2+(q-\B(\ta_+))\xi}\right] \\
            &=\int_{\Gamma_{1,\LL}} \frac{\diff \xi}{2\pi\mathrm{i}}\int_{\Gamma_{1,\RR}}\frac{\diff \eta}{2\pi\mathrm{i}} \frac{\mathrm{e}^{-\frac{1}{3}\xi^3+\alpha_1\xi^2+q\xi}}{\mathrm{e}^{-\frac{1}{3}\eta^3+\alpha_1\eta^2+p\eta}}\cdot \int_\mathbb{R}\diff s\,\mathrm{e}^{s\eta}\cdot \mathbb{E}_{\B(0)=s}\left[\exp\left(-\ta_+\xi^2-\B(\ta_+)\xi\right)\right],
        \end{aligned}
    \end{equation*}
    where the change of order of integration is justified by Proposition \ref{prop:convergence_lim}. A similar computation for the other two terms in \eqref{eq: Brownianscattering} implies that 
    \begin{equation}
        \mathbf{K}_{1,\alpha_1}^{\mathrm{hypo}(\h)} (p,q) = \int_{\Gamma_{1,\LL}} \frac{\diff \xi}{2\pi\mathrm{i}}\int_{\Gamma_{1,\RR}} \frac{\diff \eta}{2\pi\mathrm{i}} \frac{\mathrm{e}^{-\frac{1}{3}\xi^3+\alpha_1\xi^2+q\xi}}{\mathrm{e}^{-\frac{1}{3}\eta^3+\alpha_1\eta^2+p\eta}}\cdot \limess_\h(\eta,\xi),
    \end{equation}
    where $\limess_\h(\eta,\xi)$ is defined in \eqref{eq: limess}. Thus 
    \begin{equation}
        \mathrm{e}^{(\alpha_i-\alpha_1)\partial^2}\mathbf{K}_{1,\alpha_1}^{\mathrm{hypo}(\h)} (p,q) = \int_{\Gamma_{1,\LL}} \frac{\diff \xi}{2\pi\mathrm{i}}\int_{\Gamma_{1,\RR}} \frac{\diff \eta}{2\pi\mathrm{i}} \frac{\mathrm{e}^{-\frac{1}{3}\xi^3+\alpha_1\xi^2+q\xi}}{\mathrm{e}^{-\frac{1}{3}\eta^3+\alpha_i\eta^2+p\eta}}\cdot \limess_\h(\eta,\xi),
    \end{equation}
    for $2\leq i\leq m$. On the other hand, for $2\leq i\leq m$,  the heat kernel $\mathrm{e}^{(\alpha_{i-1}-\alpha_{i})\partial^2}$ can be expressed as 
    \begin{equation}
        \displaystyle\mathrm{e}^{(\alpha_{i-1}-\alpha_{i})\partial^2} (p,q) = \frac{1}{\sqrt{4\pi(\alpha_{i}-\alpha_{i-1})}} \mathrm{e}^{-\frac{(p-q)^2}{4(\alpha_{i}-\alpha_{i-1})}} = \int_{c+\mathrm{i}\mathbb{R}} \frac{\diff \xi_i}{2\pi \mathrm{i}} \mathrm{e}^{(q-p)\xi_i}\cdot \mathrm{e}^{(\alpha_{i}-\alpha_{i-1})\xi_i^2}, 
    \end{equation}
    for any $c\in \mathbb{R}$. Thus the convolution $\mathrm{e}^{(\alpha_1-\alpha_2)\partial^2}\mathbf{1}_{\leq \beta_2}\mathrm{e}^{(\alpha_{2}-\alpha_3)\partial^2}$ has the following kernel:
    \begin{equation*}
        \begin{aligned}
        \mathrm{e}^{(\alpha_1-\alpha_2)\partial^2}\mathbf{1}_{\leq \beta_2}\mathrm{e}^{(\alpha_{2}-\alpha_3)\partial^2}(p,q)& = \int_{c_2+\mathrm{i}\mathbb{R}}\frac{\diff \xi_2}{2\pi\mathrm{i}} \int_{-\infty}^{\beta_2}\diff r_2 \int_{c_3+\mathrm{i}\mathbb{R}}\frac{\diff \xi_3}{2\pi \mathrm{i}} \mathrm{e}^{-p\xi_2+r_2\xi_2-r_2\xi_3+q\xi_3}\cdot \mathrm{e}^{(\alpha_2-\alpha_1)\xi_2^2+(\alpha_3-\alpha_2)\xi_3^2}\\
        &= \int_{c_2+\mathrm{i}\mathbb{R}}\frac{\diff \xi_2}{2\pi\mathrm{i}}\int_{c_3+\mathrm{i}\mathbb{R}}\frac{\diff \xi_3}{2\pi\mathrm{i}} \frac{\mathrm{e}^{(\beta_1-p)\xi_2 + (q-\beta_3)\xi_3}}{\xi_2-\xi_3}\cdot \mathrm{F}_2(\xi_2)\cdot \mathrm{F}_3(\xi_3),
        \end{aligned}
    \end{equation*}
    where $\mathrm{F}_i(\zeta)=: \mathrm{e}^{(\alpha_i-\alpha_{i-1})\zeta^2+(\beta_i-\beta_{i-1})\zeta}$ and $c_2>c_3$. Similarly, for any $2\leq i\leq m$ we have 
    \begin{equation}
        \mathrm{e}^{(\alpha_1-\alpha_2)\partial^2}\mathbf{1}_{\leq \beta_2}\cdots \mathbf{1}_{\leq \beta_{i-1}}\mathrm{e}^{(\alpha_{i-1}-\alpha_{i})\partial^2}(p,q) = \left(\prod_{\ell=2}^i\int_{c_\ell+\mathrm{i}\mathbb{R}} \frac{\diff \xi_\ell}{2\pi \mathrm{i}}\right)\frac{\mathrm{e}^{(\beta_1-p)\xi_2 + (q-\beta_i)\xi_i}}{\prod_{\ell=2}^{i-1}(\xi_\ell-\xi_{\ell+1})}\cdot \prod_{\ell=2}^i \mathrm{F}_\ell(\xi_\ell),
    \end{equation}
    where $c_2>\cdots >c_i$. Thus 
    \begin{equation}\label{eq: path_i}
        \begin{aligned}
            &\mathrm{e}^{(\alpha_1-\alpha_2)\partial^2}\mathbf{1}_{\leq \beta_2}\cdots \mathbf{1}_{\leq \beta_{i-1}}\mathrm{e}^{(\alpha_{i-1}-\alpha_{i})\partial^2}\mathbf{1}_{>\beta_i}\mathrm{e}^{(\alpha_i-\alpha_1)\partial^2}\mathbf{K}_{1,\alpha_1}^{\mathrm{hypo}(\h)}(p,q)\\
            &=\left(\prod_{\ell=2}^i\int_{c_\ell+\mathrm{i}\mathbb{R}} \frac{\diff \xi_\ell}{2\pi \mathrm{i}}\right)\int_{\beta_i}^{\infty}\diff r \int_{\Gamma_{1,\LL}} \frac{\diff \xi_1}{2\pi\mathrm{i}}\int_{\Gamma_{1,\RR}} \frac{\diff \eta}{2\pi\mathrm{i}} \frac{\mathrm{e}^{(\beta_1-p)\xi_2 + (r-\beta_i)\xi_i}}{\prod_{\ell=2}^{i-1}(\xi_\ell-\xi_{\ell+1})}\prod_{\ell=2}^i \mathrm{F}_\ell(\xi_\ell)\frac{\mathrm{e}^{-\frac{1}{3}\xi_1^3+\alpha_1\xi_1^2+q\xi_1}}{\mathrm{e}^{-\frac{1}{3}\eta^3+\alpha_i\eta^2+r\eta}}\limess_\h(\eta,\xi_1)\\
            &= \left(\prod_{\ell=2}^i\int_{c_\ell+\mathrm{i}\mathbb{R}} \frac{\diff \xi_\ell}{2\pi \mathrm{i}}\right)\int_{\Gamma_{1,\LL}} \frac{\diff \xi_1}{2\pi\mathrm{i}}\int_{\Gamma_{1,\RR}} \frac{\diff \eta}{2\pi\mathrm{i}} \frac{\mathrm{e}^{(\beta_1-p)\xi_2 + (q-\beta_1)\xi_1}}{\prod_{\ell=2}^{i-1}(\xi_\ell-\xi_{\ell+1})\cdot (\eta-\xi_i)} \prod_{\ell=2}^i \mathrm{F}_\ell(\xi_\ell) \cdot\frac{\mathrm{f}_1(\xi_1)}{\mathrm{f}_i(\eta)}\cdot \limess_\h(\eta,\xi_1).
        \end{aligned}
    \end{equation}
    By combining \eqref{eq: path_i} with \eqref{eq: path_decomp}, we arrive at the desired expression \eqref{eq: path integral_contour} for $\mathbf{S}_\h(\lambda,\mu)$.
\end{proof}

Finally we rewrite $\mathbf{T}_\h$ properly to match with $\mathbf{S}_\h$. Deform $\Gamma_{1,\LL}$ and $\Gamma_{2,\LL}^{\mathrm{in}}$ into two vertical lines $c_i+\mathrm{i}\mathbb{R}$, $i=1,2$, with $0>c_1>c_2$.  Then we have 
\begin{equation*}
    \frac{1}{\xi_1-\xi_2} = \int_{-\infty}^0 \mathrm{d}\lambda\, \mathrm{e}^{\lambda(\xi_1-\xi_2)},
\end{equation*}
for any $\xi_1\in \Gamma_{1,\LL}$ and $\xi_2\in \Gamma_{2,\LL}^{\mathrm{in}}$. Now we write $\mathbf{T}_\h:= L_1 L_2$, where  $L_1: L^2(\mathbb{R})\to L^2(\{1,\ldots,m\}\times \Gamma_{1,\RR})$ has the following kernel:
\begin{equation}
    L_1(i,\zeta;\lambda) = 
    \begin{cases}
        \displaystyle -\mathrm{e}^{-\lambda\zeta}\mathbf{1}_{\lambda> 0},\quad &i=1,\\
        \displaystyle\prod_{\ell=2}^i\int_{\Gamma_{\ell,\LL}^{\mathrm{in}}} \frac{\mathrm{d}\xi_\ell} {2\pi\mathrm{i}}\frac{\prod_{\ell=2}^i \mathrm{F}_\ell(\xi_\ell)\mathrm{e}^{-\lambda\xi_2}}{\prod_{\ell=2}^{i-1}(\xi_\ell-\xi_{\ell+1})\cdot (\xi_i-\zeta)}\mathbf{1}_{\lambda\leq 0}, &2\leq i\leq m,
    \end{cases}
\end{equation}
and  $L_2: L^2(\{1,\ldots,m\}\times \Gamma_{1,\RR})\to L^2(\mathbb{R})$ has the following kernel:
\begin{equation}
    L_2(\lambda;j,\eta):= \int_{\Gamma_{1,\LL}}\frac{\diff \xi_1}{2\pi \mathrm{i}} \frac{\mathrm{f}_1(\xi_1)}{\mathrm{f}_j(\eta)} \mathrm{e}^{\lambda\xi_1}\cdot \limess_\h(\eta,\xi_1).
\end{equation}
Then we have
\begin{equation*}
    \det (\mathrm{I}+\mathbf{T}_\h)_{L^2(\{1,\ldots,m\}\times \Gamma_{1,\RR})} = \det(\mathrm{I}+L_1L_2)=\det(\mathrm{I}+L_2L_1):= \det(\mathrm{I}+\widehat{\mathbf{S}}_\h)_{L^2(\mathbb{R})},
\end{equation*}
where 
\begin{equation}\label{eq: conjugate_S}
\begin{aligned}
    \widehat{\mathbf{S}}_\h(\lambda,\mu) &= \sum_{i=1}^m \int_{\Gamma_{1,\RR}} \frac{\diff \eta}{2\pi\mathrm{i}} L_2(\lambda; i,\eta) L_1(i,\eta;\mu)\\
    &= -\int_{\Gamma_{1,\RR}} \frac{\diff \eta}{2\pi\mathrm{i}}\int_{\Gamma_{1,\LL}}\frac{\diff \xi}{2\pi \mathrm{i}} \frac{\mathrm{f}_1(\xi)}{\mathrm{f}_1(\eta)} \mathrm{e}^{\lambda\xi-\mu\eta}\cdot \limess_\h(\eta,\xi)\cdot \mathbf{1}_{\mu> 0}\\
    &\quad+\sum_{i=2}^m \int_{\Gamma_{1,\RR}} \frac{\diff \eta}{2\pi\mathrm{i}}\left(\prod_{\ell=1}^i \int_{\Gamma_{\ell,\LL}^{\mathrm{in}}}\frac{\diff \xi_\ell}{2\pi \mathrm{i}}\right) \frac{\prod_{\ell=1}^i \mathrm{F}_\ell(\xi_\ell)}{\prod_{\ell=2}^{i-1}(\xi_\ell-\xi_{\ell+1})\cdot (\xi_i-\eta)}\frac{\mathrm{e}^{\lambda\xi_1-\mu\xi_2}}{\mathrm{f}_i(\eta)}\cdot \limess_\h(\eta,\xi_1)\cdot\mathbf{1}_{\mu\leq 0}.
\end{aligned}
\end{equation}
Comparing \eqref{eq: conjugate_S} with \eqref{eq: path integral_contour}, we see that 
$\widehat{\mathbf{S}}_\h(\lambda,\mu) = \mathbf{S}_\h(\mu+\beta_1,\lambda+\beta_1)$.
Thus
\begin{equation*}
    \det(\mathrm{I}+\mathbf{S}_\h)_{L^2(\mathbb{R})} = \det(\mathrm{I}+\widehat{\mathbf{S}}_\h)_{L^2(\mathbb{R})}=\det(\mathrm{I}+\mathbf{T}_\h)_{L^2(\{1,\ldots,m\}\times \Gamma_{1,\RR})}.
\end{equation*}
This completes the proof of the equivalence between Proposition \ref{prop: contour_rep} and \ref{prop: pathintegral}.

\end{document}